\documentclass[11pt]{amsart}

\usepackage{amssymb}
\usepackage{latexsym}
\usepackage{amsmath}
\usepackage{amsthm}
\usepackage{amsfonts}
\usepackage{color}
\usepackage{hyperref}
\usepackage{pdfsync}
\usepackage[usenames,dvipsnames]{pstricks}
\usepackage{epsfig}
\usepackage{pst-grad} 
\usepackage{pst-plot} 

\newcommand{\R}{\mathbb{R}}
\newcommand{\N}{\mathbb{N}}

\newcommand{\I}{\mathcal{I}}

\newcommand{\VV}{\mathbf{V}}
\newcommand{\EE}{\mathbf{E}}

\newcommand{\e}{\varepsilon}
\newcommand{\ver}{\mathrm{\mathbf{\bar v}}}

\newcommand{\br}{\color{red}}

\theoremstyle{plain}
\newtheorem{defi}{Definition}[section]
\newtheorem{prop}[defi]{Proposition}
\newtheorem{teo}[defi]{Theorem}

\newtheorem{lema}[defi]{Lemma}

\theoremstyle{definition}
\newtheorem{rema}[defi]{Remark}

\theoremstyle{remark}

\numberwithin{equation}{section}

\begin{document}

\title[]{Nonlocal Hamilton-Jacobi Equations on a network with Kirchhoff type conditions}

\author[]{Guy Barles}
\address{
Guy Barles: Institut Denis Poisson (UMR CNRS 7013)
Université de Tours, Université d’Orléans, CNRS.
Parc de Grandmont. 37200 Tours, France.
{\tt guy.barles@idpoisson.fr}
}

\author[]{Olivier Ley}
\address{
Olivier Ley: Univ Rennes, INSA Rennes, CNRS, IRMAR - UMR 6625, F-35000 Rennes, France.
{\tt olivier.ley@insa-rennes.fr}
}

\author[]{Erwin Topp}

\address{
Erwin Topp:
Instituto de Matem\'aticas, Universidade Federal do Rio de Janeiro, Rio de Janeiro - RJ, 21941-909, Brazil; and Departamento de Matem\'atica y C.C., Universidad de Santiago de Chile, Casilla 307, Santiago, Chile.
{\tt etopp@im.ufrj.br; erwin.topp@usach.cl.}
}

\date{\today}

\begin{abstract}
In this article, we consider nonlocal Hamilton-Jacobi Equations on networks with Kirchhoff type conditions for the interior vertices
and Dirichlet boundary conditions for the boundary ones: our aim is to provide general existence and
comparison results in the case when the integro-differential operators are of order strictly less than $1$. The main originality of these 
results is to allow these nonlocal terms to have contributions on several different edges of the network. The existence of Lipschitz 
continuous solutions is proved in two ways: either by using the vanishing viscosity method or by the usual Perron's method. The 
comparison proof relies on arguments introduced by Lions and Souganidis. We also introduce a notion of flux-limited
solution, nonlocal analog to the one introduced by Imbert and Monneau, and prove that the solutions of the Kirchhoff problem are flux-limited solutions for a 
suitable flux-limiter. After treating in details the case when we only have one interior vertex, we extend our approach to treat general networks.
\end{abstract}

\keywords{Nonlocal Equations, Hamilton-Jacobi Equations, Networks, Kirchhoff conditions, Existence, Comparison, Regularity, Junction viscosity solutions, Flux-limited solutions}

\subjclass[2010]{35K55, 35R09, 47G20,  35B40, 33B20}

\maketitle

\section{Introduction.}

\subsection{General description of the problem.}
In this article, we are interested in nonlocal Hamilton-Jacobi Equations (NLHJE) posed on a general network $\Gamma$ with
Kirchhoff type conditions for the interior vertices and Dirichlet (or other) boundary conditions for the boundary ones.
Our aim is to provide a complete approach in the case when the nonlocal term is of order strictly less that $1$ but may involve
integral terms on several edges, and not only the current one. This means existence and uniqueness results but also a connection with the notion of flux-limited solutions of Imbert and Monneau \cite{im17}.
 
A general network $\Gamma$ in $\R^d$ is made of a finite number of vertices $\ver\in\VV$ connected with a finite
number of edges $E\in \EE$. Each edge has the differential structure of a curve and, by using a suitable
parametrization, we can define an Hamilton-Jacobi Equation on it. In general, these equations are
edge-by-edge unrelated, but after adding Kirchhoff-type conditions on some \textsl{interior vertices}
$\ver \in \VV_{i}$ (those which are connected to several edges), and suitable boundary conditions on the
remaining \textsl{boundary vertices} $\ver \in \VV_{b}$, with $\VV = \VV_{i} \cup \VV_{b}, \ \VV_i \cap \VV_b = \emptyset$,
we are lead to a system of equations on the whole network which may be well-posed.

We first describe the type of problem we have in mind at a formal level.
We consider the stationary Kirchhoff-Dirichlet problem
\begin{eqnarray}\label{edp-nl}
&&  \left\{
  \begin{array}{lll}
\lambda u -\I u(x) + H(x,u_x)=0, & x\in \Gamma\setminus \VV & \text{(NLHJE)},\\
\displaystyle \sum_{E\in \text{Inc}(\ver)} -\partial_E u(\ver) = B_\ver, & \ver\in  \VV_{i} & \text{(Kirchhoff condition)},\\
u(\ver)= h_\ver, & \ver\in  \VV_{b} & \text{(Dirichlet condition)}.
\end{array}
  \right.
  \end{eqnarray}
Here $\lambda >0$ is a constant, the nonlocal operator $\I$ and the Hamiltonian $H$
are, in fact, some collections $\{ \I_E \}_{E \in \EE}, \{ H_E \}_{E \in \EE}$
such that, on the one hand,
\begin{eqnarray}\label{nl-intro}
\I_E u(x) =  \int_{\Gamma} [u(z) - u(x)] \nu_E (x,z)dz, \qquad \text{$x\in E$},
\end{eqnarray}
is an  integro-differential operator whose kernel $\nu_E$
satisfies a Lévy-type integrability condition, namely
\begin{eqnarray}\label{levy}
&& \int_{\Gamma} \min\{ \text{dist}(x,z)^{\sigma_E} , 1\} \nu_E(x,z)dz <+\infty
\quad \text{for some $0<\sigma_E <1$},
\end{eqnarray}
where we have identified $dz$ with $d\mathcal H^1(z)$, the $1$-dimensional Hausdorff measure in $\R^d$.

For a general network, we have no intrinsic definition of $u_x$ and therefore no intrinsic definition of $H_E$.
This requires a parametrization of each edge, a definition of $u_x$ and then of $H_E$ for any $E$. We do
not want to enter into details here and we refer the reader to Section~\ref{sec:gene-net} for a complete description
of this general case.

In this introduction, and in most of our article, we are going to consider the case of a simple junction,
i.e., the case
when there is only one interior point and the different edges are segments. More specifically, $ \VV_{i}=\{O\}$, where $O$ is the origin in $\R^d$, 
connected to a finite number of finite length edges  $\{ E_i \}_{1\leq i \leq N}$ that  link it to a finite set of exterior vertices $\VV_b =\{\ver_i\}_{1\leq i \leq N}$.
In this setting, 
we write 
$$E_i=\{t\ver_i, t \in (0,1)\},$$
and, for $x\in E_i$, we set $x_i:=|x|$, which is the parametrization by arc length of the (open) segment $E_i$. Finally we define
$u_i : [0,a_i]\to \R$, where $a_i=|\ver_i|$, by
$$ u_i (x_i):=u(x)\quad \hbox{if $x\in E_i$,}$$
and we use the notation $u_{x_i}(x)=u'_i(x_i)$, again if $x\in E_i$.
With these notations, we can define in a proper way the Hamiltonians $H_i: E_i \times \R \to \R$ and our problem can be written as
\begin{eqnarray}\label{edp-junction}
&&  \left\{
  \begin{array}{ll}
\lambda u (x) -\I_i u (x) + H_i(x,u_{x_i})=0, & x\in E_i, \ 1\leq i\leq N,\\
\displaystyle \sum_{1\leq i\leq N} -u_{x_i} (O) = B,\\
u(a_i)= h_i, & 1\leq i\leq N,
\end{array}
\right.
\end{eqnarray}
where, for the sake of notations, we replace the dependence with respect to the edge $E_i$
by the simpler subscript $i$.

We refer the reader to  Sections~\ref{secdefjunction},~\ref{sec:def-visco} and~\ref{sec:gene-net} for precise definitions
and, in particular for the notion of viscosity solutions: this notion of solutions is the one which is used by Lions and
Souganidis~\cite{ls16,ls17} and which is called ``junction viscosity solutions''
in the book of Barles and Chasseigne~\cite{bc24}. 

We stress on the fact that even in this simplest version of a single interior vertex, the main difficulties of the problem are still there and for this reason we make an exhaustive study of this case. We also mention that it is possible to implement other boundary conditions on the vertices $\VV_b$ such as Neumann, or state-constraint conditions, as well as exterior conditions which are natural in nonlocal problems, see Section~\ref{secbc} for more details.

\subsection{Previous results.}
We prove the well-posedness of this simplified problem~\eqref{edp-junction} by using this notion of (junction) viscosity solutions.
Though there is an intense research activity on viscosity solutions for local Hamilton-Jacobi Equations (HJ-Equations for short)
on networks and, more generally, on stratified structures, few is known about nonlocal equations and this paper intends to be
one of the first studying nonlocal equations in this framework. 

Thus, before describing more precisely our results,
we briefly review the literature for local equations.

Despite we use a pure PDE approach,~\eqref{edp-nl} is intimately linked with optimal control problems when the Hamiltonians
$H_E$ are convex (or concave). In the case of networks, the analysis of deterministic optimal control problems posed on a 
simple junction and its relation with time-dependent HJ-Equations can be found in the seminal works Achdou et al.~\cite{acct13} 
(see also~\cite{aot15}) and Imbert et al.~\cite{imz13}. Then, Imbert and Monneau in~\cite{im17}
introduced the notion of \textsl{flux-limited solutions}, a particular notion at the junction point whose properties
makes it compatible with the optimal control perspective, and allows to weaken some of the standing assumptions
on the problem. In parallel, Lions and Souganidis~\cite{ls16, ls17} (see also Morfe~\cite{morfe20}) provide well-posedness
for HJ-Equations with Kirchhoff-type condition on the junction (as in~\eqref{edp-nl}),
providing well-posedness for the problem and investigated its relation with flux-limited solutions (See also~\cite{bbci18} for more 
on the relation among these two types of solutions).
In the case of more general networks,
we can mention the papers of Camilli and Schieborn~\cite{sc13} and Siconolfi~\cite{siconolfi22} where some of their
ideas are used in the treatment of general networks in Section~\ref{sec:gene-net}.

As far as elliptic or parabolic {\em nondegenerate} equations on networks are concerned, existence and uniqueness of classical solutions can be found in the work of Von Below~\cite{vonbelow88} for linear parabolic equations, while weak and classical solutions
for nonlinear equations are studied in~\cite{cms13, cm16, adlt19, adlt20, ohavi21}. The case of degenerate equations is more
delicate and is considered in~\cite{in17, ls17} using the notion of viscosity solutions.

For the readers who may be interested in stratified media, a pioneering work in this direction is Bressan and
Hong~\cite{bh07}: they show that the value function of a certain optimal control problem
is the viscosity solution of an associated HJ-Equation posed on the Euclidean space, but whose Hamiltonian encodes
the main features of the stratification. The book of Barles and Chasseigne~\cite{bc24} provides a more general version
of their results and, even if they do not really consider problems set on networks, their study of HJ-Equations with a
co-dimension one discontinuity contains basic ideas to compare ``junction viscosity solutions'' and ``flux-limited solutions'';
we borrow several of their technical arguments in this article. 

Let us mention
that the analysis of equations on
junctions has gone beyond existence and uniqueness, and it have also been addressed in more involved settings
such as Mean Field Systems on networks~\cite{cm16, adlt19, adlt20}, and homogenization on networks~\cite{fs20}. 


\subsection{Main results.}
In Section~\ref{secuniqueness}, we prove a (strong) comparison result for problem~\eqref{edp-junction},
from which classical Perron's method leads us to the well-posedness of a viscosity solution $u \in C(\Gamma)$ to the problem.

However, we have decided to present a  constructive approach that gives us a better understanding of the problem.
This is one of the aim of  Section~\ref{sec:exis-junct}.
The idea is to exploit the well-known vanishing viscosity method.
We replace the PDE in the edges $E_i$ in~\eqref{edp-junction}
with the approximating equation
\begin{equation}\label{eqvanish}
\begin{split}
\lambda u^\epsilon - \epsilon u^\epsilon_{x_i x_i} - \I_i u^\epsilon + H_i (x, u^\epsilon_{x_i}) & = 0 \quad \text{in $E_i$,}
\end{split}
\end{equation}
for $\epsilon \in (0,1)$, where $u^\epsilon_{x_i x_i}:=(u_i^\epsilon)''(x_i)$ for $1 \leq i \leq N$.

As a first step, we construct a viscosity solution $u^\epsilon \in C^{1,1}(\bar E_i) \cap C^2(E_i)$
following an idea
of Ohavi~\cite{ohavi21}.
It consists in solving the system with a Dirichlet boundary condition at the junction $O$
instead of the Kirchhoff condition. By a continuous dependence result and the Intermediate
Value Theorem, we prove we can choose the value of the Dirichlet condition at the junction
in order to recover the  Kirchhoff condition. 

We remark that our approximate solution  $u^\epsilon$ satisfies the Kirchhoff condition in the classical sense (the derivative at the end of each edge exists). The Kirchhoff junction condition together with the leading effect of the Hamiltonian allows us to prove uniform Lipschitz estimates on a neighborhood of $O$ for the family $\{ u^\epsilon \}_\epsilon$. By stability, a junction viscosity solution to~\eqref{edp-junction} is obtained in the passage to the limit $\epsilon \to 0$, and this solution is Lipschitz continuous at the junction point. That is the content of our existence result Theorem~\ref{teoexistence}.

Another advantage to use this approach, instead of the direct application of Perron's method,
is that we obtain, as a by-product, a general well-posedness
result for nonlinear viscous nonlocal Hamilton-Jacobi Equations with Kirchhoff conditions
on general networks (see Theorem~\ref{teo-gene}). Here, the Intermediate Value Theorem is replaced by its higher dimensional version encoded by Poincar\'e-Miranda Theorem, see Section~\ref{sec:gene-net}. This is a new result, extending in the framework
of nonlocal equations some previous works on local equations~\cite{cms13, cm16, ls17, in17, adlt19}.

The next step is the uniqueness of the solution, which is presented in Theorem~\ref{teo1}.
The proof of this result is a direct adaptation of the arguments presented in Lions and Souganidis~\cite{ls17}.
As it is usual in problems on networks, all the difficulties are related to the
vertices. Here, the Lipschitz continuity of the (sub)solution allows us to avoid the usual doubling variables procedure in the comparison proof, which is the main difficulty here since the problem is naturally discontinuous at $O$. In general,
$H_i(O, p) \neq H_j(O,p)$ and $\I_i u(O)\neq  \I_j u(O)$ for $i \neq j$.

Let us point out that the evaluation of the nonlocal operator at the junction is a
key point in our work. It requires the assumption~\eqref{levy} for $\sigma_E <1$.
The case $\sigma_E \geq 1$ is more delicate, starting with the evaluation of higher-order nonlocal operators at the junction point, and its viscosity formulation. One particular case that is considered here is when the nonlocality is \textsl{censored} to the edge, meaning that the integration in~\eqref{nl-intro} occurs only on $E_i$. Even in this simpler scenario, 
the evaluation of $\I_i u(O)$ necessarily requires that $u_{x_i}(O) = 0$,
see Guan and Ma~\cite[Theorem 5.3]{gm06}, which is not always compatible with prescribed Kirchhoff conditions at $O$. The nonlocal operators studied in~\cite{gm06} have a probabilistic interpretation and may arise in possible applications of stochastic optimal control problems on networks.
This kind of difficulty also appear in the context of (local) second-order problems.
In most of the cases the authors require the ellipticity degenerates on the junction, see~\cite{in17, ls17}.

We address the connections with the notion of flux limited solutions in Section~\ref{sec-FL}, proving its equivalence with Kirchhoff-type solutions, once an appropriate notion of flux limiter is defined. In this task, Lipschitz regularity of subsolutions and a relaxed evaluation of the nonlocal operator in the extended real line $\R \cup \{ -\infty, +\infty \}$ plays a key role to isolate the role of the critical slopes in the flux limiter, making the problem closer to the pure first-order case. That is the reason the ideas presented here hardly can be used on equations with higher-order nonlocal operators.

\medskip

The paper is organized as follows. In Section~\ref{secdefjunction} we provide the details of the definition of the junction and the PDE we consider, as well as the standing assumptions. In Section~\ref{sec:def-visco} we introduce the notion of solution and basic properties. In Section~\ref{sec:exis-junct} we prove the existence of a viscosity solution to the Kirchhoff problem, and provide some regularity estimates. In Section~\ref{secuniqueness} we prove a strong comparison result. We list some possible extensions concerning the boundary conditions in Section~\ref{secbc}. Section~\ref{sec:gene-net} is devoted to the case of general networks where we extend all the above results to this more complicated setting; the aim task is to introduce a suitable parametrization of each curve to be able to properly define $u_x$ and then $H_E$. In Section~\ref{sec-FL} we introduce the notion of flux limited solution and prove its equivalence with the Kirchhoff-type solutions in the context of junctions. In Appendix~\ref{AppregIu} we present the proofs of some auxiliar results that we use in the body of the paper.

\bigskip

\paragraph{\bf Acknowledgement.}
During the preparation of this work, several research visits have been realized,
respectively to USACH and IRMAR. In each case the concerned author wishes to
acknowledge their hosts for their hospitality and supports.
O.L. is partially supported by the ANR (Agence Nationale de la Recherche) through the COSS project ANR-22-CE40-0010 and the Centre Henri Lebesgue ANR-11-LABX-0020-01.
E.T. was supported by CNPq Grant 306022/2023-0, CNPq Grant 408169/2023-0, and FAPERJ APQ1 210.573/2024.

\section{Equation on a junction with Kirchhoff condition and Dirichlet boundary condition.}
\label{secdefjunction}

\subsection{Junction}\label{sec:junct1}
We quickly recall what we already explain in the introduction: we consider the case of a star-shaped network $\Gamma$ embedded in $\R^d$,
$d \geq 2$, where $O = (0,...,0) \in \R^d$ is the (unique) junction point ---in other words, $\VV_i =\{O\}$--- and we have a family of $N$ ``boundary vertices''
$\VV_b = \{\ver_1, ..., \ver_N \}$, hence of $N$ (open) edges $\{ E_i \}_{1\leq i\leq N}$ given by
$$E_i=\{t\ver_i, t \in (0,1)\}.$$
Of course, we assume that, for any $i,j$, $\ver_i\neq O$ and $\ver_i$, $\ver_j$ are not collinear; as a consequence the $E_i$ are non-empty and, if $i\neq j$, $E_i\cap E_j =\emptyset$. With these notations, we have
\begin{equation}\label{Gamma}
\Gamma :=  \bigcup_{i=1}^N \bar E_i,
\end{equation}
where $\bar E_i$ denotes the closure of the set $E_i$ in $\R^d$.

We now fix a natural parametrization of $\Gamma$ by arc length, namely we set $x_i=|x|$ on $E_i$, $a_i=|\ver_i|$ and $J_i = (0, a_i)$, where $|\cdot |$ denotes the usual Euclidean norm in $\R^d$.
We denote by $\gamma_i : \bar J_i=[0, a_i] \to \bar E_i$ the canonical bijection $\gamma_i(x_i) = x$ if $x \in \bar E_i$ and $|x|=x_i$. In particular, for each $x \in \Gamma \setminus \{ O \}$, there exists a unique $i$ such that $x \in E_i$, and in this case we write $x_i = \gamma_i^{-1}(x) \in J_i$.
Throughout the article we will often make the abuse of notation by identifying $x\in E_i$ and $x_i\in J_i$.
We also mention that the parametrization we choose is coherent with the one we will choose in the case
of general networks (see Section~\ref{sec:gene-net} for details).


We are going to consider the geodesic distance $\rho$ on $\Gamma$. Given $x,y\in \Gamma$, we define
\begin{equation*}
\rho(x,y) = \left \{ \begin{array}{ll} |x - y| \quad & \text{if $x,y\in \bar E_i$}, \\
|x| + |y| \quad &  \text{if $x\in \bar E_i$, $y\in \bar E_j$ with $i\neq j$}. \end{array} \right . 
\end{equation*}
Taking into account the way we parametrize the edges, we have $\rho(x,y) = |x_i - y_j|$ if $i = j$, and $\rho(x,y) = x_i + y_j$ if $i \neq j$. Thus, by abuse of notation, by writing $\rho(x_i, y_j)$ we mean  $\rho(\gamma_i(x_i), \gamma_j(y_j))$ for $x_i \in \bar J_i$ and $y_j \in \bar J_j$.
Notice that the geodesic distance $\rho$ and the distance induced on $\Gamma$ by the Eulidean norm are equivalent.

\subsection{Function spaces.}\label{sec:fct-space}

For a function $u : \Gamma \to \R$, we define $u_i = u \circ \gamma_i: \bar J_i \to \R$,
from which we have $u(x) = u_i(x_i)$ for $x\in \bar E_i$.
The function $u_i$ depends on the parametrization we chose but not its regularity.

We denote by $USC(\Gamma)$ (respectively  $LSC(\Gamma)$) the subset
of functions $u: \Gamma \to \R$ which are upper-semicontinuous (respectively
lower-semicontinuous) on $\Gamma$, that is, for all $i$, $u_i\in USC([0,a_i])$
(respectively  $u_i\in LSC([0,a_i])$).
The subset  $C(\Gamma)= USC(\Gamma)\cap LSC(\Gamma)$ is the set
of continuous functions on $\Gamma$.
It coincides with the usual notion of continuity
on $\Gamma$ induced by the geodesic distance.
For further purpose, we also introduce the subset
$SC(\Gamma)= USC(\Gamma)\cup LSC(\Gamma)$.

For $u: \Gamma \to \R$, 
we say that $u$ is differentiable at $x \in E_i$ if
$u_i$ is differentiable at $x_i$, and in that case we denote
$$
u_{x}(x) = u_{x_i}(x)=u_i'(x_i).
$$

For $m\in\mathbb{N}$, the space of $m$-times continuously differentiable functions on $\Gamma$ is defined by
\[
C^{m}\left(\Gamma\right):=\left\{ u\in C\left(\Gamma\right):u_i\in C^{m}([0,a_i])\text{ for all } i\right\}.
\]
Notice that $u\in C^{m}\left(\Gamma\right)$ is assumed to be continuous on $\Gamma$, all the $u_i$ are
$C^{m}$-continuously differentiable inside the edges and all their derivatives of order less than $m$ can be extended by
continuity to $[0,a_i]$. More precisely, when $u\in C^1\left(\Gamma\right)$, we define 
\begin{eqnarray}\label{prolong-derivative-O}
&&  u_{x_i}(O) := \lim_{x \to O, x \in E_i} u_x(x)
  = \lim_{x \to O, x \in E_i} \frac{u(x) - u(O)}{\rho(x, O)},\\
\label{prolong-derivative-bd}  
&&  u_{x_i}(\ver_i  ) := \lim_{x \to \ver_i, x \in E_i} u_x(x)
  = \lim_{x \to \ver_i , x \in E_i} - \frac{u(x) - u(\ver_i)}{\rho(x, \ver_i)}.
\end{eqnarray}
The above derivatives depend on the parametrization through the
orientation we chose for the edges (notice the minus sign in the definition of the second one).

To define more intrinsically the Kirchhoff condition, we may also use the notion of \textsl{inward derivative} of $u$
with respect to $E_i$ at $O$, denoted by $\partial_i u(O)$; this derivative has the advantage to be independent of the chosen
orientation of the edge but, of course, it is not independent of the parametrization. 
The inward derivative is the right one to consider when dealing
with Kirchhoff condition: in our simple framework, $ \partial_i u(O)= u_{x_i}(O)$, but we underline that it can be different
to $u_{x_i}(O)$ for general networks because of the different possible orientation.

In the sequel, for the sake of notations, we use the notation $u_{x_i}(x)$ for the derivatives of $u$ in $\bar E_i$.
We point out again that this simplification of notation will not be anymore possible
for general networks in Section~\ref{sec:gene-net}.

We finally recall that $u_i\in C^{0,\alpha_i}([0, a_i])$, $0<\alpha_i\leq 1$, if
\begin{eqnarray*}
[u_i]_{C^{0,\alpha_i}([0, a_i])}:= \sup_{x_i,y_i\in [0, a_i],\, x_i\not= y_i}\frac{|u_i(x_i)-u_i(y_i)|}{|x_i-y_i|^{\alpha_i}}<+\infty.
\end{eqnarray*}
It allows to define the set of H\"older continuous functions on $\Gamma$ by
\begin{eqnarray*}
C^{0,\alpha}(\Gamma):=
\left\{ u\in C\left(\Gamma\right) : u_i\in C^{0,\alpha_i}([0, a_i]) \text{ for $0<\alpha\leq \alpha_i\leq 1$}\right\}.
\end{eqnarray*}
Notice that, thanks to the inequality
\begin{eqnarray}\label{hold-ineg}
a^\alpha + b^\alpha \leq 2^{1-\alpha}(a+b)^\alpha \text{ for all $a,b\geq 0$, $0<\alpha\leq 1$},
\end{eqnarray}
if $u\in C^{0,\alpha}(\Gamma)$ and $K:=\max_i [u_i]_{C^{0,\alpha_i}([0,a_i])}$, then
\begin{eqnarray}\label{hold-gam}
&& |u(x)-u(y)|\leq  2^{1-\alpha}K\rho(x,y)^\alpha, \text{ for all $x,y\in\Gamma$.}
\end{eqnarray}

\subsection{Hamiltonian.}\label{sec:hamilt}
Now we explain the operators involved in our equation~\eqref{edp-nl}, starting with the Hamiltonian $H$. We assume the existence of a collection $\{ H_i \}_{1 \leq i \leq N}$
such that $H_i \in C(\bar E_i \times \R)$ for all $i$, 
and that satisfy 
%
\begin{eqnarray}\label{H}
&& \begin{array}{ll}
(i) & |H_i(x,p) - H_i(y, p)|\leq C_H(1+|p|)|x-y|, \quad x, y \in \bar E_i, p \in \R, \\[2mm]
(ii) & |H_i(x,p) - H_i(x, q)|\leq C_H |p - q|, \quad x \in \bar E_i, p, q \in \R, \\[2mm]
(iii) & C_H^{-1}|p| -C_H \leq H_i(x, p)\leq  C_H(1+|p|), \quad x \in \bar E_i, p \in \R,
\end{array}
\end{eqnarray}
for some $C_H > 1$. 

For simplicity, we choose here to deal with coercive Hamiltonians satisfying
  the classical assumptions coming from Optimal Control,
  see~\cite{barles13} for instance. Nevertheless, several results presented here can be readily applied to Hamiltonians with superlinear growth in the gradient, and/or less regularity on the state variable.

\begin{rema}\label{hamilt-param}
  With these notations, given $x \in \bar E_i$ and $p \in \R$, we naturally identify $H_i(x,p)$
  and $H_i(\gamma_i^{-1}(x),p)$, meaning that the HJ Equation depends on the parametrization of the network.
For instance, if we change our parametrization in Section~\ref{sec:junct1} with the opposite
orientation, $u_{x_i}$ changes sign and $H_i(x,p)$ is changed into $H_i (x,-p)$.
When thinking to applications, 
in optimal control problems for example, we often first parametrize the network to be
able to write the dynamic and the cost and then we derive the family of Hamiltonians $H_i$, which
allows to define the ``abstract'' Hamiltonian~$H$ on~$\Gamma$.
\end{rema}

\subsection{Nonlocal operator.}\label{sec:nonlocal}
Next, we define the nonlocal operator $\I u(x)$ for $x \in \Gamma$ and $u: \Gamma \to \R$.

We consider a family of two-parametric measurable functions $\{ \nu_{ij} \}_{1 \leq i,j \leq N}$ with the form $\nu_{ij}: \bar E_i\times (0,\infty)\to [0,+\infty)$, $1\leq i,j\leq N$,
  satisfying the following conditions:
\begin{eqnarray}\label{hyp-nu}
&& \begin{array}{c}
\text{There exist $\Lambda>0$ and $0<\sigma<1$ such that for all  $x,y\in \bar E_i$, $r>0$}\\[2mm]
\displaystyle  0\leq \nu_{ij}(x,r)\leq \frac{\Lambda}{r^{1+\sigma}},
  \quad
  |\nu_{ij}(x,r)-\nu_{ij}(y,r)|\leq  \frac{\Lambda}{r^{1+\sigma}}|x-y|.
\end{array}
\end{eqnarray}

In particular, this means that the following L\'evy integrability condition (see~\eqref{levy}) takes place
\begin{equation}\label{Levy1}
\sup_{1\leq i,j\leq N} \sup_{x \in \bar E_i}\int_{0}^{\infty} \min \{ r^{\gamma}, 1 \} \nu_{ij} (x, r)dr < +\infty, 
\end{equation}
for all $\gamma > \sigma$.

Then, setting, for all $x\in \bar E_i$,
\begin{eqnarray*}
  \nu_i (x,z)=
  \left\{
  \begin{array}{ll}
    \nu_{ii}(x,\rho(x,z))= \nu_{ii}(x,|x_i-z_i|), & z\in \bar E_i,\\
     \nu_{ij}(x,\rho(x,z))= \nu_{ij}(x,x_i + z_j ), & z\in \bar E_j, j\not= i,
  \end{array}
  \right.
\end{eqnarray*}
we define
\begin{eqnarray}
\label{def-I}  
 \hspace*{0.7cm} \I_{ii} u(x) 
&:=&  \int_{E_i} [u(z) - u(x)] \nu_{ii}(x, \rho(x,z))dz, \\
&=&  \int_{J_i} [u_i(z_i) - u_i(x_i)] \nu_{ii}(x, |x_i-z_i|)dz_i.\nonumber
\end{eqnarray}
It is convenient to introduce the further notation
\begin{eqnarray} \label{defIij}
  \I_{ij} u(x)&:=& \int_{E_j} [u(z) - u(x)] \nu_{ij}(x, \rho(x,z))dz\\
  &=& \int_{J_j} [u_j(z_j) - u_i(x_i)] \nu_{ij}(x, x_i+z_j)dz_j \nonumber
  .
\end{eqnarray}
in order that we can write 
\begin{eqnarray}\label{Iiu}
\I_i u(x) = \I_{ii} u(x)+ \sum_{j \neq i}  \I_{ij} u(x) \quad \text{ for  $x\in \bar{E}_i$}.
\end{eqnarray}

Notice that the nonlocal operator $\I_i u$ requires the values of $u$ on all $\Gamma$ in order to be evaluated.
Writing $\I_i u$ as in~\eqref{Iiu},
we call $\I_{ii} u$ the {\it censored part in $\bar{E}_i$} (which ``does not see'' the other edges)
and $\sum_{j \neq i}  \I_{ij} u$ is the {\it exterior part} (which ``interacts'' with other edges).
We finally say that $\I_i u$ is {\it censored to $\bar{E}_i$} if $\I_{ij} u=0$
 for $j\not= i$ (all the  $\nu_{ij}$, $i\neq j$, are zero).


Thanks to~\eqref{hyp-nu}-\eqref{Levy1}, the nonlocal operator is well-defined as soon as
$u$ is H\"older continuous on $\Gamma$ with an exponent larger than $\sigma$.
More precisely, we have the following estimate.
\begin{lema}\label{regI}
Consider the nonlocal operator $\I$ given by~\eqref{def-I}
under Assumption~\eqref{hyp-nu}.
If $u\in  C^{0,\gamma}(\Gamma)$ for some $\gamma > \sigma$, then the map 
$x \mapsto \I_i u (x)$  is in $C^{0,\gamma-\sigma}(\bar E_i)$ for each $i$.
\end{lema}

We present a proof of this lemma in Appendix~\ref{AppregIu}.


\subsection{PDE on the junction.}
We end this section by presenting the system~\eqref{edp-junction} we want to address.
We are interested in the existence and uniqueness of a function $u \in C(\Gamma)$, solving in each edge $E_i$, $1\leq i \leq N$, the nonlocal Hamilton-Jacobi equation
\begin{equation}\label{eq}
\lambda u(x) - \I_i u(x) + H_i(x, u_{x_i}(x)) = 0, \quad\hbox{for $x \in  E_i$}.
\end{equation}

We complement the equations with a Kirchhoff-type condition on the junction $O$,
\begin{equation}\label{Kirchhoff}
\sum_{i=1}^N - u_{x_i}(O)  = B, 
\end{equation}
for some $B\in\R$, and
Dirichlet boundary conditions at the end of each edge, that is
\begin{equation}\label{dirichlet-bc}
u_i (a_i)= h_i, \quad 1\leq i \leq N.
\end{equation}

The following {\bf steady assumption} is in force in all the paper:
\begin{eqnarray}\label{steady}
  && \text{$\lambda >0$, $H$ satisfies~\eqref{H}, $\nu_{ij}$ satisfy~\eqref{hyp-nu},
  $B, h_i\in \R$ are given.}  
\end{eqnarray}
  

\section{Definition of viscosity solutions on a Junction}
\label{sec:def-visco}

Using the definitions introduced in the previous section, we are
now in position to
state the notion of solution of~\eqref{edp-nl}. We give the definition for the junction problem~\eqref{edp-junction},
leaving to the reader the extension to general networks.


Given $\varphi\in SC(\Gamma)$,  a measurable subset $A \subseteq \Gamma$, and $x \in \bar E_i$,  $1\leq i \leq N$, we write
\begin{eqnarray*}
  && \I_i  [A] \varphi(x)  = \sum_{j=1}^{N} \int_{\bar E_j \cap  A} [\varphi(z) - \varphi(x)]
  \nu_{ij}(x, \rho(x,z)) dz,
\end{eqnarray*}
whenever the integrals make sense. Notice that 
$$
\I_i \varphi(x)=\I_i [A] \varphi(x)+\I_i  [\Gamma \setminus A] \varphi(x).
$$


In what follows, $B_\delta(x)=\{ z\in \Gamma :  \rho(x, z) < \delta\}$ is the open ball of radius $\delta$ centered
at $x\in \Gamma$ induced by the geodesic distance $\rho$, and
$B_\delta^c(x)=\Gamma\setminus B_\delta(x)$ is its complementary in $\Gamma$.

For $u\in SC(\Gamma)$, $\varphi\in  C^1(\Gamma)$ and every $\delta\in (0,1)$, we write
\begin{eqnarray}\label{def-G}
&& G_{i}^{\delta}(u,\varphi,p,x):=\lambda u(x) - \I_i  [B_\delta^c (x)] u(x) - \I_i  [B_\delta (x)] \varphi(x) + H_i(x, p)  
\end{eqnarray}
is well-defined for every $p\in\R$, $x\in \bar{E}_i$, $1\leq i\leq N$.

Recalling~\eqref{prolong-derivative-O}, we introduce the following

\begin{defi}\label{defi-visco}
We say that $u \in USC(\Gamma)$  is a viscosity subsolution to problem~\eqref{eq}-\eqref{Kirchhoff}-\eqref{dirichlet-bc}
if, for each $x \in \Gamma$, each $\delta > 0$ and each $\varphi \in C^1(\Gamma)$ such that $x$ is a maximum point of $u - \varphi$ in $B_\delta(x)$,  we have
\begin{eqnarray*}
G_{i}^{\delta}(u,\varphi, \varphi_{x_i}(x),x)\leq 0  &\mbox{if} & x \in E_i, 
\\
\min \Big{\{} \min_{1\leq i\leq N} G_{i}^{\delta}(u,\varphi, \varphi_{x_i}(O),O) , \ \sum_{1\leq i\leq N} -\varphi_{x_i}(O)  - B \Big{\}} \leq 0
& \mbox{if} & x=O,
\\
  \min \Big{\{} G_{i}^{\delta}(u,\varphi, \varphi_{x_i}(x),x)  , \ u(x)-h_i \Big{\}} \leq 0
& \mbox{if} & x=\ver_i.  
\end{eqnarray*}

We say that $u \in LSC(\Gamma)$  is a viscosity supersolution to problem~\eqref{eq}-\eqref{Kirchhoff}-\eqref{dirichlet-bc}
if, for each $x \in \Gamma$, each $\delta > 0$ and each $\varphi \in C^1(\Gamma)$ such that $x$ is a minimum point of $u - \varphi$ in $B_\delta(x)$, we have
\begin{eqnarray*}
G_{i}^{\delta}(u,\varphi, \varphi_{x_i}(x),x)\geq 0  &\mbox{if} & x \in E_i, 
\\
\max \Big{\{} \max_{1\leq i\leq N} G_{i}^{\delta}(u,\varphi, \varphi_{x_i}(O),O) , \ \sum_{1\leq i\leq N} -\varphi_{x_i}(O)  - B \Big{\}} \geq 0
& \mbox{if} & x=O,
\\
  \max \Big{\{} G_{i}^{\delta}(u,\varphi, \varphi_{x_i}(x),x)  , \ u(x)-h_i \Big{\}} \geq 0
& \mbox{if} & x=\ver_i.  
\end{eqnarray*}

A viscosity solution is a continuous function on $\Gamma$ which is simultaneously a viscosity sub and supersolution in the sense above. 
\end{defi}



As it is usual in the theory of viscosity solutions, we can always assume that $u(x) = \varphi(x)$ (that is, $\varphi$ ``touches" $u$ at $x$, from above or below depending the case) and that $x$ is either a strict global maxima or minima.

Next, we would like to provide an equivalent notion of solution for which we can get rid of the viscosity evaluation on the nonlocal operator and work directly with the function $u$.
For this purpose,
we introduce the subset of functions $u$ of $SC(\Gamma)$ such that
the nonlocal operator $\I u$ may be evaluated at $x\in\Gamma$ (with possibly
infinite value),
\begin{eqnarray}\label{def-F}
&&  \mathcal{F}_x=\bigg\{ u\in SC(\Gamma) \ : \
  \lim_{\delta\downarrow 0} \I_i [B_\delta^c(x)]u(x) \in [-\infty ,+\infty],\\[-2mm]\nonumber
&& \hspace*{6cm}\text{for all $i$ such that $x \in \bar{E_i}$} \bigg\}.
\end{eqnarray} 
When the limit above exists  in $[-\infty ,+\infty]$, we denote it by $\I_i u(x)$.
Notice that it does not depend on the function $\varphi$ chosen for computing the limit.
In the case of interest  $x=O$, all the limits  $\I_i u(O)$ for $1\leq i\leq N$,
must exist in order that $u\in \mathcal{F}_O$.


\begin{lema}\label{lem-F}  
  Let $u\in USC(\Gamma)$ (resp. $LSC(\Gamma)$)
  and $\varphi\in C^1(\Gamma)$ such that $u-\varphi$ has a local maximum (resp. minimum)
  at $x \in \bar E_i$ for some $i$. Then $u\in \mathcal{F}_x$ and $\I_i u(x)\in [-\infty, +\infty )$ (resp.  $(-\infty, +\infty ]$).
\end{lema}

\begin{proof} We follow the lines of~\cite[Lemma 3.3]{cs09}. 
We only prove the first statement when $u\in USC(\Gamma)$, the proof when $u\in LSC(\Gamma)$
being similar.
Assume that $u-\varphi$ has a local maximum at $x \in \bar E_i$ in $B_{\delta_0}(x)$ for some $1\leq i\leq N$, $\delta_0 > 0$ and $\varphi\in C^1(\Gamma)$ with $u(x) = \varphi(x)$.
We have
\begin{eqnarray*}
&& u(z)-\varphi(z)\leq u(x)-\varphi(x)=0, \quad z\in B_{\delta_0}(x).
\end{eqnarray*}

Denote $w_\delta = \varphi \chi_{B_\delta(x)} + u \chi_{B^c_\delta(x)}$, where $\chi_A$ is the indicator function of $A$. Notice that for each $j$,
$$
\I_{ij} w_\delta(x) = \I_{ij}[B_\delta(x)] \varphi(x) + \I_{ij}[B^c_\delta(x)] u(x),
$$ 
is well-defined and finite for each $\delta \in (0, \delta_0)$, and it is nonincreasing as $\delta \to 0$. 

Since $w_{\delta_0} - w_\delta \geq 0$ in $\Gamma$ and it is nondecreasing with $\delta$, using Monotone Convergence Theorem we get
$$
\lim_{\delta \downarrow 0} \I_{ij} (w_{\delta_0} - w_\delta)(x) \in (0, +\infty],
$$
meanwhile, by Dominated Convergence Theorem, we have
$$
\lim_{\delta \downarrow 0} \I_{ij}[B_\delta(x)] \varphi(x) = 0.
$$

Thus, writing 
$$
\I_{ij}[B^c_\delta(x)] u(x) = -\I_{ij} (w_{\delta_0} - w_\delta)(x) + \I_{ij}w_{\delta_0}(x) - \I_{ij}[B_\delta(x)] \varphi(x),
$$
we get 
$$
\lim_{\delta \downarrow 0} \I_{ij}[B^c_\delta(x)] u(x) \in [-\infty, +\infty)
$$ 
for all $j$, from which $u \in \mathcal F_x$.
%
\end{proof}

In order to state the equivalent definition of viscosity solutions
we need to introduce sub/superdifferentials and the term~\eqref{def-G}
with $\delta=0$.

Given an interval $I \subset \R$ and a function $u: I \to \R$, we recall that $p \in \R$ is in the superdifferential
$D^+_I u(x_0)$ of $u$ at $x_0 \in I$ if
$$
u(x) \leq u(x_0) + p(x - x_0) + o(|x - x_0|) \quad \mbox{for all $x$ in an $I$-neighborhood of  $x_0$.} 
$$
In the case of junctions, we say that ${\bf p} = (p_1, \cdots , p_N)$ is in the  superdifferential $D^+_\Gamma u(O)$
of $u$ at $O$ if, for each $1 \leq i \leq N$, then we have
$$
u_i(x_i) \leq u(O) + p_i x_i + o(x_i) \quad \mbox{for all $x_i$ in an $\bar J_i$-neighborhood of  $0$.} 
$$
Similarly to the classical case, if  ${\bf p}\in D^+_\Gamma u(O)$,
then there exists a function $\varphi \in C^1(\Gamma)$ such that $\varphi_{x_i}(O)=p_i$
and  $u - \varphi$ has a local maximum at $O$.
The above definitions can be stated for the subdifferential $D^-_\Gamma u(O)$.


We denote $G_i(u, p, x) = \lim_{\delta \downarrow 0} G_i^\delta(u, p, x)$ whenever the limit exists in $\R \cup \{-\infty, +\infty\}$ and meets the value 
\begin{eqnarray}\label{def-G0}
  && G_{i}(u,p,x) := 
 \lambda u(x) - \I_i u(x)+ H_i(x, p).
\end{eqnarray}
Notice that the nonlocal term decides whether $G_i(u,p, x) \in \R$ or not. It is the case when $u$ is smooth enough at $x$
(Lemma~\ref{regI}) but, in general, infinite values are possible.

\begin{lema}\label{equiv-def}
A function $u\in USC(\Gamma)$ is a viscosity subsolution to problem~\eqref{eq}-\eqref{Kirchhoff}-\eqref{dirichlet-bc} at the point $x \in \Gamma$ 
if and only if
\begin{eqnarray*}
G_{i}(u, p_i,x)\leq 0  &\mbox{if} & x \in E_i,\; p_i\in D_{\bar J_i}^+u_i(x_i), 
\\
\min \Big{\{} \min_{1\leq i\leq N} G_{i}(u,p_i,O) , \ \sum_{1\leq i\leq N} -p_i  - B \Big{\}} \leq 0
& \mbox{if} & x=O, \; {\bf p} \in D_{\Gamma}^+u(O),
\\
  \min \Big{\{} G_{i}(u,p_i,x)  , \ u(x)-h_i \Big{\}} \leq 0
& \mbox{if} & x=\ver_i, \; p_i\in D^+_{\bar J_i}u_i(a_i).
\end{eqnarray*}
The same equivalence holds for a $LSC$ supersolution with usual adaptations
(opposite inequalities, ``min'' replaced with ``max'' and the superdifferential
$D^+$ is replaced with the subdifferential $D^-$).
\end{lema}


\begin{proof}
We only provide the proof for subsolutions at $x=O$, the other cases being similar or simpler.

Assume that $u$ is a subsolution at $O$ in the sense of Definition~\ref{defi-visco}.
Let  ${\bf p} \in D_{\Gamma}^+u(O)$. There exists $\varphi\in C^1(\Gamma)$ such that
$u-\varphi$ has a local maximum at $O$ with $\varphi_{x_i}(O)=p_i$ for all $i$. From Lemma~\ref{lem-F},
$u\in \mathcal {F}_O$ and $-\I_i u(O)\in (-\infty,+\infty]$ for all $i$.
Using the subsolution characterization of Definition~\ref{defi-visco} at $x=O$, we
have that, either $\sum_i -\varphi_{x_i}(O)-B\leq 0$, or,
for every $\delta >0$,
there exists $1\leq i\leq N$, $i=i(\delta)$, such that
$G_i^\delta (u,\varphi, \varphi_{x_i}(O), O)\leq 0$.
In the first case, we conclude that $\sum_i -p_i-B\leq 0$ since  $\varphi_{x_i}(O)=p_i$.
We now suppose that the second case holds.
Since there is a finite number of edges, there exists $1\leq i_0\leq N$
and a subsequence $\delta_k\downarrow 0$ such that
$G_{i_0}^{\delta_k} (u,\varphi, \varphi_{x_{i_0}}(O), O)\leq 0$.
Since $u\in \mathcal {F}_O$, we can send $\delta_k$ to 0 in the previous
inequality to obtain $G_{i_0} (u,\varphi_{x_{i_0}}(O), O)\leq 0$, which is
the expected inequality in Lemma~\ref{equiv-def}. Note that, in this case,
we obtain that  $-\I_{i_0} u(O)$ is finite.

Conversely, assume that $u$ satisfies the inequality of Lemma~\ref{equiv-def} at $x=O$.
Let $\varphi\in C^1(\Gamma)$ such that $u-\varphi$ has a global maximum at $O$.
Then  ${\bf p}:= (\varphi_{x_1}(O), \cdots ,  \varphi_{x_N}(O)) \in D_{\Gamma}^+u(O)$.
If $\sum_i -\varphi_{x_i}(O)-B\leq 0$, then the viscosity inequality for
the subsolution at $x=O$ in Definition~\ref{defi-visco} holds and we are
done. Otherwise, there exists $i$ such that $G_{i} (u,\varphi_{x_{i_0}}(O), O)\leq 0$.
In particular, from Lemma~\ref{lem-F}, $-\I_{i} u(O)$ is finite.
Then, for all $\delta>0$, we have
\begin{eqnarray*}
  G_i^\delta (u,\varphi, \varphi_{x_i}(O), O)
  &=& G_{i} (u,\varphi_{x_{i_0}}(O), O) \\
  &&   \I_{i} u(O) - \I_i  [B_\delta^c (O)] u(O) - \I_i  [B_\delta (O)] \varphi(O) \\
  &\leq &   \I_i  [B_\delta (O)] u(O) -  \I_i  [B_\delta (O)] \varphi(O) \leq 0,
\end{eqnarray*}
the last inequality coming from the fact that,
$O$ being a maximum of $u-\varphi$, we have
$u-u(O)\leq \varphi -\varphi(O)$ in  $B_\delta (O)$, hence
$\I_i  [B_\delta (O)] u(O) \leq  \I_i  [B_\delta (O)] \varphi(O)$.
We conclude that $u$ is a subsolution in the sense of  Definition~\ref{defi-visco}.
\end{proof}

Finally, the above definitions can be extended to second-order equations like~\eqref{eqvanish}, with test functions $\varphi\in C^2(\Gamma)$. Moreover, every classical solution to the problem, that is, a function $C^2(\Gamma)$ such that~\eqref{eq} or~\eqref{eqvanish} complemented by~\eqref{Kirchhoff}-\eqref{dirichlet-bc} holds pointwisely, 
is a viscosity solution in both senses above. 



\section{Existence for the Kirchhoff-type problem on a junction.}
\label{sec:exis-junct}


Let $\theta \in \R$ and $\epsilon \in (0,1)$. The purpose of this section is to construct a function $u^\epsilon \in C(\Gamma)$ solving
the viscous Dirichlet problem
\begin{equation}\label{eqaprox0}
  \left \{ \begin{array}{l} \lambda u - \epsilon u_{x_i x_i} - \I_i u + H_i(x, u_{x_i}) = 0
    \quad \mbox{on} \ E_i, 1\leq i\leq N, \\
u(O) = \theta, \\
u_i(a_i) = h_i, \quad 1\leq i\leq N.
  \end{array} \right .
\end{equation}

The boundary condition at $a_i$ is considered fixed since it coincides with the boundary condition of the original problem. The more careful analysis is performed on the dependence of the solution in terms of $\theta$, since we will require to fix it appropriately to find a solution satisfying the Kirchhoff condition at $O$. 

We start by constructing viscosity sub- and supersolutions to~\eqref{eqaprox0}, which achieve
the boundary conditions pointwisely.

We fix $\alpha\in (0,1)$, and, for all $L, \delta > 0$, we define
\begin{equation}\label{defpsi}
\psi_{L,\delta}(x) = 
\left\{\begin{array}{ll}    
L(x - x^{1 + \alpha}), & \mbox{for} \ 0 < x \leq \delta,\\
 \psi(\delta)  & \mbox{for} \ x > \delta,\\
\end{array}\right.
\end{equation}
and we introduce
\begin{equation}\label{defC0}
C_0 := \frac{1}{\min \{ \lambda, 1 \}} \Big{(}  \max_{1\leq i\leq N, \, x\in \bar{J_i}} |H_i(x,0)| + |\theta| + \max_{1\leq i\leq N} |h_i| \Big{)}.
\end{equation}

\begin{lema}\label{sous-sur-sol}
Under the assumption~\eqref{steady},
there exist $L,\delta$ depending on the data, and $L_\epsilon, \delta_\epsilon$
depending on the data and the ellipticity constant $\epsilon$, such that
the functions $\psi^+$, $\psi^-: \Gamma\to \R$ defined by
\begin{eqnarray*}
&& \psi^+_i(x)= \min \{ C_0, \theta + \psi_{L,\delta}(x), h_i+ \psi_{L,\delta}(a_i-x)\}, \quad x\in \bar{J}_i,\\
&& \psi^-_i(x)= \max \{ -C_0, \theta - \psi_{L_\epsilon,\delta_\epsilon}(x), h_i- \psi_{L_\epsilon,\delta_\epsilon}(a_i-x)\}, \quad x\in \bar{J}_i,
\end{eqnarray*}
are continuous viscosity supersolution and subsolution, respectively, of~\eqref{eqaprox0},
satisfying $\psi^\pm(O)=\theta$ and $\psi_i^\pm(a_i)=h_i$.
\end{lema}

\begin{proof}
Throughout the proof, we systematically identify $\bar E_i$ with $\bar J_i$, $x$ with $x_i$, and $\psi$ with $\psi_i$.

At first, notice that the constant functions $x\in\Gamma \mapsto \pm C_0$, where $C_0$ is defined by~\eqref{defC0},
are obviously supersolution and subsolution of~\eqref{eqaprox0}, respectively. But they do not
satisfy the boundary condition pointwisely.
  
We divide the proof in several parts.
\smallskip

\noindent{\it Supersolution property for $\psi^+$.}
Since the minimum of supersolutions is
still a supersolution, it is enough to prove that, for $1\leq i \leq N$, both
$\theta + \psi_{L,\delta}(x)$ and $h_i+ \psi_{L,\delta}(a_i-x)$ are supersolutions
in $\bar{J}_i$ for $L$ enough large and $\delta$ enough small. The two proofs being similar,
we concentrate on $\theta + \psi_{L,\delta}(x)$.

  
For simplicity, we write $\psi$ for $\psi_{L,\delta}$.
Straightforward computations drive us to
\begin{align*}
\psi_x(x) = L(1 - (1 + \alpha) x^\alpha), \quad \psi_{xx}(x) = -L(1 + \alpha) \alpha x^{\alpha - 1}, \quad x \in [0, \delta).
\end{align*}

%

Set $\delta < (2(1+\alpha))^{-1/\alpha}$ and $\delta < \min_i a_i$. Then,
we have $L/2\leq \psi_x(x) \leq L$ for $x\in [0,\delta)$
and $L\delta /2\leq \psi(x)= \psi(\delta) \leq L\delta$ for $x\in [\delta, a_i]$.
In particular, $\psi$ is nondecreasing and Lipschitz continuous with constant $L$ in each $\bar J_i$.

Now we estimate the nonlocal operators. 
Let $x \in \bar E_i$ with $x_i \in \bar J_i \cap [0,\delta]$. We have that
\begin{eqnarray*}
| \I_{ii} \psi(x) |
&\leq&
\int_{J_i \cap \{ \rho(x, \gamma_i(z)) \leq \delta \}} |\psi_i(z)  - \psi(x)|  \nu_{ii}(x, \rho(x,\gamma_i(z))) dz \\
&& + \int_{J_i \cap \{ \rho(x,\gamma_i(z)) > \delta \}} |\psi_i(z)  - \psi(x)| \nu_{ii}(x, \rho(x,\gamma_i(z))) dz\\ 
&\leq &
\Lambda L \int_{0}^{x_i+\delta} |x_i - z||x_i - z|^{-1-\sigma} dz \\
&& + \Lambda \psi(\delta)^{1-\sigma}  \int_{x_i+\delta}^{a_i} (L|x_i - z|)^{\sigma} |x_i - z|^{-1-\sigma} dz\\
&\leq& \Lambda L \int_{0}^{x_i + \delta} |x_i - z|^{-\sigma} dz + \Lambda L \delta^{1 - \sigma} \int_{\delta}^{a_i} |x_i - z|^{-1} dz\\
&\leq &
\Lambda L \left(\frac{2^{1-\sigma}}{1-\sigma}+ \log a_i - \log \delta \right) \delta^{1-\sigma}.
\end{eqnarray*}
For $1\leq j\leq N$, $j \neq i$, similarly we have
\begin{eqnarray*}
| \I_{ij} \psi(x)|
&\leq&
\int_{J_j \cap \{ \rho(x,\gamma_j(z)) \leq 2\delta \}} |\psi_j(z)  - \psi_i(x_i)|  \nu_{ij}(x, \rho(x,\gamma_j(z))) dz \\
&& + \int_{J_j \cap \{ \rho(x,\gamma_j(z)) > 2\delta \}} |\psi_j(z)  - \psi_i(x_i)|  \nu_{ij}(x, \rho(x,\gamma_j(z))) dz\\ 
&\leq &
\Lambda L \int_{0}^{2\delta - x_i} |x_i + z|^{-\sigma} dz + \Lambda L \delta^{1 - \sigma} \int_{\delta}^{a_j} |x_i + z|^{-1} dz\\
&\leq &
\Lambda L \left(\frac{2^{1-\sigma}}{1-\sigma}+ \log a_i - \log \delta \right) \delta^{1-\sigma}.
\end{eqnarray*}
We summarize the nonlocal estimates as
\begin{equation}\label{Ipsi-full}
|\I_i \psi(x)| \leq C_\sigma (1 +|\log \delta|) L \delta^{1 - \sigma}, \ 1 \leq i \leq N,
\end{equation}
where $C_\sigma > 0$ depending on the data is of order $(1 - \sigma)^{-1}$ as $\sigma\to 1^-$  but remains stable as $\sigma \to 0^+$.


We use the notation  $G_i^\delta, G_i$ of Section~\ref{sec:def-visco}, adding 
the vanishing viscosity term accordingly.

From the above computations together with the coercivity assumption in~\eqref{H}, for all $i$ and all $x\in J_i \cap (0,\delta)$, we have
\begin{eqnarray*}
&& G_i (\theta +\psi, \psi_{x_i}(x), x)\\
&=& \lambda (\theta +\psi) - \epsilon \psi_{x_i x_i} - \I_i \psi(x) + H_i(x, \psi_{x_i}(x))\\
  &\geq & \lambda \theta +\epsilon L (1 + \alpha) \alpha x^{\alpha - 1}
-  C_\sigma (1 +|\log \delta|) L \delta^{1 - \sigma}
  + C_H^{-1}\frac{L}{2}-C_H.
\end{eqnarray*}
Since $\sigma < 1$, it is possible to choose
$\delta$ possibly smaller in order to
$$
C_\sigma (1 +|\log \delta|) \delta^{1 - \sigma} \leq \frac{1}{4 C_H}.
$$
Under this choice, and taking $L\geq 8 C_H^2$, we obtain
\begin{eqnarray*}
G_i (\theta+\psi, \psi_{x_i}(x), x)
\geq \lambda \theta + \frac{1}{8C_H} L
\quad \text{on $J_i \cap (0,\delta)$}.
\end{eqnarray*}
Choosing $L> 8 C_H \lambda |\theta|$, the left-hand side of the
above inequality is nonnegative, from which $\theta +\psi$ is a viscosity supersolution
in $J_i\cap (0,\delta)$.

Similar computations yield that $h_i +\psi(a_i-x)$ is a supersolution on $(a_i - \delta, a_i)$,
choosing $L$ large and $\delta$ small  enough. 
 Enlarging $L$ if necessary, we can assume that 
 $$
 \theta + \psi(\delta), h_i + \psi_L(\delta) \geq C_0.
 $$
 
%
%
Then, the value $C_0$ is active in the $\min$ defining $\psi^+$ on $(\delta, a_i - \delta)$, from which we infer the viscosity inequality on the full interval $J_i$.
Thus, $\psi^+$ is a viscosity supersolution for the problem with $\psi^+(O) = \theta$ and $\psi^+_i(a_i) = h_i$.

\smallskip

\noindent{\it Subsolution property for $\psi^-$.}
As, for the construction of the supersolution, the property that a maximum of subsolutions
is a subsolution allows to check the subsolution property for
$\theta - \psi(x)$ and $h_i- \psi(a_i-x)$ in $\bar J_i$ separately. Here also, the computations
being similar, we make them in the case of $\theta - \psi(x)$.

Recalling~\eqref{Ipsi-full},
%
for all $x\in J_i \cap (0,\delta)$ we have
\begin{eqnarray*}
&& G_i(\theta-\psi, -\psi_{x_i}(x), x)\\
&=& \lambda (\theta -\psi) - \epsilon \psi_{x_i x_i} - \I_i \psi(x) + H_i(x, -\psi_{x_i})\\
&\leq&
\lambda \theta -\epsilon L(\alpha +1)\alpha x^{\alpha -1}+ C_\sigma \delta^{1-\sigma}+C_H (1+L)\\
&\leq &
\lambda \theta + C_H + L\left( -\epsilon (\alpha +1)\alpha x^{\alpha -1} +  C_\sigma \delta^{1-\sigma}+C_H\right).
\end{eqnarray*}
Here, we need to use the ellipticity of the PDE~\eqref{eqaprox}, which will
provide a subsolution depending on $\epsilon$.
We first choose
\begin{eqnarray*}
  \delta=\delta_\epsilon \leq \left(\frac{\epsilon (1+ \alpha)\alpha}{C_\sigma + C_H+1}\right)^{\frac{1}{1-\alpha}},
\end{eqnarray*}
yielding
\begin{eqnarray*}
&&  G_i(\theta-\psi, -\psi_{x_i}(x), x)
\leq \lambda \theta +C_H-L, \quad x\in J_i \cap (0,\delta_\epsilon).
\end{eqnarray*}
Choosing $L\geq \lambda |\theta| +C_H$, the left-hand side of the above inequality becomes nonpositive.
Therefore $\theta-\psi$ is subsolution in $(0,\delta_\epsilon)$. Similar computations can be made with $h_i - \psi$ on $(a_i - \delta_\epsilon, a_i)$.

Enlarging $L=L_\epsilon\geq 2\delta_\epsilon^{-1}(\theta +C_0)$, we get
\begin{eqnarray*}
\theta -\psi(\delta), h_i - \psi(\delta) 
\leq -C_0, 
\end{eqnarray*}
from which, as before, we get $\psi_i^-$ is a viscosity subsolution for the problem on $\Gamma$, with $\psi^-(O) = \theta$ and $\psi^-_i(a_i) = h_i$. 
%
\end{proof}


Thanks to the above lemma, the existence of a unique viscosity solution
to~\eqref{eqaprox0} is an immediate consequence of Perron's method and comparison principles for this type of Dirichlet problem on $\Gamma$, see Theorem~\ref{teo-censA} in the Appendix.

\smallskip

However, we give an alternative proof, which consists, roughly speaking,
to construct a solution branch by branch.
The proof uses only Theorem~\ref{teo-censA} in the very particular case
of the nonlocal Dirichlet problem censored to a unique interval,
\begin{equation}\label{edp-cens1}
  \left \{ \begin{array}{l} \lambda u - \epsilon u_{x x} - \I u + H(x, u_{x}) = 0
    \quad \mbox{on} \ (0,1), \\
u(0) = a, \quad  u(1)=b,\\
  \end{array} \right .
\end{equation}
where $\I u=\I_{ii}u$ with $\Gamma=J_i=(0,1)$ (in the sense of Section~\ref{sec:nonlocal}).
We think that our proof below has its own interest
since it is constructive once we can solve~\eqref{edp-cens1}  efficiently.
\medskip

Let us turn to our construction.
For $\eta \in (0,1)$, for each $1\leq i,j\leq N$ and $x \in \bar E_i$, $r>0$, we consider the modified kernel
$$
\nu_{ij}^\eta(x, r) = \min \{ \eta^{-1}, \nu_{ij}(x,r)\},
$$
and the nonlocal operator
$\I_i^\eta$ defined as in~\eqref{def-I} but taking into account this kernel. It is easy to see that
$\nu_{ij}^\eta(x,\rho(x,z))dz_j$ is a bounded measure on $J_j$, which satisfies~\eqref{hyp-nu}
with the same constants $\Lambda, \sigma$, independently of $\eta$. Moreover,
if $u\in C^{0,\gamma}(\Gamma)$, $\gamma >\sigma$, then
we have $\I^\eta u \to \I u$ uniformly in $\Gamma$ as $\eta \to 0$
(the proof is similar to the one of Lemma~\ref{regI}).

\medskip

We consider the following approximate problem
\begin{equation}\label{eqaprox}
  \left \{ \begin{array}{l} \lambda u^{\epsilon, \eta} - \epsilon u_{x_i x_i}^{\epsilon, \eta} - \I_i^\eta u^{\epsilon, \eta} + H_i(x, u^{\epsilon, \eta}_{x_i}) = 0
    \quad \mbox{on} \ E_i, 1\leq i\leq N, \\
u^{\epsilon, \eta}(O) = \theta, \\
u_i^{\epsilon, \eta}(a_i) = h_i, \quad 1\leq i\leq N.
  \end{array} \right .
\end{equation}

\begin{lema}\label{lematildeueps}
Under the assumption~\eqref{steady},
there exists a viscosity solution $u^{\epsilon, \eta} \in C(\Gamma)$ to problem~\eqref{eqaprox}, and we have the estimate
\begin{equation}\label{bound_uieps}
\| u^{\epsilon, \eta} \|_{L^\infty(\Gamma)} \leq C_0,
\end{equation}
where $C_0$ is defined in~\eqref{defC0}.
\end{lema}

\begin{proof}
Since  $\nu_{ij}^\eta$ are bounded measures, we can set
$$
\lambda_{ij}^\eta (x) := \int_{J_j} \nu_{ij}^\eta(x, \rho(x,z))dz_j \text{ for $x\in \bar E_i$},
$$
from which we can write
\begin{align}\label{deftildeI}
\begin{split}
\I^\eta_{ij} u(x) = & \int_{J_j} u_j(z_j) \nu_{ij}^\eta(x, \rho(x, z))dz_j -  \lambda^\eta_{ij}(x) u_i(x).
\end{split}
\end{align}


We construct the solution of the problem~\eqref{eqaprox} as the uniform limit of a sequence $\{ u^k \}_k$ that we define inductively as follows. After relabeling the set of indices $1 \leq i \leq N$ if necessary, we split it as $\{ 1,...,N', N'+1, ..., N\}$ such that, for all $N' + 1 \leq i \leq N$, $\I_i^\eta =\I_{ii}^\eta$ is censored to $\bar{E}_i$, in other words,  
$$
\lambda^\eta_{ij}(x) = 0, \quad \mbox{for all} \ x \in \bar E_i.
$$

For each $N'+1 \leq i \leq N$, we solve the censored Dirichlet problem
\begin{equation}\label{censoredRennes}
\left \{ \begin{array}{l}
\lambda u - \epsilon u_{x_i x_i} - \I_{ii}^\eta u(x) + H_i(x, u_{x_i}) = 0 \quad \mbox{in} \ J_i, \\
u(0) = \theta, u(a_i) = h_i, \end{array} \right .
\end{equation}
which has a unique viscosity solution $u_i \in C([0,a_i])$ in view of Theorem~\ref{teo-censA}.
For all $k \in \N$ and $N' + 1 \leq i \leq N$, we set $u^k_i := u_i$.

Now we deal with the edges $1 \leq i \leq N'$. We initialize the sequence $\{u^k_i\}_{k, 1\leq i\leq N'}$ with affine functions $u_i^0 \in C(\bar J_i)$ taking values $u_i^0(0) = \theta$, $u_i^0(a_i) = h_i$ for each $1 \leq i \leq N'$. This concludes the definition of $u^0 \in C(\Gamma)$.

For $k \geq 0$ and $u^k$ already defined,  using again  Theorem~\ref{teo-censA},
we construct $u_i^{k + 1} \in C(\bar J_i), 1 \leq i \leq N'$ as the unique solution to the censored Dirichlet problem
\begin{equation}\label{eqiterative1}
\left \{ \begin{array}{l}
\Lambda_i^\eta(x) u - \epsilon u_{x_i x_i} - \I_{ii}^\eta u(x) - \sum \limits_{j \neq i} f^k_{ij}(x) + H_i(x, u_{x_i}) = 0 \quad \mbox{in} \ J_i, \\
u(0) = \theta, u(a_i) = h_i, \end{array} \right .
\end{equation}
where we have denoted
\begin{align*}
& \Lambda_i^\eta(x) = \lambda + \lambda_i^\eta(x) := \lambda + \sum_{j \neq i} \lambda_{i j}^\eta(x), \\
& f^k_{ij}(x) = \int_{J_j} u_j^k(z_j) \nu_{i j}^\eta(x, \rho(x, z))dz_j, \quad j \neq i.
\end{align*}

In the construction of each $u_i^{k+1}$, we only require the information of
the continuous functions $u^k$ the functions $f_{ij}^k$. Notice that $f_{ij}^k$ is continuous by Lebesgue Theorem.

The key step on the proof is the following

\smallskip

\noindent
{\bf Claim.} \textsl{There exists $\lambda^* \in (0,1)$ such that, for all $k$ large enough we have
\begin{align}
\label{bounduk}
& \max_{1\leq i \leq N} \| u^k_i \|_{L^\infty(\bar J_i)} \leq  C_0, \\
\label{Cauchy2}
& \| u_i^{k + 1} - u_i^{k} \|_{L^\infty(\bar J_i)} \leq \lambda^* \| u^{k} - u^{k-1}\|_{L^\infty(\Gamma)}, \quad \mbox{for all} \ 1 \leq i \leq N,
\end{align}
where $C_0$ is the constant in~\eqref{defC0}.
}

\smallskip

We start with~\eqref{bounduk}. For $k = 0$ this is evident, as well as for each $u_i^k$ with $N'+1 \leq i \leq N$
by the comparison principle Lemma~\ref{comp-censA} for~\eqref{censoredRennes}. If the bound is true for $u^k$, then we have that for each $1 \leq i \leq N'$,  $u_i^{k+1}$ is a viscosity subsolution to
$$
\Lambda_{i}^\eta u - \epsilon u_{x_i x_i} - \I_{i i}^\eta u + H_i(x, u_{x_i}) \leq C_0 \lambda_{i}^\eta(x) \quad \mbox{in} \ J_i,
$$
with its respective boundary conditions. From here, recalling that $\Lambda_i^\eta = \lambda + \lambda_i^\eta$, it is easy to see that the constant function equal to $C_0$ in $J_i$ is a supersolution for the problem solved by $u_i^{k + 1}$, from which $u^{k + 1}_i \leq C_0$. The lower bound can be found in the same way, and~\eqref{bounduk} follows. Notice that  the estimate in~\eqref{bounduk} does not depend neither on $k, \epsilon$ nor on $\eta$.

Next, we prove~\eqref{Cauchy2}. Notice that this inequality trivially holds for $N'+1 \leq i \leq N$ since in that case we have $u_i^{k + 1} = u_i^k$ for all $k \geq 0$. Thus, we concentrate on the case $1 \leq i \leq N'$.

For each $k$, denote $w_i^k = u_i^{k + 1} - u_i^k \in C(\bar J_i)$. Under the Lipschitz assumption on $H$ in~\eqref{H}, $w^k_i$ solves the problem
\begin{align*}
\left \{ \begin{array}{l} \Lambda^\eta_i w^{k} - \epsilon w^k_{x_i x_i} - \I_{i i}^\eta w^k - C_H |w^k_{x_i}| \leq \lambda_{i}^\eta \| w^{k - 1} \|_{L^\infty(\Gamma)} , \quad \mbox{in} \ J_i, \\
 w^{k}(0) = w^k(a_i) = 0, \end{array} \right .
\end{align*}
see~\cite[Lemma 5.3]{barles13} or~\cite[Lemma 3.7]{bbp97}.
Denote 
$$
\lambda^* = \max_{1 \leq i \leq N'} \max_{x \in \bar J_i} \frac{ \lambda_i^\eta(x)}{\Lambda_i^\eta(x)}  \in (0, 1).
$$
It is easy to see that the constant function $x \mapsto \lambda^* \| w^{k - 1}\|_{L^\infty(\Gamma)}$
is a supersolution for the problem solved by $w^k_i$, from which by comparison 
we get
$$
w^{k}_i \leq \lambda^* \| w^{k - 1}\|_{L^\infty(\Gamma)} \quad \mbox{in} \ \bar J_i,
$$
and since a lower bound can be established in the same way, we arrive at
$$
\| w^k_i \|_{L^\infty(\bar J_i)} \leq \lambda^* \| w^{k - 1}\|_{L^\infty(\Gamma)},
$$
which concludes~\eqref{Cauchy2}.

Thus, we can pass to the limit as $k \to +\infty$ in~\eqref{eqiterative1}, and by stability properties of viscosity solutions we arrive at a solution to~\eqref{eqaprox}.
\end{proof}

We now deduce the well-posedness of~\eqref{eqaprox0} from the above lemma and give properties of the solution. 

\begin{prop}\label{prop1} Assume~\eqref{steady}.
For each $\theta \in \R$, there exists a unique viscosity solution $u^\epsilon \in C(\Gamma)$ to problem~\eqref{eqaprox0}. For each $1\leq i \leq N$, we have

\begin{itemize}
\item[(i)]
There exist $L,\delta > 0$, not depending on $\epsilon$,
such that
\begin{eqnarray}
&&  u_i^\epsilon(x) \leq \theta + L x, \quad x \in [0, \delta], \label{barriereOabov}\\
&&  u_i^\epsilon(x) \leq h_i + L (a_i-x), \quad x \in [a_i -\delta, a_i].\label{barriereaiabov}
\end{eqnarray}

\item[(ii)]
There exist $L_\epsilon,\delta_\epsilon > 0$
such that
\begin{eqnarray}
&&  \theta - L_\epsilon x \leq u_i^\epsilon(x), \quad x \in [0, \delta_\epsilon], \label{barriereObel}\\
&&  h_i - L_\epsilon (a_i-x) \leq u_i^\epsilon(x), \quad x \in [a_i -\delta_\epsilon, a_i].\label{barriereaibel}
\end{eqnarray}

\item[(iii)] For every $r >0$, there exists $C_r >0$, independent of $\epsilon$, such that
\begin{eqnarray} \label{inLip} 
&& |u_i^\epsilon(x) - u_i^\epsilon(y)| \leq C_r |x - y|, \quad x, y \in [r, a_i-r]. 
\end{eqnarray}
If $C_\epsilon=\max \{L,L_\epsilon,C_{\delta \wedge \delta_\epsilon}\}$, then
\begin{eqnarray}  \label{uepslip}
&& |u_i^\epsilon(x) - u_i^\epsilon(y)| \leq C_\epsilon |x - y|, \quad x, y \in \bar J_i. 
\end{eqnarray}

\item[(iv)] $u^\epsilon \in C^{2, 1-\sigma}(\Gamma)$.

\item[(v)] The map $\theta\in\R\mapsto u^\epsilon \in C^{2, \alpha}(\Gamma)$
is continuous for all $\alpha <1-\sigma$.
  
\end{itemize}
\end{prop}

\begin{proof}
Let $u^{\epsilon,\eta}$ be the solution found in Lemma~\ref{lematildeueps}. In what follows, we are going to prove that this solution satisfies the estimates given in the statement of this proposition, and that these estimates are independent of $\eta$. Thus, the results for the problem~\eqref{eqaprox0} follow by Ascoli Theorem and stability, sending $\eta \to 0$.
The uniqueness of $u^{\epsilon}$ comes from Lemma~\ref{comp-censA}.
\medskip

\noindent
(i) and (ii).
The proofs of the barriers are a direct application
of the comparison principle Lemma~\ref{comp-censA}, which yields
$\psi^-\leq u^{\epsilon,\eta}\leq \psi^+$ in $\Gamma$, where $\psi^\pm$ are the sub and supersolution
constructed in Lemma~\ref{sous-sur-sol}.
Notice that~\eqref{barriereObel} and~\eqref{barriereaibel} depend on
the ellipticity constant $\epsilon$ since $\psi^-$ depend on $\epsilon$ in Lemma~\ref{sous-sur-sol}.
\medskip

\noindent
(iii) {\it Lipschitz estimates.}
\smallskip

\noindent{\it Step 1. Claim: There exists $K_0 = K_0(N, \lambda, \Lambda, \sigma, C_H, C_0, \mathrm{osc}_\Gamma \{ u^{\epsilon, \eta}\})$ such that, for any $x_0\in\Gamma$ and $\varphi(x)=K\rho(x,x_0)$ with $K>K_0$,
\begin{eqnarray}\label{sup683}
&& \sup_{x\in \Gamma} \{ u^{\epsilon, \eta} (x) - u^{\epsilon, \eta} (x_0)- \varphi(x) \}
\end{eqnarray}
is achieved either at $\bar{x}=x_0$, or at
a vertex $\bar{x} \in \VV$.}

Let us underline that this claim relies on the coercivity of the Hamiltonian and does not depend neither on the ellipticity constant $\epsilon$ nor on $\eta$ since $u^{\epsilon, \eta}$ and $\mathrm{osc}_\Gamma \{ u^{\epsilon, \eta}\}$ are bounded independently of $\epsilon, \eta$.

We prove the claim, assuming without loss of generality, that $\bar x\in E_i$ and $\bar x\neq x_0$. It follows that we can use $\varphi$ as a test-function for the subsolution $u^\epsilon$ of~\eqref{eqaprox0} at $\bar x$. For 
any $\varrho > 0$ we have
\begin{eqnarray}\label{viscGeps}
G_i^\varrho(u^{\epsilon, \eta}, \varphi, \varphi_{x_i}(\bar x), \bar x) 
\leq 0.  
\end{eqnarray}


We estimate the different terms: we have 
\begin{align*}
& u^{\epsilon, \eta} (\bar{x})\geq -C_0, \\
& H(\bar x, \varphi_{x_i}(\bar x))\geq C_H^{-1}|\varphi_{x_i}(\bar x)| -C_H= C_H^{-1}K -C_H, \ \mbox{and} \\ 
& \varphi_{x_i x_i}(\bar{x})=0.
\end{align*}
On the other hand, we have
\begin{eqnarray*}
|\I_i[B_\varrho^c(\bar x)] u^\epsilon (\bar x )|
&\leq&
\mathrm{osc}_\Gamma \{ u^{\epsilon, \eta}\} \sum_{j=1}^N \int_{J_j \cap \{ \rho(\gamma_j(z), \bar x) \geq \varrho\}} \nu_{ij}^\eta(\bar x, \rho(\gamma_j(z),\bar x)) dz\\
&\leq&
\Lambda \mathrm{osc}_\Gamma \{ u^{\epsilon, \eta}\} \sum_{j=1}^N \int_{J_j \cap \{ \rho(\gamma_j(z), \bar x) \geq \varrho\}} \rho(\gamma_j(z), \bar x)^{-1-\sigma}dz \\
&\leq& 2 \mathrm{osc}_\Gamma \{ u^{\epsilon, \eta}\} N\Lambda \sigma^{-1} \varrho^{-\sigma},
\end{eqnarray*}
and using that $\varphi$ is $K$-Lipschitz, we have
\begin{eqnarray*}
|\I_{i}[B_\varrho(\bar x)] \varphi(\bar x)|
&\leq&
\sum_{j=1}^N \int_{J_j \cap \{ \rho(\gamma_j(z), \bar x) < \varrho \}} |\varphi_j(z)  - \varphi(\bar x)| \nu_{ij}^\eta(\bar x, \rho(\gamma_j(z),\bar x)) dz\\
&\leq & K \Lambda \sum_{j=1}^N \int_{J_j \cap \{ \rho(\gamma_j(z), \bar x) < \varrho \}} \rho(\gamma_j(z), \bar x)^{-\sigma}dz \\
&\leq& \frac{NK\Lambda \varrho^{1- \sigma}}{1- \sigma}.
\end{eqnarray*}

Plugging these estimates into the viscosity inequality~\eqref{viscGeps}, we arrive at
\begin{eqnarray*}
  C_H^{-1}K \leq  2\mathrm{osc}_\Gamma \{ u^{\epsilon, \eta}\}  N\Lambda \sigma^{-1} \varrho^{-\sigma} + \frac{NK\Lambda \varrho^{1- \sigma}}{1- \sigma} + C_H + \lambda C_0.
\end{eqnarray*}
At this point, we fix $\varrho$ small enough in order that $\frac{N\Lambda \varrho^{1- \sigma}}{1- \sigma} \leq \frac{C_H^{-1}}{2}$, from which we get
$$
\frac{C_H^{-1}K}{2} \leq  2\mathrm{osc}_\Gamma \{ u^{\epsilon, \eta} \} N\Lambda \sigma^{-1} \varrho^{-\sigma} + C_H + \lambda C_0. 
$$
Therefore, choosing $K\geq K_0$, with
$$
K_0 := 2C_H(2 \mathrm{osc}_\Gamma \{ u^{\epsilon, \eta}\} N\Lambda \sigma^{-1} \varrho^{-\sigma} + C_H + \lambda C_0) + 1,
$$
we reach a contradiction with~\eqref{viscGeps}.

\medskip

\noindent{\it Step 2. Interior Lipschitz estimates.}
Let  $1\leq i \leq N$, $r > 0$ and $x_0 \in E_i$ such that $(x_0)_i \in [r, a_i - r]$. We consider
\begin{equation}\label{choixKdelta}
K > \max \{\frac{4C_0}{r} , K_0\},
\end{equation} 
with $K_0$ as in Step 1.

The supremum in~\eqref{sup683} is nonnegative and it is achieved at some $\bar x \in \Gamma$
such that
\begin{eqnarray}\label{estim-pt-max}
&&  K \rho(\bar{x}, x_0) \leq  u^{\epsilon, \eta} (\bar{x}) - u^{\epsilon, \eta} (x_0) \leq 2 C_0,
\end{eqnarray}
from which, by the choice of $K$, we conclude that $\bar x_i \in J_i \cap [r/2, a_i - r/2]$.
Thanks to the claim, we get $\bar{x}=x_0$
and therefore the supremum is 0. Since $x_0$ is arbitrary in the set of points at distance at least $r$ to the vertices, we conclude the interior Lipschitz estimates~\eqref{inLip}.

\medskip

\noindent{\it Step 3. Lipschitz estimates near the vertices and global Lipschitz estimates.} These estimates rely on the barriers and will concentrate on $O$, the estimates near other vertices being similar.

Consider the constants $L,\delta, L_\epsilon, \delta_\epsilon$ appearing in~\eqref{barriereOabov}-\eqref{barriereaiabov}-\eqref{barriereObel}-\eqref{barriereaibel}
and now choose 
$$
K> \max \{L, L_\epsilon, \frac{4C_0}{\delta_\epsilon}, K_0\} 
$$
and $x_0\in\Gamma \cap B_{\delta_\epsilon/2}(O)$
in~\eqref{sup683}.

From~\eqref{estim-pt-max} and the choice of $K$, we obtain that the maximum in~\eqref{sup683} is achieved at $\bar{x}\in B_{\delta_\epsilon}(O)$. 
From the claim in Step 1, we have that either $\bar{x}=x_0$ and in this case we are done since the maximum is 0, or $\bar{x}=O$. In this latter case,
using the barriers~\eqref{barriereOabov} and~\eqref{barriereObel}, it follows
that 
\begin{eqnarray*}
&& u^{\epsilon, \eta}(O) - u^{\epsilon, \eta}(x_0) - K \rho(O, x_0)
\leq 
\left(\max \{L ,L_\epsilon\} -K\right) \rho(O, x_0)\leq 0,
\end{eqnarray*}
and the maximum is also 0.
This proves the Lipschitz estimates near the vertices.

Putting the previous results together,
we conclude for the global Lipschitz bound~\eqref{uepslip}
(depending on $\epsilon$ since the barriers (ii) depend
on $\epsilon$, but not depending on $\eta$ since $\nu_{ij}^\eta$ satisfies~\eqref{hyp-nu} uniformly wrt $\eta$).

Let us mention that the proofs of the Lipschitz estimates work readily when $u$ is a $USC$ viscosity subsolution
of~\eqref{eqaprox0}.
In fact, Steps~1 and~2 do not rely on the ellipticity constant $\epsilon$ in~\eqref{eqaprox0},
so we can take $\epsilon=0$ in~\eqref{eqaprox0}, or consider~\eqref{eq}-\eqref{Kirchhoff}-\eqref{dirichlet-bc}, see Lemma~\ref{sous-sol-lip} below.
Step~3 needs the whole set of 
barriers~\eqref{barriereOabov}-\eqref{barriereaiabov}-\eqref{barriereObel}-\eqref{barriereaibel}, and the last two may be lost when the ellipticity vanishes.

\medskip

\noindent
(iv) {\it $C^{2,1-\sigma}$ regularity.} The function $u_i^{\epsilon,\eta}$ is a continuous viscosity solution of
\begin{eqnarray}\label{eq-ui}
&&- \epsilon u_{x_i x_i}^{\epsilon,\eta} = f := -\lambda u_i^{\epsilon,\eta} + \I_i^\eta u_i^{\epsilon,\eta} - H_i(x, u^{\epsilon,\eta}_{x_i}) \text{ in $(0,a_i)$.}
\end{eqnarray}
Since $u_i^{\epsilon,\eta}$ is $C_\epsilon$-Lipschitz continuous on $[0,a_i]$, 
thanks to Lemma~\ref{regI}, we obtain that $f\in L^\infty ([0,a_i])$
and the $L^\infty$ bound depends on $C_0$ and the Lipschitz
constant $C_\epsilon$ of $u_i^{\epsilon,\eta}$.
Therefore, by Han-Lin~\cite[Theorem 5.22]{hl97}, we obtain that
$u_i^{\epsilon,\eta}\in W_{\rm loc}^{2,p}([0,a_i])$ for all $1<p<\infty$. 
By Sobolev inequalities (see e.g., Evans~\cite[Theorem 6, p.270]{evans98}),
it follows that $u_i^{\epsilon,\eta}\in C^{1,\gamma}([0,a_i])$ for all $\gamma <1$.
It means that one can upgrade the regularity of $f$ in~\eqref{eq-ui}:
from Lemma~\ref{regI} once again and
since $H_i$ is Lipschitz continuous,
we get that $f\in C^{0,1-\sigma}([0,a_i])$. Thus
$u_{x_i x_i}^{\epsilon,\eta}\in C^{0,1-\sigma}([0,a_i])$ and
we conclude $u_{i}^{\epsilon,\eta}\in C^{2,1-\sigma}([0,a_i])$
as desired.
Note that $|| u_i^{\epsilon,\eta}||_{C^{2,1-\sigma}(\bar J_i)}$ depends on the $L^\infty$ bound~\eqref{bound_uieps}
and the Lipschitz
constant $C_\epsilon$ of $u_i^{\epsilon,\eta}$ (see~\eqref{uepslip}).

All the above estimates are independent of $\eta$, from which we can pass to the limit as $\eta \to 0$ in problem~\eqref{eqaprox} and conclude the existence of a solution $u^\epsilon$ to~\eqref{eqaprox0} by stability. Reproducing the arguments above, we obtain that $u_i^\epsilon\in C^{2,1-\sigma}([0,a_i])$.
\medskip


\noindent
(v) {\it Continuous dependence with respect to $\theta$.}
Let $\theta_n \to \theta$ and denote $u_i^{\epsilon, n}$ the associated solution of~\eqref{eqaprox0}
in $\bar J_i$. By~(ii), the family $\{u_i^{\epsilon, n}\}_n$ is uniformly bounded as $n \to \infty$. Since $u_i^{\epsilon,n}\in C^{2,1-\sigma}(\bar J_i)$ with a
$C^{2,1-\sigma}$ norm depending only on $|\theta_n|$ through the $L^\infty$ bound
and the Lipschitz constant of $u_i^\epsilon$, it follows that
$||u^{\epsilon,n}||_{C^{2,1-\sigma}(\bar J_i)}$ is bounded with respect to $n$.
By Ascoli Theorem, we can extract a subsequence still denoted by $\{u_i^{\epsilon, n}\}$,
such that $u_i^{\epsilon, n} \to \bar u_i^\epsilon$ in $C^{2, \alpha}(\bar J_i)$ for $0 < \alpha < 1-\sigma$,
and such a function $ \bar u_i^\epsilon$ must satisfy~\eqref{eqaprox0} with pointwise Dirichlet condition
thanks to (i)-(ii).
By uniqueness of the solution of~\eqref{eqaprox0}, 
all the subsequences converge to the same limit from which we infer that
the whole sequence $u_i^{\epsilon, n}$ converges to  the solution of~\eqref{eqaprox}
associated with $\theta$. This proves the result.
\end{proof}

The following is the existence result for the original Kirchhoff-type equation. 

\begin{teo}\label{teoexistence} Under assumption~\eqref{steady},
there exists a viscosity solution $u \in C(\Gamma)$ for the Kirchhoff problem~\eqref{eq}-\eqref{dirichlet-bc}-\eqref{Kirchhoff}.
Moreover, $u\in C^{0,\gamma}(\Gamma)$ for any $\gamma\in (0,1)$, and is locally Lipschitz continuous in $\Gamma \setminus \{ \ver_i \}_{1 \leq i \leq N}$. 
\end{teo}
  
\begin{proof}
We divide the proof in several steps.
\smallskip

\noindent {\it Step 1. Getting the Kirchhoff condition.}
For $\theta \in \R$, denote by $u^{\epsilon, \theta} \in C^{2,1-\sigma}(\Gamma)$ the unique solution of~\eqref{eqaprox0}
associated with $\theta$ given by Proposition~\ref{prop1}.
At this point we follow the arguments of Ohavi~\cite{ohavi21}.
Let
\begin{eqnarray*}
  && F(\theta)= \sum_{i=1}^N -u_{x_i}^{\epsilon, \theta}(O) -B.
\end{eqnarray*}
and $p = (p_i)_{1 \leq i \leq N}$ be any vector in $\R^N$ such that
\begin{equation*}
\sum_{i=1}^N -p_i = B.
\end{equation*}

Define $U \in C^2(\Gamma)$ by $U_i(x) := \theta + p_i x_i$ for $x \in \bar E_i$, $1\leq i\leq N$.
We have
\begin{eqnarray}\label{GiU}
&& G_i(U,U_{x_i}(x),x)= \lambda (\theta +p_ix)- \I_i U(x) + H_i(x,p_i),
\end{eqnarray}
with 
\begin{align*}
|\I_i U(x)| = & \Big| \int_{J_i} p_i (z - x_i) \nu_{ii}(x, \rho(x,\gamma_i(z)))dz \\
& + \sum_{j \neq i} \int_{J_j} (p_j z - p_i x_i) \nu_{ij}(x, \rho(x, \gamma_j(z)))dz  \Big|  \\
\leq & \Lambda |p| \Big{(} \int_{J_i} |z - x_i|^{-\sigma} dz + \sum_{j \neq i} \int_{J_j} (x_i + z)^{-\sigma} dz \Big{)}\\
\leq & \frac{(2\bar{a})^{1-\sigma} N \Lambda |p|}{1-\sigma},
\end{align*}
and $|H_i(x,p_i)|\leq C_H(1+|p_i|)\leq C_H(1+|p|)$. Here we have adopted the notation $|p|$ for the Euclidean norm of $p \in \R^N$
and $\bar{a}=\max_i a_i$.

We claim that, if
$\theta=\theta^+$ is large enough, then $U$ is a supersolution of~\eqref{eqaprox0}.
Indeed, from~\eqref{GiU},  for $x\in E_i$, we have
\begin{eqnarray*}
  && G_i(U,U_{x_i}(x),x)\geq  \lambda (\theta^+ -\bar{a}|p|)
  - \frac{(2\bar{a})^{1-\sigma} N \Lambda |p|}{1-\sigma} -  C_H(1+|p|) \geq 0,
\end{eqnarray*}
provided $\theta^+\geq \bar{a}|p|+ \lambda^{-1}( (2\bar{a})^{1-\sigma}(1-\sigma)^{-1}\Lambda N |p| +  C_H(1+|p|)$.
If, in addition,  $\theta^+\geq \bar{a}|p|+\max_i|h_i|$, then
$U(O)=\theta^+$ and $U_i(a_i)=\theta^+ + p_i a_i \geq h_i$. Therefore
the boundary conditions at the vertices are satisfied.
Notice that $\theta^+$ does not depend on $\epsilon$.


By Lemma~\ref{comp-censA}, we can compare the solution $u^{\epsilon, \theta^+}$ to~\eqref{eqaprox0} associated with $\theta^+$
with $U$ to obtain
that, for all $1\leq i \leq N$ and $x\in \bar J_i$,
\begin{eqnarray*}
&& u_i^{\epsilon, \theta^+}(x) \leq U_i(x)=\theta^+ +p_i x =  u_i^{\epsilon, \theta^+}(O)+p_i x,
\end{eqnarray*}
from which we deduce
\begin{eqnarray*}
 u_{x_i}^{\epsilon, \theta^+}(O)\leq p_i.
\end{eqnarray*}
Then
\begin{equation*}
F(\theta^+)=  \sum_{i=1}^N -u_{x_i}^{\epsilon, \theta^+}(O) -B\geq \sum_{i=1}^N -p_i -B=0.
\end{equation*}

In the same way, if we consider $\theta=\theta^-$ adequate (say, negative and large in absolute value), then $U$ is a subsolution of~\eqref{eqaprox0}.
Similarly, we obtain  that the solution
$u^{\epsilon, \theta^-}$ to problem~\eqref{eqaprox0} satisfies
$$
F(\theta^-)=  \sum_{i=1}^N -u_{x_i}^{\epsilon, \theta^-}(O) -B\leq \sum_{i=1}^N -p_i -B=0.
$$

Then, by continuity of $\theta\mapsto F(\theta)$
(Proposition~\ref{prop1} (v)),
there exists $\theta_\epsilon\in [\theta^-,\theta^+]$ such that the associated solution $u^{\epsilon, \theta_\epsilon}$
to problem~\eqref{eqaprox0} satisfies
\begin{equation}\label{Kir}
F( \theta_\epsilon)= 0,
\end{equation}
which is the Kirchhoff condition~\eqref{Kirchhoff}.
\medskip

\noindent{\it Step 2. The family $\{ u^{\epsilon, \theta_\epsilon} \}_\epsilon$ is uniformly Lipschitz continuous on a neighborhood of $O$.}
We will take profit of the Kirchhoff condition.
At first, since $\theta^\pm$ do not depend on $\epsilon$, the value 
$\theta_\epsilon$ is bounded from
above and below uniformly with respect to $\epsilon$. Thus, it follows that $u^\epsilon:= u^{\epsilon, \theta_\epsilon}$ is uniformly bounded with respect to $\epsilon$, with $L^\infty$-bounds given by $C_0$ in~\eqref{defC0} with $|\theta|=\max\{|\theta^+|,|\theta^-|\}$.

Let $\delta < \min_{1 \leq i \leq N} |a_i|$ and consider 
\begin{eqnarray}\label{gamma-d}
&& \Gamma^\delta := \{ x \in \Gamma : \rho(x, \ver_i) > \delta, 1\leq i\leq N\} 
= \Gamma\setminus \bigcup_{1\leq i\leq N} B(\ver_i,\delta).
\end{eqnarray}

Our goal is to prove that $u^\epsilon$ is Lipschitz continuous in $\Gamma^\delta$
uniformly with respect to $\epsilon$.

We proceed as in the proof of Proposition~\ref{prop1} (iii) borrowing some
arguments of~\cite[Proposition 6.3]{bbci18}.
Let $x_0\in \Gamma^\delta\cap \bar{E}_i$ and consider
the maximum as in~\eqref{sup683}, that is
\begin{eqnarray}\label{sup365}
&& \sup_{x\in \Gamma} \{ u^\epsilon (x) - u^\epsilon (x_0)- \varphi(x) \}
\end{eqnarray}
but this time with the modified function $\varphi =\varphi^i$ defined by
$$
\varphi^i (x) = \left \{ \begin{array}{ll} K\rho(x, x_0)& \text{if $x\in E_i$}\\
K \rho(O, x_0)+L\rho (x,O) & 
\text{if $x\in E_j$, $j\neq i$},
\end{array} \right .
$$
for $K,L\geq K_0$, where $K_0$ is defined
in the proof of Proposition~\ref{prop1}\,(iii) and does not depend on $\epsilon$.
With a straightforward adaptation of the proof of Proposition~\ref{prop1}\,(iii), we obtain that~\eqref{sup365} is achieved at $\bar{x}=x_0$ or at a vertex.
On the one hand, if $\bar{x}=x_0$, then the maximum~\eqref{sup365} is 0, from which we get the Lipschitz estimates on $\Gamma^\delta$.

On the other hand,
up to enlarge $K,L$ in order that $K,L\geq 4\| u^\epsilon\|_\infty/\delta$, 
we have $\rho(\bar{x},x_0)\leq \delta/2$
and therefore $\bar{x}$
cannot be equal to any exterior vertex 
$\ver_i, \ 1 \leq i \leq N$.

It remains to prove that the case $\bar{x}=O\neq x_0$ is not possible either.
If $\bar{x}=O$, then, using that $u^\epsilon$ is a classical (hence, viscosity) solution to the Kirchhoff problem, we have
\begin{equation*}
\sum_{1 \leq j \leq N} -\varphi_{x_j}^i(O) \leq B.
\end{equation*}

A direct computation shows that
$$
\varphi_{x_i}^i(O) = -K, \quad \varphi_{x_j}^i(O) = L, \ \mbox{for all} \ j \neq i,
$$
from which the Kirchhoff condition reads 
\begin{equation*}
K - (N - 1)L \leq B.
\end{equation*}
and from here, fixed $L$ as above and enlarging $K$, we get a contradiction. 

Finally, taking $K,L\geq \max \{K_0, 4\|u^\epsilon\|_\infty / \delta\}$ and $K >(N-1)L + B$, the maximum~\eqref{sup365} is 0.
Since, this is true for all $x_0\in \Gamma^\delta$, we obtain that $u^\epsilon$ is Lipschitz continuous
in $\Gamma^\delta$,
with a constant 
$\max \{K,L\}$ depending on $\delta$ but independent of $\epsilon$.

\medskip

\noindent{\it Step 3. Sending $\epsilon$ to 0.}
In the previous steps, we have obtained that 
$u^\epsilon:= u^{\epsilon, \theta_\epsilon}$ is bounded independently
of $\epsilon$ on $\Gamma$, and uniformly
Lipschitz continuous in $\Gamma^\delta$ for every $\delta >0$ small enough. Denote $\mathring \Gamma = \Gamma \setminus \{ \ver_i \}_{1 \leq i \leq N}$.
By Ascoli Theorem and a diagonal process, we
can extract a subsequence $\{ u^\epsilon \}_\epsilon$ (not relabeled) which converges locally uniformly
to some $u\in C(\mathring \Gamma)$, which is locally Lipchitz continuous in $\mathring \Gamma$
(notice that  we loose both the Lispchitz continuity and the barriers 
on $\{ \ver_i \}_{1 \leq i \leq N}$ when $\epsilon\to 0$).
Thanks to the coercivity assumption~\eqref{H}\,(iii), from~\cite[Theorem 2.1]{bklt15},
we have $u \in C^{0,\gamma}(\mathring \Gamma)$ for any $\gamma\in (0,1)$, with finite H\"older seminorm in $\mathring \Gamma$
(\cite[Theorem 2.1]{bklt15} is applied with $m=1$ and $\theta=0$, yielding $\gamma_0=1$).
It follows that $u$ can be extended by continuity up to the ends of the edges
into a function (still called $u$) in $C^{0,\gamma}(\Gamma)$.
By standard stability for viscosity solutions, $u$ is a viscosity solution of~\eqref{eq} in each edge $E_i$.
It remains to establish the boundary conditions.

We first prove that $u$ satisfies the generalized Kirchhoff condition at $O$.
We only check the subsolution property since the proof for the supersolution is similar.
If $\phi \in C^2(\Gamma)$ is such that $u - \phi$ has a strict maximum point at $O$, by uniform convergence
of $u^\epsilon$ in a neighborhood of $O$, there exists a sequence $x^\epsilon \to O$ of local maximum points of $u^\epsilon - \phi$. If $x^\epsilon = O$ along a subsequence $\epsilon\to 0$, we use the classical Kirchhoff
condition for $u^\epsilon$ to conclude that
$$
B = - \sum_{1\leq i \leq N} u_{x_i}^\epsilon(O) \geq - \sum_{1\leq i \leq N} \phi_{x_i}(O),
$$
from which the Kirchhoff condition is satisfied by $u$. On the other hand, if $x^\epsilon \in \Gamma\setminus \{0\}$ for all $\epsilon$ small,  then we can assume that there exists $i$ such that $x^\epsilon \in E_i$ for all $\epsilon$ (up to subsequences).
It follows $G_i^\delta (u^\epsilon, \phi , \phi_{x_i}(x_\epsilon), x_\epsilon)\leq 0$ and,
since $\phi_i \in C^1(\bar J_i)$ and $u^\epsilon\to u$ uniformly as $\epsilon\to 0$,
we can pass to the limit in the inequality to get
$G_i^\delta (u, \phi , \phi_{x_i}(O), O)\leq 0$ and conclude that
the generalized Kirchhoff condition holds.

We turn to the Dirichlet condition at $\ver_i$.
Since the barrier~\eqref{barriereaiabov} from above is independent of $\epsilon$,
we can pass to the limit $\epsilon\to 0$ in the equality $u_i^\epsilon (x)\leq h_i + L(a_i-x)$
to obtain $u_i (x)\leq h_i + L(a_i-x)$ for $0\leq x < a_i$. Then, sending $x\to a_i^-$
and recalling that $u$ is extended by continuity at $a_i$, we get
$u_i(a_i)\leq h_i$. Therefore, the boundary condition for subsolutions is satisfied pointwisely.

For the supersolution property at  $\ver_i$, the barrier from below  depend
on $\epsilon$ and we cannot conclude in the same way.
Consider $\phi \in C^2(\Gamma)$ such that $u - \phi$ has a strict minimum point at $\ver_i$.
For every $\epsilon >0$, let $x_\epsilon$ be a minimum of $u^\epsilon - \phi$ on $\Gamma$, which by standard arguments
is in $\bar E_i$, i.e.,
\begin{eqnarray}\label{min852}
&&  u_i^\epsilon (x) - \phi_i(x) \geq u_i^\epsilon (x_\epsilon)- \phi(x_\epsilon), \quad x\in \bar J_i.
\end{eqnarray}

By~\eqref{bound_uieps}, we can extract a subsequence (still denoted $\epsilon$) such that
$x_\epsilon\to x_0\in \bar J_i$ and $u_i^\epsilon (x_\epsilon)\to \ell \in [-C_0,C_0]$.
If $x_0\not= a_i$, then, thanks to the local uniform convergence in $[0,a_i)$,
passing to the limit in~\eqref{min852}, we obtain
$u_i (x) - \phi_i(x) \geq u_i(x_0)- \phi(x_0)$ for all $x\in \bar J_i$, which
is a contradiction with the strict minimum at $a_i$. Therefore $x_\epsilon\to a_i$.
Sending $\epsilon\to 0$ in~\eqref{min852} for $0\leq x<a_i$, we have
$u_i(x)-\phi_i(x)\geq \ell - \phi_i(a_i)$. Then, letting $x\to a_i^-$, we obtain 
\begin{eqnarray}\label{limueps123}
u_i(a_i)\geq \lim_{\epsilon\to 0}  u_i^\epsilon(x_\epsilon) = \ell.
\end{eqnarray}
It remains to consider the case when $x_\epsilon\to x_0=a_i$.
If $x_\epsilon =a_i$ along a subsequence, then, using that
$u^\epsilon$ satisfies the Dirichlet condition pointwisely at $\ver_i$,
yields at the limit $u_i(a_i)\geq \ell\geq h_i$. In this case
$u$ satisfies also the condition pointwisely.
Now we deal with the case when $x_\epsilon < a_i$ for all $\epsilon$.
Writing the viscosity supersolution inequality for
$u^\epsilon$, we get
\begin{eqnarray*}
&& G_i^\delta (u^\epsilon, \phi , \phi_{x_i}(x_\epsilon), x_\epsilon)\\
&=& \lambda u_i^\epsilon(x_\epsilon) - \I_i  [B_\delta^c (x_\epsilon)] u^\epsilon(x_\epsilon) - \I_i  [B_\delta (x_\epsilon)] \phi(x_\epsilon) + H_i(x_\epsilon, \phi_{x_i}(x_\epsilon))  
  \geq 0.
\end{eqnarray*}
We want to pass to the limit  $\epsilon\to 0$ in the above inequality.
On the one hand, by Lebesgue Theorem, using~\eqref{hyp-nu} and~\eqref{limueps123}, we have
\begin{eqnarray*}
&& \lim_{\epsilon\to 0}  \lambda u_i^\epsilon(x_\epsilon)-\I_i  [B_\delta^c (x_\epsilon)] u^\epsilon(x_\epsilon)\\
&=& \lambda \ell -\sum_{1\leq j\leq N}\int_{\{z_j\in J_j:  \rho(\ver_i,\gamma_j(z))\geq \delta \}} [u_j(z)-\ell]\nu_{ij}(\ver_i, \rho(\ver_i,\gamma_j(z)))dz \\
  &\leq &  \lambda u_i(a_i) -\sum_{1\leq j\leq N}\int_{\{z_j\in J_j:  \rho(\ver_i,\gamma_j(z))\geq \delta \}}
  \!\!\!\!\!\!\!\!\!\!\! [u_j(z)-u_i(a_i)]\nu_{ij}(\ver_i, \rho(\ver_i,\gamma_j(z)))dz\\
  &=&  \lambda u_i(a_i) -\I_i  [B_\delta^c (\ver_i)] u(\ver_i).
\end{eqnarray*}

On the other hand, 
thanks to the smoothness of $\phi$, the convergence of the other terms is obvious,
and we arrive finally at 
$G_i^\delta (u, \phi, \phi_{x_i}( \ver_i), \ver_i)\geq 0$,
which proves that the generalized Dirichlet condition holds at $\ver_i$.

Hence, we have concluded the existence of $u \in C^{0,\gamma}(\Gamma)$ for any $\gamma\in (0,1)$, solution of the Kirchhoff problem~\eqref{eq}-\eqref{Kirchhoff}-\eqref{dirichlet-bc}.
\end{proof}

\begin{rema}\label{loss-en-dessous}
$(i)$ Using the same global sub/supersolutions for the Kirchhoff-Dirichclet problem in Step 1 in the proof above, we can complement the standard Perron's method with the strong comparison principle given by Theorem~\ref{teo1} below, and directly obtain the existence and uniqueness of the generalized Kirchhoff problem~\eqref{eq}-\eqref{Kirchhoff}-\eqref{dirichlet-bc} at once.

\smallskip

\noindent
$(ii)$ Since all the subsolutions we use to have the estimates from below, namely~\eqref{barriereObel}-\eqref{barriereaibel},
depend on $\epsilon$, we may have
a loss of boundary conditions at the vertices when passing to the limit $\epsilon\to 0$.
This comes from the fact that there is no enough ellipticity in the nonlocal term
to counterbalance the coercivity of the Hamiltonian in absence of classical diffusion.
At the junction $O$, the
the Kirchhoff condition allows to recover some control from below
and to prove that the solution $u^\epsilon$ is
Lipschitz continuous uniformly in $\epsilon$, providing a continuous limit
$u$ at $O$. In contrast,
at the boundary point, the limit $u$ may be discontinuous.
But thanks to H\"older regularity results for coercive
equations, we can extend it by continuity up to the boundary into
a function satisfying the generalized Dirichlet boundary conditions
(see~\cite{bdl04} for related discussions).
\end{rema}


\section{Uniqueness for the Kirchhoff-type problem on a junction.}
\label{secuniqueness}

The comparison result is the following

\begin{teo}\label{teo1}
Assume~\eqref{steady}. Let $u\in USC(\Gamma)$ be a viscosity subsolution
and $v\in LSC(\Gamma)$ be a viscosity supersolution
to the problem~\eqref{eq}-\eqref{Kirchhoff}-\eqref{dirichlet-bc}.
Then, $u \leq v$ on $\Gamma$.

In particular, there exists a unique continuous viscosity solution $u\in C^{0,1}(\Gamma_\delta)$
to the problem, which coincides with
the solution found in Theorem~\ref{teoexistence}
(recall that $\Gamma^\delta$ is defined by~\eqref{gamma-d}).
\end{teo}

\begin{lema}\label{sous-sol-lip}
  Let $u\in USC(\Gamma)$ be a bounded viscosity subsolution of~\eqref{eq}-\eqref{Kirchhoff}.
  Then $u$ is Lipschitz continuous on $\Gamma\delta$ for some $\delta > 0$.
\end{lema}

\begin{proof}
We follow the proof of the Lipschitz estimates
in Proposition~\ref{prop1}\,(iii). Steps~1 and~2 work readily
for any $USC$ viscosity subsolution of~\eqref{eq}-\eqref{Kirchhoff}-\eqref{dirichlet-bc};
they rely only on the coercivity of the Hamiltonian
and do not depend on the ellipticity
constant in the equation.

Nevertheless, we cannot follow the proof of Step~3 since the barriers from below are not true
anymore due to the degeneracy of the equation.
To recover the Lipschitz estimate at the junction point,
we revisit the proof of the uniform
Lipschitz estimates near the junction in Theorem~\ref{teoexistence}.
It is enough to prove that the maximum in~\eqref{sup365}, with $u^\epsilon=u$ subsolution to~\eqref{eq}-\eqref{Kirchhoff}
and $x_0\in \Gamma_\delta\cap E_i$ for some $i$, cannot be achieved
at $\bar x= O$ for $K,L$ large enough.

Indeed, we first notice, as a direct consequence of Step~1 in
the proof of Proposition~\ref{prop1}\,(iii),
and the choice of $K\geq K_0$,
that the PDE~\eqref{eq} cannot hold at $\bar x=O$ for subsolutions, i.e.,
$G_i^\delta (u,\varphi, \varphi_{x_i}(0), O) >0$ for any $1\leq i\leq N$.
We then prove that the Kirchhoff condition at $O$ does not
hold neither for large $K,L$
exactly as in the proof of Theorem~\ref{teoexistence}, Step 2. Thus, the maximum in~\eqref{sup365} holds in the interior of the branch, from which the result holds by the choice of $K,L$. 
\end{proof}  


In order to prove Theorem~\ref{teo1}, we first have to examine the viscosity inequalities which are
satisfied at the junction on each branch in terms of sub- and superdifferential of the solutions (see~\cite{ls17, bc24}).   

To do so, for each $i$, we define
\begin{eqnarray*}
&& \bar p_i = \limsup_{x_i \to 0^+} \frac{u_i(x_i) - u_i(0)}{x_i}, \qquad \underline p_i = \liminf_{x_i \to 0^+} \frac{u_i(x_i) - u_i(0)}{x_i}, \\
&& \bar q_i = \limsup_{x_i \to 0^+} \frac{v_i(x_i) - v_i(0)}{x_i}, \qquad \underline q_i = \liminf_{x_i \to 0^+} \frac{v_i(x_i) - v_i(0)}{x_i}.
\end{eqnarray*}

Notice that, if $D^+ u_i(0) \neq \emptyset$, then $\bar p_i < +\infty$, meanwhile, if $D^- v_i(0) \neq \emptyset$, then
$\underline q_i > -\infty$. And both cases happen if either $D^+_\Gamma u(O)\not= \emptyset$ or $D^-_\Gamma v(O)\not=~\emptyset$.

\begin{lema}\label{inegGi}
Let $u\in USC(\Gamma)$ (respectively  $v\in LSC(\Gamma)$)
be a bounded visco\-sity subsolution (respectively supersolution)
of~\eqref{eq}-\eqref{Kirchhoff} such that  $D^+_\Gamma u(O)\not= \emptyset$
(respectively $D^-_\Gamma u(O)\not= \emptyset$).
Then

\begin{itemize}
\item[(i)]
If $\bar p_i$ is finite, then $D^+ u_i(0)=[\bar p_i ,+\infty)$, $\I_i u(O) \in \R$ and
\begin{eqnarray*}
&& G_i (u,\bar p_i,O)\leq 0.
\end{eqnarray*}
Similarly, if  
$\underline q_i$ is finite, then $D^- v_i(0)=(-\infty, \underline q_i]$, $\I_i v(O) \in \R$ and
\begin{eqnarray*}
&& G_i (v,\underline q_i,O)\geq 0.
\end{eqnarray*}

\item[(ii)]
If, in addition, $u$ is $C^{0,\sigma + \varepsilon}$ in an $\Gamma$-neighborhood of $O$ for some $\varepsilon> 0$, then we have
\begin{eqnarray}\label{inegGip}
&& G_i (u,p,O)\leq 0 \qquad \mbox{for all} \ p \in (\underline p_i, \bar p_i),
\end{eqnarray}
and the inequality holds for $p = \underline p_i$ and/or $p = \bar p_i$ if some is finite.
Similarly, if $v$ is $C^{0,\sigma + \varepsilon}$ in an $\Gamma$-neighborhood of $O$ for some $\varepsilon> 0$, then we have
\begin{eqnarray*}
&& G_i (v,q,O)\geq 0 \qquad \mbox{for all} \ q \in (\underline q_i, \bar q_i), 
\end{eqnarray*}
and the inequality holds for $q = \underline q_i$ and/or $q = \bar q_i$ if some is finite.

\end{itemize}
\end{lema}

Part~(i) of Lemma~\ref{inegGi} is a basic result: it is the nonlocal analogue to \cite[Proposition~2.10, p.~84]{bc24} and
it does not require any additional regularity. On the contrary, Part~(ii) is a more sophisticated result inspired from Lions and 
Souganidis~\cite{ls16,ls17} (see also \cite[Lemma 15.1, p.258]{bc24}) and it is also surprising because none of the
$p \in (\underline p_i, \bar p_i)$ is in $D^+ u_i(0)$ and none of the $q \in (\underline q_i, \bar q_i)$ is in $D^- v_i(0)$.
Unfortunately, we were unable obtain this stronger result without the additional regularity we impose; of course, the difficulty was
to take into account the nonlocal term. This additional assumption is automatically satisfied by subsolutions since they
are Lipschitz continuous by Lemma~\ref{sous-sol-lip} but it is not the case for supersolutions.  For this reason, in the comparison proof, we can use Result~(ii) for subsolutions but only Result~(i) for supersolutions.


\begin{proof} We first prove (i). We provide the proof for supersolutions, the one for subsolutions follows from the same arguments.

Since $D_\Gamma^- v(O) \neq \emptyset$, there exists a function $\varphi \in C^1(\Gamma)$ such that
$v- \varphi$ has a minimum at $O$. It is possible to choose $\varphi$ is order that the minimum is global over
$\Gamma$ and strict at $O$. Using that $\underline q_i$ is finite, we have that $D^- v_i(0) = (-\infty, \underline q_i]$
(see~\cite[Proposition 2.10, p.~84]{bc24}), and, up to a modification of $\varphi$ in $E_i$, we can take $\varphi$ such that $\varphi_{x_i}(O) = \underline q_i$.

Now, for $\epsilon > 0$, we consider the perturbed test-function $\varphi^\varepsilon$ given on $E_j$ ($j=1,...,N$), by
$$
\varphi^\varepsilon_j(x_j) = \varphi_j (x_j)+ \delta_{ij} \epsilon x_j, 
$$
where $\delta_{ij}$ is the usual Kronecker symbol. 

We claim that there exists a sequence $x^\varepsilon \in E_i$ of global minimum points for $v - \varphi^\varepsilon$, and
this sequence tends to $O$. The existence of a sequence  of global minimum points $x^\varepsilon \in \Gamma$, which tends to
$O$, is given by standard arguments since $O$ is a global, strict minimum point of $v-\varphi$ on $\Gamma$. 

It remains to prove that $x^\varepsilon \in E_i$. In fact, if $x^\varepsilon \in E_j$ with $j \neq i$, then
we have $v(x^\varepsilon) - \varphi^\varepsilon(x^\varepsilon) = v(x^\varepsilon) - \varphi(x^\varepsilon) >  v(O) - \varphi(O)
=v(O) - \varphi^\varepsilon(O)$ by the strict minimum property.
Thus, $x^\varepsilon \in \bar E_i$. If $x^\varepsilon = O$ along a subsequence, for each $x \in E_i$ we have
$$
v_i (x_i) - \varphi^\varepsilon_i(x_i)=v_i(x_i) - \varphi_i (x_i) - \epsilon x_i \geq v_i(0) - \varphi_i(0),
$$
from which we deduce
$$
v_i(x_i) \geq v_i(0) + (\varphi_i(x_i) + \epsilon x_i - \varphi_i(0)) \geq v_i(0) + (\underline q_i + \epsilon) x_i + o(x_i),
$$
and $\underline q_i + \epsilon \in D^- v_i(0)$, which contradicts the maximality of $\underline q_i$ in the subdifferential of $v_i$ at $0$. It ends the proof of the claim.

It follows that we can write the viscosity inequality for the supersolution 
at $x^\varepsilon \in E_i$, and since the testing is global, for all $\delta > 0$, we have 
$$
G_i^\delta(v, \varphi^\varepsilon, \underline q_i + \epsilon, x^\varepsilon) \geq 0.
$$
Using that $v\in LSC(\Gamma)$, we have $v(x^\varepsilon) \to v(O)$ by standard arguments. Adding that $\varphi^\varepsilon$ is Lipschitz continuous with a constant bounded uniformly with respect to $\epsilon$, we can control the nonlocal terms $\I_i[B_\delta(x^\varepsilon)] \varphi^\varepsilon(x^\varepsilon)$ and $\I_i[B_\delta^c(x^\varepsilon)] v(x^\varepsilon)$ with Dominated Convergence Theorem, for all
fixed  $\delta > 0$. It allows to pass to the limit $x^\varepsilon \to 0$
in the previous inequality to obtain
\begin{equation*}
G_i^\delta(v, \varphi , \underline q_i, O) \geq 0.
\end{equation*}
By Lemma~\ref{lem-F}, $v\in \mathcal{F}_O$ and we can send $\delta\downarrow 0$ to conclude that $G_i(v, \underline q_i, O) \geq 0$ holds. As a by-product, we have that $\mathcal{I}_i v(O)$ is finite.\\

Now we turn to the proof of (ii) which is strongly inspired by~\cite[Lemma 15.1, p.258]{bc24}.
We concentrate on the subsolution's case, the proof for supersolution being similar.
Since $D^+_\Gamma u(O)\not= \emptyset$, we know, by Lemma~\ref{lem-F}, that $u\in \mathcal{F}_O$
and we can find $\varphi\in C^1(\Gamma)$ such that $u-\varphi$ achieves a maximum at $O$. We can even assume that
$u(O)=\varphi (O)$.

We fix $1 \leq i \leq N$. We first assume that $\underline p_i < \bar p_i$ and we take any $p \in (\underline p_i, \bar p_i)$.
By definition of $\underline p_i, \bar p_i$, there exists sequences $0 < a_k < b_k$ converging to $0$ such that
\begin{equation}\label{ineg498}
u_i(b_k) - u_i(0) \leq (\underline p_i + k^{-1}) b_k , \quad u_i(a_k) - u_i(0) \geq (\bar p_i- k^{-1})a_k.
\end{equation}

For $k$ large enough, in order to have $ \underline p_i +k^{-1}<p< \bar p_i-k^{-1}$, we consider the function
$$
y_i \in [0,b_k] \mapsto \chi(y_i):=u_i(y_i) - u_i(0) - py_i.
$$
We have $\chi(0)=0$, $\chi(b_k)<0$ and $\chi(a_k)>0$ and therefore the function $\chi$ has a maximum points $y_k \in (0,b_k)$.
Of course, $y_k \to 0$ and, by classical arguments, $u_i(y_k)\to u_i(0)$. 

Modifying the function $\varphi$ by setting  $\varphi_i(y)=u_i(0)+ py$ on the branch $E_i$, we obtain that $u-\varphi$ has a local maximum inside the branch $E_i$ at $\gamma_i(y_k)$. Since $u$ is $(\sigma+\varepsilon)$-H\"older continuous, we use Lemma~\ref{regI} and Lemma~\ref{equiv-def} to arrive at
\begin{equation*}
G_i(u , p , \gamma_i(y_k)) \leq 0,
\end{equation*}
and since $G_i$ is continuous, using Lemma~\ref{regI} again we conclude that
$$
G_i(u, p, O) \leq 0.
$$

When $\underline p_i$ is finite (respectively $\bar p_i$ is finite) with $\underline p_i < \bar p_i$, the same inequality holds
for $p=\underline p_i$ (respectively $p=\bar p_i$) by continuity, letting $p$ tend to $\underline p_i$ (respectively to $\bar p_i$).

It remains to treat the case when both $\underline p_i$ and $\bar p_i$ are finite with
$p=\underline p_i=\bar p_i$. In this case, we perturb $u$ in the branch $E_i$
by setting 
$$
u^\epsilon_i(x)=u_i(x)+\epsilon x\sin(\log (x)), \ x \in J_i, \quad  u_i^\epsilon(0) = u_i(0),
$$ 
which
converges uniformly to $u_i$ in $\bar J_i$ as $\epsilon \to 0$, and satisfies
\begin{eqnarray*}
  \liminf_{x \to 0^+} \frac{u_i^\epsilon(x) - u_i^\epsilon(O)}{x}=: \underline{p}_i^\epsilon =  \underline p_i-\epsilon
  < \bar{p}_i +\epsilon = \bar{p}_i^\epsilon :=
\limsup_{x \to 0^+} \frac{u_i^\epsilon(x) - u_i^\epsilon(O)}{x}.
\end{eqnarray*}

Denote $\phi(x) = x\sin(\log (x))$. As in the first step above, since $\underline{p}_i^\epsilon <p< \bar{p}_i^\epsilon$, the function
$$
y \in \bar J_i \mapsto  u_i^\epsilon(y) - u_i^\epsilon(O) - py = u_i(y) - u_i(O) - py + \epsilon \phi(y)
$$
has a sequence of local maximum points $y_k \to 0^+$.
We modify the test function $\varphi$ in $E_i$ by setting
$\varphi_i (x)= p x + \epsilon \phi(x)$. Noticing that $\phi$ is Lipschitz continuous.
Writing the viscosity inequality for $u$
at $y_k$, and using the regularity of $u$, we obtain
$G_i (u, p+\epsilon \phi_x(y_k), \gamma_i(y_k))\leq 0$.
Using~\eqref{H}\,(ii),
it follows $G_i (u, p, \gamma(y_k))\leq O(\epsilon)$,
where $O(\epsilon)$ is independent of $k$. Sending $y_k\to 0$,
we infer $G_i (u , p, O)\leq O(\epsilon)$.
We then send $\epsilon\to 0$ and finally $\delta\downarrow 0$ to conclude. 
\end{proof}


\begin{proof}[Proof of Theorem~\ref{teo1}]
By contradiction, we assume that
$$
\sup_{\Gamma} \{ u - v \} > 0.
$$

Since $\Gamma$ is compact, the supremum is attained at some point $x_0 \in \Gamma$. 

If $x_0 \in \Gamma \setminus \VV$, where $\VV$ is the set of vertices,
then the proof follows standard arguments in the viscosity theory.
If $x_0 = \ver_i$ for some $i$,
then the proof follows the boundary analysis presented in~\cite{topp14, soner86a}, which is possible since $u \in C^\gamma(\Gamma)$ for all $\gamma \in (0,1)$.
So, we consider only the case $x_0 = O$, that is 
\begin{equation*}
\max_{\Gamma} \{ u - v \} = (u - v)(O)=: M > 0,
\end{equation*}
which concentrates the main difficulties.
We follow closely the lines of~\cite{ls17}, see also~\cite[Section 15.3]{bc24}.

At first, notice that $u$ is Lipschitz continuous on $\Gamma_\delta$ for some $\delta >0$
by Lemma~\ref{sous-sol-lip}. From Lemma~\ref{regI}, we obtain that $\I_iu(O)\in\R$
for all $i$. Next, since $u$ is touching $v$ from below at $O$, from Lemma~\ref{lem-F},
we obtain that $v\in \mathcal{F}_O$ and
\begin{equation}\label{comp-nl234}
-\infty < \I_i u(O)\leq \I_i v(O) \leq +\infty , \quad \text{for all $i$.}
\end{equation}

Notice that $u\not\in C^1(\Gamma_\delta)$ as required in the statement of Lemma~\ref{lem-F}
but $u\in C^{0,1}(\Gamma_\delta)$ is enough in the proof.
Using the Lipschitz continuity of $u$, we can apply Lemma~\ref{inegGi}~(ii) to $u$
and then Lemma~\ref{inegGi}~(i) to $v$. Since $O$ is a maximum point of $u-v$, we also have
\begin{equation*}
  -\infty < \underline p_i \leq \bar p_i <+\infty
  \quad \text{and} \quad
   \underline p_i \leq \underline q_i \leq +\infty \quad \text{for all $i$.}
\end{equation*}

We divide the analysis in two cases.

\smallskip

\noindent
\textsl{Case 1. There exists $i$ such that $\underline q_i \leq \bar p_i$.} We have $\underline q_i$ is finite and we can choose $\underline q_i$ in the viscosity inequality of $u$ and $v$, from which we get
\begin{align*}
G_i(u,\underline q_i,O)=\lambda u(O) - \sum_{1\leq j \leq N} \int_{J_j} [u_j(z)- u(O)] \nu_{i j}(O, z)dz + H_i(O, \underline q_i) & \leq 0, \\
G_i(v,\underline q_i,O)=\lambda v(O) - \sum_{1\leq j \leq N} \int_{J_j} [v_j(z)- v(O)] \nu_{i j}(O, z)dz + H_i(O, \underline q_i) & \geq 0,
\end{align*}
where both nonlocal terms are finite.

Subtracting both inequalities, we arrive at
\begin{equation*}
\lambda M \leq \sum_{1\leq j \leq N} \int_{J_j} [(u_j - v_j)(z) - M] \nu_{ij}(O, z) dz,
\end{equation*}
and since $O$ is the maximum point of $u - v$ on $\Gamma$, we arrive at 
$0 < \lambda M \leq 0$, which is a contradiction. This concludes this case.

\smallskip




\noindent
\textit{Case 2. For all $i$, $\bar p_i <\underline q_i \leq +\infty$.} In this case, for each $i$, we define
\begin{align*}
r \in [\bar p_i, \underline q_i] \mapsto U_i(r) := \lambda \frac{u(O) + v(O)}{2} - \I_i u(O) + H_i(O, r).
\end{align*}

By Lemma~\ref{inegGi}~(ii), we have 
\begin{equation*}
U_i(\bar p_i) = G_i(u, \bar p_i, O) - \lambda \frac{M}{2} < 0.
\end{equation*}
Meanwhile,  if $\underline q_i <+\infty$, then, by Lemma~\ref{inegGi}~(i), we have that $\I_i v(O)$ is finite, and
\begin{eqnarray*}
  && U_i(\underline q_i)=  G_i(u, \underline q_i, O) - \lambda \frac{M}{2}
  \geq  G_i(v, \underline q_i, O) + \lambda \frac{M}{2} >0.
\end{eqnarray*}
If $\underline q_i =+\infty$, then, thanks to the coercivity of $H_i$
at $+\infty$, we can find $\hat q_i >\bar p_i$ such that
\begin{eqnarray*}
  && U_i(\hat q_i)=  G_i(u, \hat q_i, O) - \lambda \frac{M}{2}>0.
\end{eqnarray*}

By continuity of $U_i$, there exists $r_i \in (\bar p_i, \underline q_i)$
(or in $(\bar p_i, \hat q_i)$ in the case $\underline q_i=+\infty$)
such that 
$$
U_i(r_i) = 0 \quad \mbox{for all} \ i,
$$
hence
\begin{eqnarray*}
G_i(u, r_i, O) = U_i(r_i) +  \lambda \frac{M}{2} >0,  \quad \mbox{for all} \ 1\leq i\leq N.
\end{eqnarray*}

Using the continuity of $H_i$, it follows that there exists $\varrho >0$ small enough
such that $r_i - \varrho > \bar p_i$ and
\begin{equation}\label{ineq053}
G_i(u, r_i - \varrho, O) > 0 \quad \mbox{for all} \ i.
\end{equation}

For all $1\leq i\leq N$ and $x\in \bar J_i$, by definition of $\bar p_i$, we have
\begin{eqnarray*}
&& u_i(x)-u(O)\leq \bar p_i x +o(x),  
\end{eqnarray*}
thus
\begin{eqnarray*}
&& u_i(x)-u(O) - ( r_i - \varrho )x \leq (\bar p_i - ( r_i - \varrho )) x +o(x).  
\end{eqnarray*}
Since $\bar p_i - (r_i - \varrho) <0$, we obtain that
the function $x\in \Gamma \mapsto u(x)-\phi(x)$, where $\phi\in C^1(\Gamma)$ is defined by
\begin{eqnarray*}
&& \phi_i(x)= u(O) + (r_i - \varrho) x, \quad x \in \bar{J}_i,
\end{eqnarray*}
has a local maximum at $x=O$.

Then, recalling~\eqref{ineq053}, the Kirchhoff condition at $x=O$ is activated in the viscosity inequality of Lemma~\ref{equiv-def}, from which
\begin{equation}\label{hola}
\sum_{i} -(r_i - \rho) = \sum_{i} -r_i +N \rho \leq B. 
\end{equation}

On the other hand, by~\eqref{comp-nl234} we have $-\infty \leq -\I_i v(O) \leq -\I_i v(O) < +\infty$, hence 
\begin{eqnarray}\label{ineq054}
&& G_i(v, r_i, O)\leq  G_i(u, r_i, O) -\lambda M = U_i(r_i) -\lambda  \frac{M}{2} <0  \quad \text{for all $i$.}
\end{eqnarray}

Since $r_i < \underline{q}_i$, arguing as above, we have that
the function $x\in \Gamma \mapsto v(x)-\psi(x)$, where $\psi\in C^1(\Gamma)$ is defined by
\begin{eqnarray*}
&& \psi_i(x)= v(O) + r_i x, \quad x \in \bar{J}_i,
\end{eqnarray*}
has a local minimum at $x=O$.

From~\eqref{ineq054}, the viscosity inequality of Lemma~\ref{equiv-def} at $x=O$ for supersolution
implies again that the  Kirchhoff condition holds, that is
\begin{equation*}
\sum_i -r_i \geq B,
\end{equation*}
which contradicts~\eqref{hola}. This concludes Case 2 and the proof of the comparison.
\end{proof}

\section{Other classes of boundary conditions.}
\label{secbc}

In this section we briefly discuss some possible extensions of our results for other boundary conditions on the exterior vertices $\{ \ver_i \}_{1 \leq i \leq N}$. 

The key point here is that the arguments around the junction point $O$ can be isolated from the boundary conditions imposed on the exterior vertices $\{ \ver_i \}_{1 \leq i \leq N}$. For instance, Lipschitz regularity of subsolutions  and/or local barriers around at the junction point $O$ can be performed in the same way by localizing the analysis on the set $\Gamma^\delta$ in~\eqref{gamma-d}. Analogously, the local analysis on each of these boundary points $\{ \ver_i \}_{1 \leq i \leq N}$ when other type of boundary conditions are imposed is not substantially modified by the Kirchhoff condition on the junction point $O$. This remark is valid both for the arguments presented in Sections~\ref{sec:exis-junct} and~\eqref{secuniqueness}.

Since a precise presentation of the result would very demanding and would enlarge too much the manuscript, we provide the main ideas of other boundary conditions, referring to the articles where the main ideas to address the problems can be found.

%

\medskip

\noindent
\textsl{1.- Unbounded edges.} It is possible to assume that for some $1\leq i \leq N$, $a_i = + \infty$, and conclude the existence and uniqueness of a bounded solution for the problem, provided the datas are bounded, that is, $\sup_{x \in \bar E_i} |H_i(x, 0)| < +\infty$, following the approach in~\cite[Theorem 3]{bi08}.

\medskip

\noindent
\textsl{2.- Exterior  data.} For a given subset $I\subset \{1,\cdots , N\}$ and $i\in I$, let
$0 < A_i \leq +\infty$ with $a_i < A_i$, $E_i^{out} = \{ x e_i : x \in [a_i, A_i] \}$ where $e_i:=\ver_i/a_i$, and
$g_i \in C_b(E_i^{out})$ (the ``exterior datas'').
Then, we define a new nonlocal operator with the formula~\eqref{def-I}
with an extended domain of integration,
by replacing $J_i$ with 
with  $J_i^* = (0, A_i)$ for $i\in I$. We complemented~\eqref{eq}-\eqref{Kirchhoff} with~\eqref{dirichlet-bc}
for $i\in  \{1,\cdots , N\}\setminus I$ and
the exterior condition 
$$
u(x) = g_i(x), \quad x \in E_i^c, i\in I.
$$
In this case, the arguments in~\cite[Theorem 2.1]{topp14} can be adapted to conclude the existence and uniqueness of a viscosity solution.

\medskip

\noindent
\textsl{3.- State-constraint conditions.} Instead of~\eqref{dirichlet-bc}, we can impose the state-constraint boundary condition
$$
\lambda u( \ver_i) - \I_i u( \ver_i) + H_i( \ver_i, u_{x_i}( \ver_i)) \geq 0,
$$
where the inequality is understood in the viscosity sense:
for each $\phi \in C^1(\Gamma)$ such that $u - \phi$ has a minimum point constrained to $\bar{E_i}$
 at $\ver_i$, then $G_i^\delta (u,\phi, \phi_{x_i}(\ver_i), \ver_i)\geq 0$
(see~\eqref{prolong-derivative-bd} and Definition~\eqref{defi-visco} for the notations). Here, it is possible to follow the arguments discussed in~\cite[Proposition 4.1]{bt16b} to provide a strong comparison principle for problems with this type of boundary conditions.

\medskip

\noindent
\textsl{4.- Neumann boundary conditions.} Instead of~\eqref{dirichlet-bc}, we can impose the Neumann boundary condition
$$
u_{x_i}(\ver_i) = 0,
$$
where the boundary condition is understood in the generalized (viscosity) sense, see~\cite{cil92}. The current setting fits into the assumptions considered by Ghilli~\cite[Theorem 3.1]{ghilli17}, to conclude the well-posedness of the problem.

\section{The problem on a network.}
\label{sec:gene-net}

\subsection{Problem set-up.}\label{secnet0}
In this section we address the problem on a general connected network, that is, in the presence of more than one junction point. Most of the basic definitions and notations were already presented in the introduction, and we complement it here for sake of clarity.


We consider a nonempty, finite set of vertices $\VV \subset \R^d$ and a finite set of edges $E\in \EE$, which are straight lines joining a pair of vertices.
More specifically, for each edge $E \in \EE$, there exists exactly two vertices $\ver_1, \ver_2 \in \VV$
such that $\bar E\cap \VV=\{\ver_1 , \ver_2\}$ and
\begin{eqnarray}\label{defE}
E = \{ \ver_1 + t \frac{\ver_2 - \ver_1}{|\ver_2 - \ver_1|} : t \in J_E \},
\end{eqnarray}
where $J_E = (0, a_E) \subset \R$, $a_E = |\ver_1 - \ver_2|$ is the length of the edge $E$
and $\bar E$ for the closure of the subset $E$ in $\R^d$. We assume that, for any $E_1,E_2 \in \EE$, if $E_1\neq E_2$, then
$E_1\cap E_2=\emptyset$; this means that, if $E_1,E_2$ have a common point, this is necessarely a vertex.

A network $\Gamma \subset \R^d$ is the set 
$$
\Gamma = \bigcup_{E \in \EE} \bar E.
$$


We parametrize the network $\Gamma$ by arc length in the following way (Notice that it is not exclusive).
We label the vertices, namely $\VV = \{ \ver_i \}_{i = 1}^{\mathrm{card}(\VV)}$.
Therefore,
for each $E \in \EE$, there exist a unique pair of indices $i, j$ with $i < j$ such that 
$$
\bar E = \{ \ver_i + t \frac{\ver_j - \ver_i}{a_E} : t \in \bar J_E \},
$$
and we denote $\gamma_E(t) = \ver_i + t \frac{\ver_j - \ver_i}{a_E}$, the parametrization of the (closure) of the edge $E$.

%

%
%


We treat the case of connected networks, that is, for each pair of vertices $\ver_1, \ver_2 \in \VV$,
there exists a finite set of edges $\{ E_i \}_{i=1}^k$ connecting them, i.e., such that $\bar E_1$ contains $\ver_1$, $\bar E_k$ contains $\ver_2$, and $\bar E_i\cap \bar E_{i+1}\neq \emptyset$. Therefore $\alpha := \bar E_1 \cup ... \cup \bar E_k$ is connected with respect to the usual topology in $\R^d$. Such an $\alpha$ is called a path between $\ver_1$ and $\ver_2$. 


Given two points $x, y \in \Gamma$, the geodesic distance $\rho(x,y)$ between $x$ and $y$ is defined by
$$
\rho(x,y) = \inf \{\mathrm{length}(\alpha) : \alpha \ \mbox{is a path between } x \mbox{ and }  y \}.
$$

For each $\ver \in \VV$,
we denote by $\mathrm{Inc}(\ver)$ the set of the {\it incident} edges to $\ver$, that is
the edges $E \in \EE$ such that $\gamma_E(0) = \ver$ or $\gamma_E(a_E) = \ver$.
The vertices are divided according to their nature in terms of its role in the equation.
We write $\VV=\VV_i\cup \VV_b$, where $\ver\in \VV_i$ are like junction points (where we will impose
a Kirchhoff condition), and $\ver\in\VV_b$
are like boundary vertices (where we will impose Dirichlet boundary conditions, or possibly,
another boundary condition, see Section~\ref{secbc}).

\medskip

We define the function spaces as in Section~\ref{sec:fct-space} with straightforward adaptations.
In particular, given $E \in \EE$, for each $x \in \bar E$ we denote $x_E \in \bar J_E$ such that $\gamma_E(x_E) = x$. In this setting, for $u: \Gamma \to \R$, we denote $u_E: \bar J_E \to \R$ as $u_E(t) = u(\gamma_E(t))$, which can be  equivalently written as $u(x) = u_E(x_E)$. When $u \in C^1(\Gamma)$, we write
\begin{equation*}
u_x(x) = u_E'(x_E) \qquad \text{for  $x \in E$, $E \in \EE$,}
\end{equation*}
and we extend the derivative at the vertices by setting, if $E\in \mathrm{Inc}(\ver)$,
\begin{equation*}
u_{x_E}(\ver) = \lim_{x\to \ver, x\in E}u_x(x). 
\end{equation*}
Then we define the inward derivative of $u$ at $\ver$ with respect to $E$ as
\begin{equation}\label{innerdervertex}
\partial_E u(\ver) = \lim_{x \to \ver, x \in E} \frac{u(x) - u(\ver)}{\rho(x, \ver)}. 
\end{equation}
Be careful, $\partial_E u(\ver)$ and $u_{x_E}(\ver)$ do not always coincide
as it was the case for junctions when $\ver =O$. More precisely,
defining the {\em index of incidence} of $\ver$ with respect to $E \in \mathrm{Inc}(\ver)$ as
\begin{equation*}
i_E(\ver) = \left \{ \begin{array}{ll} +1 \quad & \mbox{if} \ \ver = \gamma_E(0), \\
-1 \quad &  \mbox{if} \ \ver = \gamma_E(a_E), \end{array} \right .
\end{equation*}
we have
\begin{equation}\label{indinc}
  u_{x_E}(\ver) = i_E(\ver)\partial_E u(\ver).
\end{equation}


\medskip

Next, we introduce the terms involved in~\eqref{edp-nl} and the assumptions of each one. These conditions, together with the previous definitions and assumptions, are going to be the standing ones throughout this section.

We start with the Hamiltonian $H$. We intend it as a collection $\{ H_E \}_{E \in \EE}$ with $H_E \in C(\bar E \times \R)$,
satisfying~\eqref{H}, which reads here,
\begin{equation}\label{HE}
\left \{ \begin{array}{l}
C_H^{-1} |p| - C_H \leq H_E(x, p) \leq C_H(1 + |p|), \\
|H_E(x, p) - H_E(y, q)| \leq C_H \Big{(}(1 + |p| \wedge |q|)|x  - y| + |p - q| \Big{)},
\end{array} \right .
\end{equation}
for all $x, y \in \bar E, E \in \EE, p, q \in \R.$

\smallskip

For $u\in SC(\Gamma)$, $E \in \EE$ and $x \in \bar E$, we set
\begin{equation}\label{Inet}
\I_E u(x) = \sum_{F \in \EE} \int_{J_{F}} [u_{F}(z)- u(x)] \nu_{E F}(x, \rho(x,\gamma_F(z)))dz.
\end{equation}

The kernels $\nu_{EF}$ satisfy condition~\eqref{hyp-nu}, that is there exist $\Lambda > 0$ and $\sigma \in (0,1)$ such that
\begin{equation}\label{IE}
\left \{ \begin{array}{l} 0 \leq \nu_{EF}(x, t) \leq \Lambda t^{-(1 + \sigma)}, \\
|\nu_{EF}(x,t) - \nu_{EF}(y, t)| \leq \Lambda |x - y| t^{-(1 + \sigma)}, \end{array} \right .
\end{equation}
for each $E, F \in \EE$, $x, y \in \bar E$, and $t > 0$.

We then fix  some $B_v \in \R$ for each  $v \in \VV_i$ and $h_v \in \R$ for each $v \in \VV_b$
to state the Kirchhoff and Dirichlet boundary conditions in~\eqref{edp-nl}.


Before defining the solutions of~\eqref{edp-nl}, we extend
  the notations of the beginning to Section~\ref{sec:def-visco} to the network case.

Given $\varphi\in SC(\Gamma)$,  $E \in \EE$, $A \subseteq \Gamma$ measurable, and $x \in \bar E$, we write
\begin{eqnarray*}
  && \I_E  [A] \varphi(x)  := \sum_{F \in \EE} \int_{\{ z \in F \cap A \}} [\varphi(z) - \varphi(x)]
  \nu_{EF}(x, \rho(x,z)) dz,
\end{eqnarray*}
which is the natural extension of the nonlocal operator defined in\eqref{def-I}-\eqref{defIij}-\eqref{Iiu}. 

Then, for viscosity evaluation, given $u\in SC(\Gamma)$, $\varphi\in  C^1(\Gamma)$, $\delta > 0$, $E \in \EE$, $x \in \bar E$ and $p \in \R$, we write
\begin{align*}
G_{E}^{\delta}(u,\varphi,p,x) := \lambda u(x) - \I_E [B_\delta^c (x)] u(x) - \I_E  [B_\delta (x)] \varphi(x) + H_E(x, p).
\end{align*}



\begin{defi}\label{defi-visco-net}
We say that $u \in USC(\Gamma)$  is a viscosity subsolution to Problem~\eqref{edp-nl} if, for each $x \in \Gamma$, each $\delta > 0$ and each $\varphi \in C^1(\Gamma)$ such that $x$ is a maximum point of $u - \varphi$ in $B_\delta(x)$,  we have
\begin{eqnarray*}
G_E^{\delta}(u,\varphi, \varphi_x(x), x)\leq 0  &\mbox{if} & x \in E, 
\\
\min \Big{\{} \min_{E \in \mathrm{Inc}(x)} G_E^{\delta}(u,\varphi, \varphi_{x_E}(x),x) , \ \sum_{E \in \mathrm{Inc}(x)} -\partial_E \varphi(x)  - B_x \Big{\}} \leq 0
& \mbox{if} & x \in \VV_i,
\\
  \min \Big{\{} \min_{E \in \mathrm{Inc}(x)} G_{E}^{\delta}(u,\varphi, \varphi_{x_E}(x),x), \ u(x)-h_x \Big{\}} \leq 0
& \mbox{if} & x \in \VV_b.
\end{eqnarray*}

We say that $u \in LSC(\Gamma)$  is a viscosity supersolution to Problem~\eqref{edp-nl}
if, for each $x \in \Gamma$, each $\delta > 0$ and each $\varphi \in C^1(\Gamma)$ such that $x$ is a minimum point of $u - \varphi$ in $B_\delta(x)$, we have
\begin{eqnarray*}
G_E^{\delta}(u,\varphi, \varphi_x(x), x)\geq 0  &\mbox{if} & x \in E, 
\\
\max \Big{\{} \max_{E \in \mathrm{Inc}(x)} G_E^{\delta}(u,\varphi, \varphi_{x_E}(x),x) , \ \sum_{E \in \mathrm{Inc}(x)} -\partial_E\varphi(x)  - B_x \Big{\}} \geq 0
& \mbox{if} & x \in \VV_i,
\\
  \max \Big{\{} \max_{E \in \mathrm{Inc}(x)} G_{E}^{\delta}(u,\varphi, \varphi_{x_E}(x),x)  , \ u(x)-h_x \Big{\}} \geq 0
& \mbox{if} & x \in \VV_b.
\end{eqnarray*}

A viscosity solution is a continuous function on $\Gamma$ which is simultaneously a viscosity sub and supersolution in the sense above. 
\end{defi}

\begin{rema}
The definition depends on the parametrization we choose at the beginning,
which is close to what is done in Siconolfi~\cite{siconolfi22}.
Notice that, even in the case of a junction,  the parametrization
is not canonical, see for instance~\cite{im17} and~\cite{ls17}.
\end{rema}





Finally, since we also use a vanishing viscosity approach for the existence in networks, for functions  $u \in C^2(\Gamma)$, we define $u_{xx}(x) := u_E''(x_E)$ for $x \in E, E \in \EE$. For the second-derivative at a vertex $\ver \in \VV$, we denote
$$
u_{x_E x_E}(\ver) = \lim_{x \to \ver, x \in E} u_{xx}(x). 
$$

%

\subsection{Well-possedness on a general network.} We have the following result, analogous to
Proposition~\ref{prop1}.

\begin{prop}\label{prop-existence-net}
Let $\lambda > 0$, $\{ H_E \}_{E \in \EE}$ satisfies~\eqref{HE}, $\{ \I_E\}_{E \in \EE}$ satisfies~\eqref{IE}, and let  $\{ h_{\ver} \}_{\ver \in \VV_b}, \{ \theta_{\ver} \}_{\ver \in \VV_i} \subset \R$ be given.
Then, for each $\epsilon > 0$, there exists a unique viscosity solution $u^\epsilon \in C(\Gamma)$ to the problem
\begin{equation}\label{eq-net-eps}
\lambda u - \epsilon u_{xx} - \I_E u(x) + H_E(x, u_x) = 0 \quad \mbox{on} \ E, \ E \in \EE,
\end{equation}
complemented by Dirichlet conditions 
\begin{equation}\label{bc-net}
u(\ver)= h_\ver \text{ for $\ver\in  \VV_{b}$}, \qquad
u(\ver) =\theta_{\ver} \text{ for $\ver\in  \VV_{i}$}.
\end{equation}

The solution enjoys analogous estimates as the ones presented in Proposition~\ref{prop1}.
\end{prop}

We stress on the fact that the construction of our solution on a junction in Lemma~\ref{lematildeueps}
is based on the well-posedness of the PDE in an interval with Dirichlet condition at the two end points
(see Appendix~\ref{sec:comp-censor}).
It follows that the proof of Proposition~\ref{prop-existence-net} follow
same lines as the proof of  Lemma~\ref{lematildeueps} and
Proposition~\ref{prop1} with minor changes.

The main result of this section is the following

\begin{teo}\label{teo-existence-net}
Assume $\lambda >0, \{ H_E \}_{E \in \EE}$ satisfies~\eqref{HE}, $\{ \I_E \}_{E \in \EE}$ satisfies~\eqref{IE}, $\{ h_{\ver} \}_{\ver \in  \VV_b}, \{ B_{\ver} \}_{\ver \in \VV_i} \subset \R$  are given.

Then, there exists a unique viscosity solution $u \in C(\Gamma)$ for the general Kirchhoff problem~\eqref{edp-nl} on $\Gamma$.
This solution is Lipschitz continuous in $\displaystyle\Gamma\setminus \bigcup_{\ver\in \VV_b} B(\ver ,\delta)$
for all $\delta >0$.
\end{teo}

\begin{proof} We follow here the ideas of Theorem~\ref{teoexistence}, adapted to the general network setting.
The main difference here is to prove the existence of a family $\Theta:=\{\theta_\ver\}_{\ver \in \VV_i}$
for which the Kirchhoff condition
\begin{eqnarray}\label{Kir-net}
 \sum_{E\in \text{Inc}(\ver)} -\partial_E u^{\epsilon,\Theta}(\ver) = B_\ver
\end{eqnarray}
holds pointwisely for the solution $u^{\epsilon,\Theta}$ of~\eqref{eq-net-eps}-\eqref{bc-net}.

We set $\bar B = \max_{\ver \in \VV_i} |B_\ver| + 1$ and define $\psi_0 \in C^2[0,1]$ such that $\psi_0(x) = \psi_0(1 - x)$ for all $x \in [0,1]$, non-increasing in $[0,1/2]$, with $\psi_0(x) = -\bar B x$ if $x \in [0, 1/8]$ and $\psi_0(x) = -3\bar B/16$ if $x \in [1/4,1/2]$. Notice this function is such that $\| \psi_0 \|_{C^2[0,1]} \leq C_0 \bar B$ for some universal constant $C_0 > 0$.

Then, we define the function $\Psi\in C^2(\Gamma)$ by
$$
\Psi_E(x_E) = a_E \psi_0(a_E^{-1} x_E), \quad\text{for $x_E\in \bar J_E$, $E\in \EE$.} 
$$

Notice that $\| \Psi \|_{C^2(\Gamma)} \leq \underline{a}^{-1} C_0 \bar B$, where $\underline{a} = \min_{E \in \EE} |a_E|$,
and, for each $\ver \in \VV$, we have $\Psi(\ver) = 0$.

With similar computations as in the proof of Theorem~\ref{teoexistence},
we obtain that,
for each $E \in \EE$ and $x \in E$,
\begin{eqnarray*}
&& |H_E(x, \Psi_x(x))| \leq C_H (1+C_0 \bar B),\\
&& |\I_E \Psi(x)| \leq \frac{2 \Lambda {\rm card}(\EE)^2 C_0 \bar B  }{1-\sigma},\\
&& |\epsilon \Psi_{xx}(x)| \leq  \underline{a}^{-1} C_0 \bar B, \quad \epsilon\in (0,1),
\end{eqnarray*}
and, for $\theta^+$ large enough in terms of the above estimates, $\bar h = \max_{\ver \in \VV_i} |h_\ver|$
and $\lambda$, the function
$$
\Psi^+(x) = \theta^+ + \Psi(x), \quad x \in \Gamma
$$
is a viscosity supersolution of~\eqref{eq-net-eps}-\eqref{bc-net}, for all $\theta_\ver \in [-\theta^+, \theta^+]$.

Now, we enumerate the set of interior vertices, writing $\VV_i = \{ \ver_1, \ver_2, ..., \ver_{N_i} \}$
(assuming for instance that the interior vertices are labelled first in the starting parametrization)
and $\Theta=  (\theta_1, ..., \theta_{N_i})$, where $N_i={\rm card}(\VV_i)\geq 1$.
For $\Theta\in  [-\theta^+, \theta^+]^{N_i}$, let  $u^{\epsilon,\Theta}\in C^{2,1-\sigma}(\Gamma)$ be the unique
solution of~\eqref{eq-net-eps}-\eqref{bc-net} given by Proposition~\ref{prop-existence-net}, and define
the functions
$$
\Theta\in  [-\theta^+, \theta^+]^{N_i} \mapsto F_j(\Theta) = \sum_{E \in \mathrm{Inc}(\ver_j)} -\partial_E u^{\epsilon,\Theta}(\ver_j) - B_{v_j},
\quad 1\leq j\leq N_i.
$$

For $\Theta\in  [-\theta^+, \theta^+]^{N_i}$, we denote $\Theta^{j,+}$ the vector $\Theta$ where
the coordinate $\theta_j$ is replaced by $\theta^+$.
By comparison (Lemma~\ref{comp-censA} and Remark~\ref{rmk-comparison}), we have
$u^{\epsilon,\Theta^{j,+}}\leq \Psi^+$ on $\Gamma$. Since $\Psi^+(\ver_j)=\theta^+= u^{\epsilon,\Theta^{j,+}}(\ver_j)$, it follows
\begin{eqnarray*}
  &&  \partial_E u^{\epsilon,\Theta^{j,+}}(\ver_j)\leq \partial_E \Psi^+(\ver_j)= \psi_0'(0)=-\bar B,
  \quad \text{for all $E\in  \mathrm{Inc}(\ver_j)$},
\end{eqnarray*}
leading to
$$
F_j(\Theta^{j,+}) = \sum_{E \in \mathrm{Inc}(\ver_j)} -\partial_E u^{\epsilon,\Theta^{j,+}}(\ver_j) - B_{\ver_j}
\geq \mathrm{card}(\mathrm{Inc}(\ver_j)) |\bar B| - B_{\ver_j} > 0,
$$
for all $j$ and all $\Theta$.

Similarly, $\Psi^-(x) = \theta^- - \Psi(x)$ with $\theta^-=-\theta^+$
is a  viscosity subsolution for the same problem, from which we conclude that
$F_j(\Theta^{j,-})<0$ for all $j$ and all $\Theta$.

We are in position to apply Poincaré-Miranda theorem (see~\cite{kulpa97})
to the family $F_j$, $1\leq j\leq N_i$, in $[-\theta^+, \theta^+]^{N_i}$.
It yields the existence of $\Theta^* \in [-\theta^+, \theta^+]^{N_i}$
such that $F_j(\Theta^*) = 0$ for all $j$.

Then, the Kirchhoff condition~\eqref{Kir-net} is satisfied for the solution
$u^{\epsilon,\Theta^*}$ of~\eqref{eq-net-eps}-\eqref{bc-net}, from which we conclude
that $u^{\epsilon,\Theta^*}$ is a solution to~\eqref{edp-nl} with the vanishing
viscosity term. By similar arguments as in the proof of Theorem~\ref{teoexistence},
we can send $\epsilon \to 0$ in the vanishing viscosity problem and obtain
a solution $u\in C(\Gamma)$ to~\eqref{edp-nl}, which is locally Lipschitz continuous in
$\Gamma\setminus\VV_b$.

The proof of uniqueness follows readily the one of Theorem~\ref{teo1}
since, arguing by contradiction, we are brought back to the case when a positive
maximum is achieved at a vertex $\ver\in \VV_i$ and we can argue locally as in
the junction case.
\end{proof}


Collecting the results in the paper, we obtain the following
well-posedness for viscous nonlocal Hamilton-Jacobi
with Kirchhoff conditions on general networks.
This result has its own interest.

\begin{teo}\label{teo-gene}
Assume hypotheses of Theorem~\ref{teo-existence-net} hold, and consider families $\{ f_E \}_{E \in \EE}, \{ \mu_E \}_{E \in \EE}$ such that 
\begin{eqnarray*}
&&
\left \{ \begin{array}{l}
\text{$0\leq \underline{\mu}\leq \mu_E\in C(\bar E)$, \quad $|\sqrt{\mu_E(x)}- \sqrt{\mu_E(y)}|\leq C_\mu|x-y|$,\quad
  $x,y\in \bar{E}$,}\\[1mm]
\text{$f_E\in  C(\bar E)$.}
\end{array} \right .
\end{eqnarray*}

Then, there exists a viscosity solution $u \in C(\Gamma)$ to the problem
\begin{eqnarray}\label{edp-nl-viscous}
&&  \left\{
  \begin{array}{ll}
\lambda u - \mu(x)u_{xx} -\I u(x) + H(x,u_x)=f(x), & x\in \Gamma\setminus \VV,\\
\displaystyle \sum_{E\in \text{Inc}(\ver)} -\partial_E u(\ver) = B_\ver, & \ver\in  \VV_{i},\\
u(\ver)= h_\ver, & \ver\in  \VV_{b},
\end{array}
  \right.
  \end{eqnarray}  
which is Lipschitz continuous in a neighborhood of each vertex $v \in \VV_i$.

Moreover, if the equation is strictly elliptic (i.e., $\underline{\mu}>0$ in \eqref{steady-plus}),
then there exists a unique classical
solution $u\in C^{2,1-\sigma}(\Gamma)$, which satisfies the conditions at the
vertices pointwisely.
\end{teo}

\begin{proof}
The first part of the Theorem is a straightforward adaptation of
Proposition~\ref{prop1} and Theorem~\ref{teoexistence} to the case
of general networks. The only difference is that we argue
on~\eqref{eqaprox0} by replacing the diffusion $-\e u_{x_E x_E}$
by  $-(\e +\mu_E) u_{x_E x_E}$ in the edges $E$, $E\in \EE$.

When the equation is strictly elliptic,
the existence of a $C^{2,1-\sigma}(\Gamma)$ solution
is directly given by Proposition~\ref{prop1}
and Steps 1-2 of the proof of Theorem~\ref{teoexistence},
arguing directly on the PDE in~\eqref{edp-nl-viscous}
without adding $\epsilon$ to the diffusion.

The proof of the uniqueness of the
classical solution is an easy adaptation of~\cite[Lemma 3.6]{adlt19}
or~\cite{cms13}. We provide it to be self-contained.

Consider two solutions $u,v\in C^{2,1-\sigma}(\Gamma)$
of~\eqref{edp-nl-viscous}. 
Assume by contradiction that $\delta:= \max_\Gamma u-v >0$.

If the maximum is achieved inside a branch at $x_0\in \EE \setminus \VV$, then
we perform a classical proof of the maximum principle. More precisely,
the PDE in~\eqref{edp-nl-viscous} holds pointwisely both for
$u$ and $v$ at $x_0$. Using that $(u-v)(x_0)=\delta >0$,
$u_x(x_0)=v_x(x_0)$, $u_{xx}(x_0)\leq v_{xx}(x_0)$ and $\I u(x_0)\leq \I v(x_0)$,
by subtracting the inequality, we reach a contradiction.

It follows that the maximum is necessarily achieved at a vertex.
Since the Dirichlet condition holds pointwisely, $x_0$ cannot be a boundary vertex,
thus $x_0\in \VV_i$.

By maximality, we have $(u-v)(x)\leq (u-v)(x_0)$ for all $x\in\Gamma$.
Hence, $\partial_E u(x_0)\leq \partial_E v(x_0)$ for all $E\in \mathrm{Inc}(x_0)$.
From the Kirchhoff conditions at $x_0$, we get
$\sum_{E\in \text{Inc}(x_0)} (\partial_E u(x_0)-\partial_E u(x_0))  = 0$.
We finally obtain
\begin{eqnarray}\label{egalKFF}
&& \partial_E u(x_0)= \partial_E v(x_0) \quad \text{ for all $E\in \mathrm{Inc}(x_0)$.}
\end{eqnarray}

Let $E\in \text{Inc}(x_0)$ and $x\in E$. Writing the PDEs at $x$ for $u$ and $v$, we infer
\begin{eqnarray*}
  && \mu(x)(u_{xx}(x)-v_{xx}(x))\\
  &=& \lambda (u-v)(x) -\I_E u(x)+\I_E v(x) + H_E(x,u_x(x))-  H_E(x,v_x(x)).
\end{eqnarray*}
Sending $x\to x_0$, we have $\lambda (u-v)(x)\to \lambda \delta >0$,
$-\I_E u(x)+\I_E v(x) \to -\I_E u(x_0)+\I_E v(x_0) \geq 0$ by Lemma~\ref{regI} and maximality
of $x_0$, and 
$$
H_E(x,u_x(x))-  H_E(x,v_x(x))\to 0
$$ 
since
\begin{eqnarray*}
   \lim_{x\to x_0, x\in E} u_x(x) = u_{x_E}(x_0)&=&   i_E( x_0)\partial_E u(x_0)\\
  &=&
  i_E(x_0)\partial_E v(x_0) = v_{x_E}(x_0) = \lim_{x\to x_0, x\in E} v_x(x)
\end{eqnarray*}
by~\eqref{indinc} and~\eqref{egalKFF}. We finally obtain that
$\mu(x_0)(u_{x_E x_E}(x_0)-v_{x_E x_E}(x_0))>0$ with $u_{x_E}(x_0)=v_{x_E}(x_0)$.
Therefore, the function $u-v$ is strictly convex in a neighborhood of $x_0$
in the edge $E$,
with a minimum at $x_0$, which contradicts the fact that $x_0$ is a
maximum of $u-v$. This ends the proof.
\end{proof}

\section{On flux limited solutions.}
\label{sec-FL}

We want to make the link between the solutions of the Kirchhoff 
problem~\eqref{eq}-\eqref{Kirchhoff} and flux limited solutions introduced
in~\cite{im17, ls17, bc24}. 
For the sake of simplicity of notations, we come back to the simpler case of a junction, see
Section~\ref{secdefjunction} for the setting and Section~\ref{sec:def-visco}
for the definition of viscosity solutions.

We start by giving a precise definition
of flux limited solutions in our context.

At first, we need to add an extra assumption on $H$.
In~\cite{im17}, they assume quasiconvexity, 
namely that all $H_i(x,\cdot)$ are nonincreasing on some interval $(-\infty, p_i^0(x)]$
and then nondecreasing on $[p_i^0(x),+\infty)$.
Here, to simplify the presentation, we follow the assumptions
of~\cite{ls17}. Throughout this section, we assume
\begin{eqnarray}\label{Hconvex}
&& \begin{array}{ll}
(i) &\text{$H$ satisfies~\eqref{H}} \\[2mm]
(ii) & \text{For all $1\leq i\leq N$, $x\in \bar{E}_i$, $p\in \R\mapsto H_i(x,p)$ is convex}\\
     &\text{with a unique minimum $p_i^0(x)$.}
\end{array}
\end{eqnarray}
Under this assumption, we can define the nondecreasing and nonincreasing parts of $H_i$, respectively,
by
\begin{eqnarray}
  &&
H_i^-(x,p)=\left\{
 \begin{array}{ll}
H_i(x, p_i^0(x)) & \text{if $p\leq p_i^0(x)$},\\[1mm]
H_i(x,p) &  \text{if $p\geq p_i^0(x)$},
 \end{array}
\right.
\\[2mm]
&& H_i^+(x,p)=\left\{
 \begin{array}{ll}
H_i(x, p) & \text{if $p\leq p_i^0(x)$},\\[1mm]
H_i(x,p_i^0(x)) &  \text{if $p\geq p_i^0(x)$}.
 \end{array}
\right.
\end{eqnarray}
Notice that
\begin{eqnarray}\label{maxHi}
H_i(x,p)=\max \{H_i^-(x,p), H_i^+(x,p)\}.
\end{eqnarray}


Following the notations~\eqref{def-G} and~\eqref{def-G0}
and recalling the definition of $\mathcal{F}_x$ introduced in~\eqref{def-F},
for every $u\in  SC(\Gamma)$, $\varphi\in C^1(\Gamma)$, $\delta\in (0,1)$,
$x\in \bar{E}_i$ and $p\in\R$,
we denote
\begin{eqnarray}\label{def-Gpm}
&& G_{i}^{\pm, \delta}(u,\varphi,p,x):=\lambda u(x) - \I_i  [B_\delta^c (x)] u(x) - \I_i  [B_\delta (x)] \varphi(x) + H^\pm_i(x, p),
\end{eqnarray}
and, for $u\in\mathcal{F}_x$,
\begin{eqnarray}\label{def-G0pm}
  && G_{i}^{\pm}(u,p,x):=\lim_{\delta\downarrow 0}  G_{i}^{\pm, \delta}(u,\varphi,p,x)
  =\lambda u(x) - \I_i u(x)+ H^\pm_i(x, p).
\end{eqnarray}

As for~\eqref{def-G0}, this latter quantity is well-defined in $[-\infty,+\infty]$
and does not depend on the function $\varphi$ chosen for computing the limit.
When we can touch an $USC$ function from above or a $LSC$ function from
below, we have more precise
estimates, see Lemma~\ref{lem-F}. When $u$ is H\"older continuous,
$\I_i u(x)$ is well-defined in $\R$ (see Lemma~\ref{equiv-def}) and so $G_{i}^{\pm}(u,p,x)$.


Next we introduce the so-called flux limiter. Here we follow the approach presented in~\cite[Section 16.4]{bc24}
but, due to the nonlocal nature
of our problem, the definition is more involved.
\begin{defi}[Flux limiter]  \label{defFL}
Let $B \in \R$ be given. For every $u\in  SC(\Gamma)$, $\varphi\in C^1(\Gamma)$ and $\delta\in (0,1)$,
\begin{eqnarray*}
{FL}^{-,\delta}(u,\varphi, O):= \min_{p_1, ..., p_N \in\R} \max \Big{\{} \max_{1\leq i\leq N} G_i^{-,\delta}(u,\varphi, p_i,O), \sum_{i=1}^N - p_i - B \Big{\}}.
\end{eqnarray*}

For  $u \in \mathcal{F}_O$, the {\em flux limiter operator} is defined by
\begin{eqnarray*}
{FL}^-(u,O)=  \min_{p_1, ..., p_N \in\R} \max \Big{\{} \max_{1\leq i\leq N} G_i^{-}(u,p_i,O), \sum_{i=1}^N - p_i - B \Big{\}} \in [-\infty,+\infty].
\end{eqnarray*}
\end{defi}
These definitions make sense. Indeed,
$\sum_{j=1}^N - p_j - B\to -\infty$ whereas,  when they are finite,
$G_i^{-,\delta}(u,\varphi, p_i,O), G_i^{-}(u,p_i,O)\to +\infty$
as $p_i\to +\infty$, thus
the minima are well-defined. See the related~\cite[Lemma 5.3.1]{bc24} for details.

We are now in position to define the  viscosity flux limited (FL) solutions.

\begin{defi}\label{defi-FLsol}
We say that $u \in USC(\Gamma)$  is a viscosity FL subsolution to problem~\eqref{eq}-\eqref{dirichlet-bc}
if, for each function $\varphi \in C^1(\Gamma)$ such that $u - \varphi$ has a local maximum at $x\in\Gamma$, for each $\delta > 0$, we have
\begin{eqnarray*}
G_{i}^{\delta}(u,\varphi, \varphi_{x_i}(x),x)\leq 0  &\mbox{if} & x \in E_i, 
\\
\max \Big{\{} \max_{1\leq i\leq N} G_{i}^{+,\delta}(u,\varphi, \varphi_{x_i}(O),O) , \ {FL}^{-,\delta}(u,\varphi, O) \Big{\}} \leq 0
& \mbox{if} & x=O,
\\
  \min \Big{\{} G_{i}^{\delta}(u,\varphi, \varphi_{x_i}(x),x)  , \ u(x)-h_i \Big{\}} \leq 0
& \mbox{if} & x=\ver_i.  
\end{eqnarray*}

We say that $u \in LSC(\Gamma)$  is a viscosity FL supersolution to problem~\eqref{eq}-\eqref{dirichlet-bc}
if, for each function $\varphi \in C^1(\Gamma)$ such that $u - \varphi$ has a local minimum at $x\in\Gamma$, for each $\delta > 0$, we have
\begin{eqnarray*}
G_{i}^{\delta}(u,\varphi, \varphi_{x_i}(x),x)\geq 0  &\mbox{if} & x \in E_i, 
\\
\max \Big{\{} \max_{1\leq i\leq N} G_{i}^{+,\delta}(u,\varphi, \varphi_{x_i}(O),O) , \  {FL}^{-,\delta}(u,\varphi, O)   \Big{\}} \geq 0
& \mbox{if} & x=O,
\\
  \max \Big{\{} G_{i}^{\delta}(u,\varphi, \varphi_{x_i}(x),x)  , \ u(x)-h_i \Big{\}} \geq 0
& \mbox{if} & x=\ver_i.  
\end{eqnarray*}

A viscosity solution is a continuous function on $\Gamma$ which is simultaneously a viscosity sub and supersolution in the sense above. 
\end{defi}

We give an equivalent definition for the FL solutions. We concentrate in the case of a junction point, the other cases being similar.
\begin{prop}\label{equiv-FL}
An $u \in USC(\Gamma)$  (resp. $LSC(\Gamma)$)
is a viscosity FL subsolution (resp. supersolution) to problem~\eqref{eq}-\eqref{dirichlet-bc}
if and only if the viscosity inequality at $x=O$ is replaced with
\begin{eqnarray}\label{ineg-equiv-FL}
&& \max \Big{\{} \max_{1\leq i\leq N} G_{i}^{+}(u,p_i,O) , \ {FL}^{-}(u, O) \Big{\}} \leq 0 \text{ (resp. $\geq 0$),}
\end{eqnarray}
for every ${\bf p} =(p_1,\cdots, p_N)\in D^+_\Gamma u(O)$ (resp.  $p\in D^-_\Gamma u(O)$).
\end{prop}

\begin{proof}
We focus on the proof for subsolutions. Let ${\bf p} \in D^+_\Gamma u(O)$. There exists $\varphi \in C^1(\Gamma)$ such that $O$ is a local maximum point of $u - \varphi$, and $\varphi_{x_i}(O) = p_i$ for each $i$.

Since $u$ is a subsolution in the sense of Definition~\ref{defi-FLsol}, we have for all $\delta > 0$ small enough that
\begin{align*}
G_i^{+, \delta}(u, \varphi, p_i, O) \leq 0, \quad \mbox{for all} \ i, 
\end{align*}
and
\begin{align}\label{FLsub}
FL^{-, \delta}(u, \varphi, O) \leq 0.
\end{align}

From the first set of inequalities and Lemma~\ref{lem-F}, we infer
$\I_i u(O) = \lim_{\delta \to 0} \I_i[B_\delta(O)^c] u(O) \in \R$ for all $i$, and taking limit as $\delta \to 0$ we arrive at
$$
G_i^+(u, p_i, O) \leq 0, \quad \mbox{for all} \ i.
$$

On the other hand, from~\eqref{FLsub} and the definition of $FL^{-, \delta}(u, \varphi, O)$, we see that for all $\delta > 0$, there exists ${\bf q}^\delta \in \R^N$ such that
\begin{align*}
\max \Big{\{} \max_{1 \leq i \leq N} \{ G_i^{-,\delta}(u, \varphi, q_i^\delta, O), \sum_{i=1}^{N} - q_i^\delta - B \Big{\}} \leq 0.
\end{align*}
Using that $H_i^-(O, p) \to +\infty$ as $p \to +\infty$, and since we already know that $\I_i u(O)$ is finite, we infer that there exists $C > 0$ not depending on $\delta$ such that $q_i^\delta \leq C$ for all $i$ and all $\delta$ small. Using this and the fact that $\sum_{i=1}^{N} - q_i^\delta \leq B$, we conclude that, enlarging $C$ if necessary, we have $|q_i^\delta| \leq C$ for all $i$ and $\delta$ small. Thus, for each $i$, there exists $\tilde q_i \in \R$ such that, taking subsequences, we have $q_i^\delta \to \tilde q_i$ as $\delta \to 0$. Taking limit as $\delta \to 0$, we arrive at $FL^- (u, O) \leq 0$, which concludes this case.

Now we deal with the converse. Assume $u$ is a subsolution in the sense of~\eqref{ineg-equiv-FL}.
Let $\varphi \in C^1(\Gamma)$ be a test function such that $u-\varphi$ has a local maximum at $O$.
Then, ${\bf p} \in \R^N$ such that $p_i = \varphi_{x_i}(O)$, $1 \leq i \leq N$, belongs to $D^+_\Gamma u(O)$. By~\eqref{ineg-equiv-FL}, we have
\begin{equation}\label{FL1}
G_i^+(u, p_i, O) \leq 0  \text{ for $1 \leq i \leq N$} \quad \text{and} \quad FL^- (u, O) \leq 0.
\end{equation}
Since $u$ is touched from above, we have $\I_i u(O) \in [-\infty, +\infty)$ by Lemma~\ref{lem-F}, and thus, the above inequalities imply that $\I_i u(O) \in \R$.

Using that $u - \varphi$ has a local maximum at $O$, we readily have
\begin{align*}
  \I_i u(O) =& \I_i [B_\delta^c(O)]u(O)+   \I_i [B_\delta(O)]u(O)\\
  \leq&  \I_i [B_\delta^c(O)]u(O)+   \I_i [B_\delta(O)]\varphi (O),
\end{align*}
and using these inequalities into~\eqref{FL1} we conclude the viscosity inequality for subsolutions in the sense of Definition~\ref{defi-FLsol}.

For supersolution, the proof is similar to the one above complemented with Lemma~\ref{lem-F} for functions touched from below by $C^1$ functions. The arguments are slightly simpler since we require that only one expression inside the maximum is nonnegative to get the viscosity inequality.
\end{proof}

\subsection{Kirchhoff solutions are flux limiter solutions.}

\begin{teo}
For any $B\in\R$, a viscosity subsolution (resp. supersolution)
to the Kirchhoff problem~\eqref{eq}-\eqref{dirichlet-bc}-\eqref{Kirchhoff}
is a viscosity FL subsolution (resp. supersolution) for the
flux limiter  introduced in Definition~\ref{defFL}.
\end{teo}

\begin{proof}
The only difference in Definitions~\ref{defi-visco} and~\ref{defi-FLsol} is at the junction point~$O$.

\medskip
\noindent
\textsl{1.- Subsolution case.} In this case we can assume that $u$ is Lipschitz continuous at the junction, see Lemma~\ref{sous-sol-lip}.

Let ${\bf p} =(p_1,\cdots , p_N)\in D_\Gamma^+ u(O)$.
We want to prove that
\begin{equation*}
\max \Big{\{} \max_i G_i^+(u, p_i,O), FL^-(u,O) \Big{\}} \leq 0.
\end{equation*}

We start by proving that for all $i$, we have $G_i^+(u, p_i,O) \leq 0$. Since ${\bf p}\in D_\Gamma^+ u(O)$,
there exists $\varphi \in C^1(\Gamma)$ such that $O$ is a strict maximum point
of $u-\varphi$ and $\varphi_{x_i}(O)=p_i$.
Using that $O$ is a strict maximum point of $u - \varphi$, for all $\epsilon \in (0,1)$ small enough, we have that for all $i$, the function 
$$
x_i \mapsto u_i(x_i) - (\varphi_i(x_i) + \epsilon x_i^{-\beta}), \quad x_i \in J_i,
$$
attains its maximum at some point $x_\epsilon \in J_i$, and this point is such that $x_\epsilon \to 0^+$, and $u_i(x_\epsilon) - \varphi_i(x_\epsilon) \to u(O) - \varphi(O)$ as $\epsilon \to 0$. 
Hence, we can use the viscosity inequality for subsolutions on $J_i$, and since $u$ is Lipschitz,
from Lemma~\ref{equiv-def}, we have
\begin{equation*}
G_i(u, p_i - \beta \epsilon x_\epsilon^{-\beta - 1}, \gamma_i(x_\epsilon)) \leq 0.
\end{equation*} 
From~\eqref{maxHi}, it follows
\begin{equation*}
G_i^+(u, p_i - \beta \epsilon x_\epsilon^{-\beta - 1}, \gamma_i(x_\epsilon)) \leq 0,
\end{equation*} 
and since $p \mapsto G_i^+(u, p, x)$ is nonincreasing, we arrive at
\begin{equation*}
G_i^+(u, p_i, \gamma_i(x_\epsilon)) \leq 0.
\end{equation*}

Taking the limit $\epsilon \to 0$, by continuity of the map $x \mapsto G_i^+(u, p, x)$ (see Lemma~\ref{regI}),
we conclude the asserted inequality.

Thus, it remains to prove that $FL^-(u, O) \leq 0$.
For this, a simple generalization of~\cite[Lemma 5.3.1]{bc24}
implies the existence of a $(\tilde p_1, ..., \tilde p_N)\in \R^N$
such that, for all $i$ we have
\begin{eqnarray}\label{bc-lemma}
FL^-(u,O) = G_i^-(u, \tilde p_i, O) = \sum_{i=1}^N - \tilde p_i - B.
\end{eqnarray}

Now, for each $\epsilon \in (0,1)$, consider the function $\psi \in C^1(\Gamma)$ given by
$
\psi_i(x_i) = \tilde p_i x_i + \epsilon^{-1} x_i^2, \ 1 \leq i \leq N.
$
By boundedness of $u$, for all $\epsilon$ small enough we have the existence of a maximum point
$x_\epsilon \in \Gamma$ to $u - \psi$ such that $x_\epsilon \to O$ as $\epsilon \to 0$.

If there is a subsequence $\epsilon\to 0$ such that
$x_\epsilon \in E_i$ for some $i$, then by the viscosity inequality inside the edge we have
\begin{equation*}
G_i(u, \tilde p_i + 2\epsilon^{-1} x_\epsilon, x_\epsilon) \leq 0,
\end{equation*}
which implies, thanks to~\eqref{maxHi},
\begin{equation*}
G_i^-(u, \tilde p_i + 2\epsilon^{-1} x_\epsilon, x_\epsilon) \leq 0,
\end{equation*}
and, since $p \mapsto G_i^-(u, p, x)$ is nondecreasing, we obtain
\begin{equation*}
G_i^-(u, \tilde p_i, x_\epsilon) \leq 0.
\end{equation*}
Sending $\epsilon \to 0$ and using~\eqref{bc-lemma}, we arrive at 
$$
FL^-(u, O) = G_i^-(u, \tilde p_i, O) \leq 0.
$$

If $x_\epsilon = O$ for all $\epsilon$, writing the subsolution inequality at $x_\epsilon = O$
in Lemma~\ref{equiv-def}, we obtain that, either $G_i(u, \tilde p_i, O)\leq 0$ for some $i$,
or $\sum_{j=1}^N -\tilde p_j - B\leq 0$. In both cases, using~\eqref{maxHi} and~\eqref{bc-lemma},
we obtain $FL^-(u, O)\leq 0$,
which concludes the proof of the subsolution case.
\medskip

\noindent
\textsl{2.- Supersolution case.} Here we cannot assume $u$ is Lipschitz continuous.

Thanks to Proposition~\ref{equiv-FL}, it is enough
to prove that, for all ${\bf p} =(p_1,\cdots ,p_N)\in  D^-_{\Gamma} u(O)$, we have
\begin{equation}\label{condition-FL}
\max \{ \max_{1\leq i\leq N} G_i^{+}(u, p_i, O), FL^-(u, O) \} \geq 0,
\end{equation}

If $D_\Gamma^-u(O)=\emptyset$ or $FL^-(u, O)\geq 0$, then we are done.
Therefore, in the rest of the proof, we assume that $D_\Gamma^-u(O)\neq\emptyset$ (and a fortiori $D_{\bar J_i}^+ u_i(0) \neq \emptyset$ for all $i$)
and
\begin{equation}\label{contradiction-FL}
FL^-(u, O) <0.
\end{equation}

From the fact that
the subdifferential is non empty, we can deduce two things.

At first, from Lemma~\ref{lem-F}, $-\I_i u(O)$, hence $G_i(u, p_i, O)$ and $G_i^\pm(u, p_i, O)$,
is in $[-\infty,\infty)$ for all $i$.
We define  the subset $\mathcal{A}$ of ``active indices'' $i$, for which they are finite. 

Secondly,
from~\cite[Proposition 2.5.4]{bc24}, the set $D^-_{\bar J_i} u_i(0)$ is a nonempty interval,
either $D^-_{\bar J_i} u_i(0) = \R$,
or $\underline p_i=\mathop{\rm lim\,inf}_{x\to 0^+}\frac{u_i(x)-u_i(0)}{x}$ is finite and
$D^-_{\bar J_i} u_i(0) = (-\infty, \underline p_i]$.

Now, we divide the proof in two cases showing that
  the first one yields~\eqref{condition-FL} as desired, whereas the second one leads to a contradiction.
\smallskip

\noindent
{\it Case 1. There exists $i$ such that, for all $p \in D^-_{\bar J_i} u_i(0)$, $G_i(u,p,O)\geq 0$.}
In this case, we have that $i \in \mathcal A$.

If $p_i^0(O)\in  D^-_{\bar J_i} u_i(0)$ (see~\eqref{Hconvex}), then $G_i(u,p_i^0(O),O)= G_i^{+}(u,p_i^0(O),O)\geq 0$.
Since $G_i^{+}(u,p_i^0(O),O)=\min_{p\in\R} G_i^{+}(u,p,O)$, we conclude that~\eqref{condition-FL}
holds true.

If $p_i^0(O)\not\in  D^-_{\bar J_i} u_i(0)$, then
$p_i^0(O)>\underline p_i\in  D^-_{\bar J_i} u_i(0)$, thus
$G_i(u,\underline p_i,O)= G_i^{+}(u,\underline p_i,O)\geq 0$.
Using that $H_i^+(O,\cdot)$ is nonincreasing, we obtain
that $G_i^{+}(u,p,O)\geq 0$ for all $p \in D^-_{\bar J_i} u_i(0)$, 
and~\eqref{condition-FL} also holds true. 
The proof is done in this case.
\smallskip

\noindent
{\it Case 2. For all $i$, there exists $p_i\in D^-_{\bar J_i} u_i(O)$ such that
\begin{equation}\label{cond-neg123}
G_i(u,p_i,O)< 0.
\end{equation}
}

We claim that, in this case, all indices $i$ are active.
To prove the claim, let us
fix ${\bf p} =(p_1,\cdots ,p_N)\in D^-_\Gamma u(O)$ such that the inequalities~\eqref{cond-neg123} hold.

At first, if there exists $i_0$ such that
$G_{i_0}(u,p_{i_0},O)=-\infty$, i.e., $i_0\not\in \mathcal{A}$, and
$D^-_{\bar J_{i_0}}u_{i_0}(O)=\R$
then we can increase $p_{i_0}$
in order that we still have $G_{i_0}(u,p_{i_0},O)= -\infty$ and moreover the Kirchhoff
condition is negative, i.e., $\sum_{i=1}^N -p_i - B<0$. Putting this together
with~\eqref{cond-neg123}, we obtain that $u$ cannot be a Kirchhoff subsolution at $O$,
which is a contradiction.

Secondly, if there exists $i_0$ such that
$G_{i_0}(u,p_{i_0},O)=-\infty$ and
$D^-_{\bar J_{i_0}}u_{i_0}(O)= (-\infty, \underline p_{i_0}]$,
then we also reach a contradiction since, by Lemma~\ref{inegGi} we should have $G_{i_0}(u,p_{i_0},O)\geq 0$.

Therefore, all the indices are active.
Taking into account~\eqref{contradiction-FL},
from Definition~\ref{defFL} of the flux limiter
and~\cite[Lemma 5.3.1]{bc24}, we infer the existence of ${\bf \tilde{p}} = (\tilde p_1, \cdots , \tilde{p}_N)\in \R^N$
such that
\begin{align}\label{egal-lem-guy}
FL^-(u, O) = G_i^-(u, \tilde p_i, O) = \sum_{i=1}^N -\tilde p_i - B <0  \quad \text{for all $i$}.
\end{align}

From~\eqref{maxHi} and~\eqref{cond-neg123}, $G_i^+(u, p_i, O) < 0$ for all $i$.
Since  $p_i \leq \underline p_i$ and  $p\mapsto H_i^+(O,p)$ is nonincreasing, 
we get $G_i^+(u, \underline p_i, O) < 0$.
Since $G_i(u, \underline p_i, O) \geq 0$ from
Lemma~\ref{inegGi}, necessarily
$$
G_i^-(u, \underline p_i, O) \geq 0   \quad \text{for all $i$}.
$$
Using that $G_i^-(u, \tilde p_i, O) < 0$ and $p\mapsto H_i^-(O,p)$ is nondecreasing,
we infer that $\underline p_i$ is in the increasing part
of $H_i^-$, and therefore $\tilde p_i < \underline p_i$.
It follows that we can find $p_i^* \in (\tilde p_i, \underline p_i)$, thus $p_i^*\in D^-_{\bar J_i} u_i(0)$,
such that $p_i^*$ is in the increasing part of $H_i^-$ and
\begin{eqnarray}\label{inegGimoins}
G_i(u, p_i^*, O) = G_i^-(u, p_i^*, O) < 0  \quad \text{for all $i$}.
\end{eqnarray}

At the end, we have constructed ${\bf p^*} = (p_1^*,\cdots , p_N^*)\in D^-_\Gamma u(O)$ such that
$G_i^-(u, p_i^*, O) < 0$ for all $i$.
In addition, since $p_i^* > \tilde p_i$,~\eqref{egal-lem-guy}
implies
\begin{eqnarray}\label{Kmoins}
\sum_{i=1}^N -p_i^* - B <  \sum_{i=1}^N -\tilde p_i - B <0.
\end{eqnarray}

Finally,~\eqref{inegGimoins} and\eqref{Kmoins} contradict the fact that $u$ is a Kirchhoff
supersolution at $O$.
This concludes the proof.
\end{proof}

\subsection{Flux limiter solutions are Kirchhoff solutions.}

\begin{teo}
For any $B\in\R$,  a viscosity FL subsolution (resp. supersolution)
is  a viscosity subsolution (resp. supersolution)
to the Kirchhoff problem~\eqref{eq}-\eqref{dirichlet-bc}-\eqref{Kirchhoff}.
\end{teo}

\begin{proof}
We again concentrate on the junction point $O$.

\medskip
\noindent
\textsl{1.- Subsolution case.} Let ${\bf p} =(p_1, \cdots , p_N)\in D_\Gamma^+ u(O)$.
We want to prove that
\begin{equation}\label{ineq-for-K}
\min \{ \min_{1\leq i\leq N} G_i(u, p_i, O), \sum_{i=1}^N - p_i - B \} \leq 0,
\end{equation}
knowing that
$$
\max \{ \max_{1\leq i\leq N} G_i^+(u, p_i, O), FL^-(u, O) \} \leq 0.
$$
In particular, taking into account Lemma~\ref{lem-F}, all the nonlocal terms are finite.

Assume that $\sum_{i=1}^N -p_i > B$. The flux limited condition $FL^-(u, O) \leq 0$
together with~\cite[Lemma 5.3.1]{bc24}
implies the existence of $( \tilde p_1, \cdots , \tilde p_N)$ such that
$$
FL^-(u, O) = G_i^-(u, \tilde p_i, O) = \sum_i -\tilde p_i - B \leq 0 \quad\text{for all $i$.}
$$
In particular, we have $\sum_i -p_i > \sum_i -\tilde p_i$, from which there exists
an index $i_0$ such that $p_{i_0} < \tilde p_{i_0}$.
Since $p\mapsto H_i^-(O,p)$ is nondecreasing, we have that $G_{i_0}^-(u, p_{i_0}, O) \leq 0$.
The inequality for the FL subsolution also provides the inequality $G_{i_0}^+(u, p_{i_0}, O) \leq 0$.
Finally, we obtain $G_{i_0}(u, p_{i_0}, O) \leq 0$, from which we conclude
that~\eqref{ineq-for-K} holds.
\medskip

\noindent
\textsl{2.- Supersolution case.}
Let ${\bf p} =(p_1, \cdots , p_N)\in D^+_\Gamma u(O)$.
We want to prove that
$$
\max \{ \max_{1\leq i\leq N} G_i(u, p_i, O), \sum_{i=1}^N - p_i - B \} \geq 0,
$$
knowing that
$$
\max \{ \max_{1\leq i\leq N} G_i^+(u, p_i, O), FL^-(u, O) \} \geq 0.
$$

The above inequality means that, either there exists an index $i$ such that
$$
G_i^+(u, p_i, O) \geq 0,
$$
or $FL^-(u, O) \geq 0$. In the first case, we get the conclusion since $G_i \geq G_i^+$.
In the second case,~\cite[Lemma 5.3.1]{bc24}
implies the existence of $( \tilde p_1, \cdots , \tilde p_N)$ such that
$$
FL^-(u, O) = G_i^-(u, \tilde p_i, O) = \sum_i -\tilde p_i - B \geq 0 \quad\text{for all $i$.}
$$
If $\sum_{i=1}^N - p_i < B$, then, as before, there exists $i_0$ such that $\tilde p_{i_0} < p_{i_0}$, from which
$$
0 \leq G_{i_0}^-(u, \tilde p_{i_0}, O) \leq G_{i_0}^-(u, p_{i_0}, O) \leq G_{i_0}(u, p_{i_0}, O),
$$
and the result follows.
\end{proof}

\appendix


\section{Proof of Lemma~\ref{regI}.}
\label{AppregIu}

This subsection is completely devoted to the proof of Lemma~\ref{regI}.

Let $x, y \in \bar E_i$. We have that
\begin{align*}
\I_iu(x) - \I_i u(y) = 
\sum_{j=1}^N \I_{ij}u(x) - \I_{ij} u(y). 
\end{align*}

In what follows, we identify $x,y,z \in \bar E_i$ and $x_i,y_i,z_i \in [0, a_i]$. 
Assume $x \leq y$. Then, we see that
\begin{align*}
\I_{ii} u(x) - \I_{ii}u(y) = & -\int_{-y}^{-x} [u_i(x + z) - u_i(x)] \nu_{ii}(x, |z|)dz \\
&  + \int_{-x}^{a_i - y} [u_i(x + z) - u_i(x)] \nu_{ii}(x, |z|)dz \\
& - \int_{-x}^{a_i - y} [u_i(y + z) - u_i(y)] \nu_{ii}(y, |z|)dz \\
& + \int_{a_i - y}^{a_i - x} [u_i(y + z) - u_i(y)] \nu_{ii}(y, |z|)dz \\
=: & \ I_1 + I_2 + I_3 + I_4.
\end{align*}

The terms $I_1$ and $I_4$ can be estimated similarly. 
For the first term, using the H\"older continuity of $u_i$ and
the inequality
\begin{eqnarray*}
&&y^\alpha - x^\alpha \leq (y-x)^\alpha \quad \text{for $0\leq x\leq y$ and $\alpha\in [0,1]$},
\end{eqnarray*}
we have
\begin{align*}
-\int_{-y}^{-x} [u_i(x + z) - u_i(x)] \nu_{ii}(x, |z|)dz \leq & \Lambda [u_i]_{C^{0,\gamma}(\bar J_i)} \int_{-y}^{-x} |z|^{\gamma -1-\sigma} dz \\
\leq & \frac{\Lambda}{\gamma - \sigma} [u]_{C^{0,\gamma}(\Gamma)} |x - y|^{\gamma - \sigma}.
\end{align*}
We obtain the same estimate for $I_4$.

Next, we estimate
\begin{align*}
I_2 + I_3 = &  \int_{-x}^{a_i - y} [u_i(x + z) - u_i(x)] (\nu_{ii}(x, |z|) - \nu_{ii}(y, |z|))dz \\
& + \int_{-x}^{a_i - y} [u_i(x + z) - u_i(x) - (u_i(y + z) - u_i(y))] \nu_{ii}(y, |z|)dz \\
=: & \ I_5 + I_6.
\end{align*}
For $I_5$, using H\"older estimates of $u_i$ and the Lipschitz assumption on $\nu$, we see that
\begin{equation*}
I_5 \leq 2\Lambda |x -y| [ u_i ]_{C^{0,\gamma}(\bar J_i)} \int_{-x}^{a_i - y} |z|^{\gamma - 1 - \sigma} dz \leq C \Lambda \|u\|_{C^\gamma} |x - y|,  
\end{equation*}
where $C > 0$ is a constant just depending on $\sigma, \gamma$ and $a_i$.

For $I_6$, we write $[-x, a_i - y] = A \cup B$ with $A = [-x, a_i - y] \cap [-|x - y|, |x - y|]$ and $B = [-x, a_i - y] \setminus A$. We take profit of the H\"older regularity of $u$, that is,
\begin{equation}\label{ui-hold123}
\begin{split}
& u_i(x + z) - u_i(x) - (u_i(y + z) - u_i(y)) \leq 2 [u]_{C^{0,\gamma}(\Gamma)} |z|^\gamma \quad \mbox{for} \ z \in A, \\
& u_i(x + z) - u_i(x) - (u_i(y + z) - u_i(y)) \leq 2 [u]_{C^{0,\gamma}(\Gamma)} |x - y|^\gamma \quad \mbox{for} \ z \in B.
\end{split}
\end{equation}
to obtain, thanks to the assumptions on $\nu$,
\begin{align*}
I_6 \leq & 2 \Lambda [u]_{C^{0,\gamma}(\Gamma)} \Big{(} \int_A |z|^{\gamma - 1 - \sigma} dz + |x - y|^\gamma \int_B |z|^{-1-\sigma} dz \Big{)} \\
\leq & C \Lambda [u]_{C^{0,\gamma}(\Gamma)} |x - y|^{\gamma - \sigma},
\end{align*}
for some $C > 0$ just depending on $\sigma, \gamma$ and $a_i$.

Since the argument are symmetric in $x$ and $y$, we conclude that
$$
|\I_{ii}u(x) - \I_{ii}u(y)| \leq C \Lambda [u]_{C^{0,\gamma}(\Gamma)} |x - y|^{\gamma - \sigma}.
$$

For $j \neq i$, since for 
$x,y\in E_i$ and 
$z \in E_j$, we have $\rho(x,z) = x_i + z_j$, $\rho(y,z) = y_i + z_j$, by definition we obtain
\begin{eqnarray}
  \label{reecrit234}
\quad\I_{ij} u(x) - \I_{ij} u(y) & =  & \int_{x_i}^{a_j + x_i} [u_j(z_j- x_i) - u_i(x_i)] \nu_{ij}(x_i, z_j)dz_j \\ \nonumber
&& - \int_{y_i}^{a_j + y_i} [u_j(z_j-y_i) - u_i(y_i)] \nu_{ij}(y_i, z_j)dz_j,
\end{eqnarray}
and from here we follow the same estimates as above, noticing that now, thanks to~\eqref{hold-ineg} and~\eqref{hold-gam}, the estimate~\eqref{ui-hold123}
reads
\begin{eqnarray*}
&& u_i(z_j-x_i) - u_i(x_i) - (u_i(z_j-y_i) - u_i(y_i)) \leq 2^{2-\gamma} [u]_{C^{0,\gamma}(\Gamma)} |z_j|^\gamma \quad \mbox{for} \ z \in A.
\end{eqnarray*}
\qed

\section{Well-posedness for censored problems}
\label{sec:comp-censor}

We prove here the comparison principle for the Dirichlet problem
\begin{equation}\label{dirichlet-gene}
  \left \{ \begin{array}{l} \lambda u - \mu_i(x) u_{x_i x_i} - \I_i u + H_i(x, u_{x_i}) = f_i(x)
    \quad \mbox{on} \ E_i, 1\leq i\leq N, \\
u(O) = \theta, \\
u_i(a_i) = h_i, \quad 1\leq i\leq N,
  \end{array} \right .
\end{equation}
on a junction $\Gamma$ when the sub and supersolution are ordered at the vertices.

The PDE~\eqref{dirichlet-gene} is a little more general than~\eqref{eqaprox0}.
We will assume that the datas satisfy~\eqref{steady} and, in addition, for all $1\leq i\leq N$,
\begin{eqnarray}\label{steady-plus}
&&
\left \{ \begin{array}{l}
\text{$0\leq \underline{\mu}\leq \mu_i\in C(\bar J_i)$, \ $|\sqrt{\mu_i(x)}- \sqrt{\mu_i(y)}|\leq C_\mu|x-y|$,\quad
  $x,y\in \bar{J}_i$,}\\[1mm]
\text{$f_i\in  C(\bar J_i)$.}
\end{array} \right .
\end{eqnarray}

\begin{lema}\label{comp-censA}
Assume $\Gamma$ is a junction with $N \geq 1$ edges.
Assume that~\eqref{steady} and~\eqref{steady-plus} hold and $\theta\in\R$. Let
$u \in USC(\Gamma)$ be a viscosity subsolution and $v \in LSC(\Gamma)$ be a viscosity supersolution
of~\eqref{dirichlet-gene}
such that $u \leq v$ on $\VV = \{ \ver_i \}_{1 \leq i \leq N} \cup \{ O \}$. Then $u \leq v$ on $\Gamma$.
\end{lema}

\begin{rema}\label{rmk-comparison}
For the sake of notations, we present the proof in the case of a junction.
But it applies readily to the case of a general network (up to use the tedious
notations of Section~\ref{sec:gene-net}). Indeed, since $u\leq v$ at the vertices,
a positive maximum of $u-v$ cannot be achieved at the vertices and we can localize the following proof
by contradiction inside one particular edge. Moreover, all the nonlocal estimates are done
edge by edge, and do not depend on the special structure of the network.
\end{rema}

\begin{proof}
By contradiction, we assume 
$$
M := \max_\Gamma \{ u - v\} > 0.
$$

Then, doubling variables, by standard arguments in the viscosity solution's theory, we have the function
$$
\Phi(x, y)= u(x) - v(y) - \alpha^{-2} \rho(x,y)^2, \quad  x,y\in \Gamma,
$$
attains its maximum at some point $(\bar x, \bar y)$ (which depends upon $\alpha$) such that
\begin{align}\label{x-ysmall}
\alpha^{-2} \rho(\bar x, \bar y)^2 \to 0, \ u(\bar x) - v(\bar y) \to M, \quad \mbox{as} \ \alpha \to 0.
\end{align}

In particular, up to subsequences (not relabeled), we have $\bar x, \bar y \to \hat x \in \Gamma$ as $\alpha \to 0$. Since $(x,y) \mapsto u(x)-v(y)$ is upper semi continuous on $\Gamma \times \Gamma$ and $u-v \leq 0$ on the vertices $\VV$, there exists $\delta_0 > 0$ such that $\rho (\bar x, \VV), \rho (\bar y, \VV)\geq  \delta_0$ for all $\alpha$ small enough. Thus, without loss of generality we may assume that $\bar x, \bar y \in E_i$ for some fixed index $i$, for all $\alpha$ small.

Now, denoting $\phi(x,y) = \alpha^{-2} |x - y|^2$ and $\bar p = D_x \phi(\bar x, \bar y) = - D_y \phi(\bar x, \bar y)$, we use~\cite[Corollary 1]{bi08} to infer the existence of $\bar \varrho > 0$ such that, for all $\varrho < \bar \varrho$, there exists $X_\varrho, Y_\varrho \in \R$ such that, for all $\delta > 0$ small enough, we have the viscosity inequalities
\begin{eqnarray*}
  && \lambda u(\bar x) - \mu_i(\bar x) X_\varrho
  - \I_i [B_\delta(\bar x)] \phi(\cdot, \bar y) (\bar x) - \I_i[B_\delta^c(\bar x)] u(\bar x) + H_i(\bar x, \bar p)\\
  && \hspace*{8cm}\leq f_i(\bar x) +  \alpha^{-2} o_\varrho(1), \\
  && \lambda v(\bar y) - \mu_i(\bar y) Y_\varrho + \I_i [B_\delta(\bar y)] \phi(\bar x, \cdot) (\bar y)
  - \I_i[B_\delta^c(\bar y)] v(\bar y) + H_i(\bar y, \bar p)\\
  &&  \hspace*{8cm}\geq  f_i(\bar y) + \alpha^{-2}o_\varrho(1),
\end{eqnarray*}
where $X_\varrho, Y_\varrho$ are such that
$$
\left [ \begin{array}{cc} X_\varrho & 0 \\ 0 & -Y_\varrho \end{array} \right ] \leq D^2_{(x,y)} \phi(\bar x, \bar y) + o_\varrho(1),
$$
and where $o_\varrho(1) \to 0$ as $\varrho \to 0$ uniformly with respect to the other parameters. Taking into account~\eqref{steady-plus} for $a_i$ and
following the classical arguments in~\cite{cil92} on the matrix inequality above, we see that
\begin{equation*}
    \mu_i(\bar x) X_\varrho -  \mu_i(\bar y) Y_\varrho
    \leq 3C_\mu^2 \alpha^{-2}\rho (\bar x,\bar y)  + o_\varrho(1) \leq
    o_\alpha(1) + o_\varrho(1)
\end{equation*}
for all $\varrho$. 

On the other hand, by the assumptions~\eqref{H},~\eqref{steady-plus}
on $H_i$, $f_i$, and the asymptotic properties of the maximum point $(\bar x, \bar y)$, we have
$$
-H_i(\bar x, \bar p) + H_i(\bar y, \bar p) + f_i(\bar x)-f_i(\bar y) \leq  o_\alpha(1).
$$

To estimate the nonlocal terms, we first
assume $0 < \delta < \delta_0/2$, from which $B_\delta(\bar x), B_\delta(\bar y) \subset E_i$ for all $\alpha$ small.
Using the smoothness of $\phi$ and~\eqref{hyp-nu}, We have
\begin{equation*}
|\I_i[B_\delta]\phi(\cdot, \bar y)(\bar x)|, |\I_i[B_\delta]\phi(\bar x, \cdot)(\bar y)| \leq C \Lambda \alpha^{-2} \delta^{1 - \sigma},
\end{equation*}
for some $C > 0$ just depending on $\sigma$.

Subtracting the viscosity inequalities and using the above estimates, we arrive at
\begin{equation}\label{app-ineq}
\lambda M - \alpha^{-2} O(\delta^{1 - \sigma}) - o_\alpha(1) - \alpha^{-2} o_\varrho(1) \leq I, 
\end{equation}
where we have denoted
$$
I:=\I_i[B_\delta^c(\bar x)] u(\bar x) - \I_i[B_\delta^c(\bar y)] v(\bar y). 
$$

It remains to estimate this term. If $\bar x = \bar y$ along a subsequence $\alpha\to 0$, we have
\begin{align*}
  I = & \sum_{j\neq i} \int_{J_j} [u_j(z) - v_j(z) - (u_i(\bar x) - v_i(\bar x))] \nu_{ij}(\bar x, \rho(\bar x,z)) dz \\
  & + \int_{J_i \setminus B_\delta(\bar x)} [u_i(z) - v_i(z) - (u_i(\bar x) - v_i(\bar x))] \nu_{ii}(\bar x, |\bar x - z|)dz.
\end{align*}
Since $\Phi(\bar x, \bar x) \geq \Phi(z,z)$ for all $z \in \Gamma$, we obtain that $I \leq 0$. Replacing it into~\eqref{app-ineq} and sending $\varrho \to 0, \delta \to 0$ and $\alpha \to 0$, we arrive at a contradiction with the fact that $M > 0$.

In the case $\bar x \neq \bar y$, we divide the analysis depending on which edge we integrate.

For $j \neq i$, following~\eqref{reecrit234}, we write
\begin{align*}
\I_{ij}u(\bar x) - \I_{ij} v(\bar y) = & \int_{\bar x}^{a_j + \bar x} [u_j(z - \bar x) - u_i(\bar x)] \nu_{ij}(\bar x, z)dz \\
& - \int_{\bar y}^{a_j + \bar y} [v_j(z - \bar y) - v_i(\bar x)] \nu_{ij}(\bar y, z) dz.
\end{align*}
Without loss of generality, we can assume $\bar x < \bar y$. As we already mentioned, there exists $\delta_0 > 0$ such that $\delta_0 < \bar x$ for all $\alpha$. Hence, by the boundedness of $u, v$, we have
\begin{equation}\label{traslape}
\begin{split}
& \Big{|} \int_{\bar x}^{\bar y} [u_j(z - \bar x) - u_i(\bar x)] \nu_{ij}(\bar x, z)dz \Big{|} \leq C \Lambda \| u \|_\infty \delta_0^{-\sigma} |\bar x - \bar y|= o_\alpha(1), \\
& \Big{|} \int_{a_j + \bar x}^{a_j + \bar y} [v_j(z - \bar y) - v_i(\bar y)] \nu_{ij}(\bar y, z)dz \Big{|} \leq C \Lambda \| v \|_\infty \delta_0^{-\sigma} |\bar x - \bar y| =o_\alpha(1),
\end{split}
\end{equation}
since $|\bar x - \bar y| \to 0$ as $\alpha \to 0$. It follows
\begin{align*}
\I_{ij}u(\bar x) - \I_{ij} v(\bar y) = &  \int_{\bar y}^{a_j + \bar x} [u_j(z - \bar x) - v_j(z - \bar y) - (u_i(\bar x) - v_i(\bar y))] \nu_{ij}(\bar x, z)dz \\
& +  \int_{\bar y}^{a_j + \bar x} [v_j(z - \bar y) - v_i(\bar y)] (\nu_{ij}(\bar x, z) - \nu_{ij}(\bar y, z))dz + o_\alpha(1).
\end{align*}
Using that $(\bar x, \bar y)$ is maximum point of $\Phi$, we see that
\begin{align*}
\I_{ij}u(\bar x) - \I_{ij} v(\bar y) \leq \int_{\bar y}^{a_j + \bar x} [v_j(z - \bar y) - v_i(\bar y)] (\nu_{ij}(\bar x, z) - \nu_{ij}(\bar y, z))dz + o_\alpha(1),
\end{align*}
and by the continuity assumption on the kernels and the fact that we are integrating away the origin, we conclude that
\begin{align*}
\I_{ij}u(\bar x) - \I_{ij} v(\bar y) \leq C \Lambda \| v \|_\infty |\bar x - \bar y|  + o_\alpha(1) = o_\alpha(1).
\end{align*}

In order to estimate the same term for $j=i$, we need to be careful since we require an estimate independent of $\delta$.
Recalling $\delta <\delta_0/2$, we can write
\begin{eqnarray*}
  && \I_{ii}[B_\delta^c(\bar x)] u(\bar x) - \I_{ii} [B_\delta^c(\bar y)] v(\bar y)\\
  &=& \left(\int_{-\bar x}^{-\delta} + \int_{\delta}^{a_i -\bar x}\right) (u_i(\bar x + z) - u_i(\bar x)) \nu_{ii}(\bar x, z)dz\\
  && -  \left(\int_{-\bar y}^{-\delta} + \int_{\delta}^{a_i -\bar y}\right) (v_i(\bar y + z) - v_i(\bar x)) \nu_{ii}(\bar y, z)dz.
\end{eqnarray*}
We then perform the same measure decomposition as in~\cite{bci11}. For $t > 0$, we consider the kernels
\begin{align*}
\tilde \nu_{ii}(t) = & \min \{ \nu_{ii}(\bar x, t), \nu_{ii}(\bar y, t) \}, \\
\nu_{ii}^+(t) = & \nu_{ii}(\bar x, t) - \tilde \nu_{ii}(t), \\
\nu_{ii}^-(t) = & \tilde \nu_{ii}(t)- \nu_{ii}(\bar y, t),
\end{align*}
from which, assuming without loss of generality $\bar x <\bar y$,  denoting $A = (-\bar x, -\delta) \cup (\delta, a_i - \bar y)$, 
and using similar estimates as in~\eqref{traslape}, we have
\begin{align*}
& \I_{ii}[B_\delta^c(\bar x)] u(\bar x) - \I_{ii} [B_\delta^c(\bar y)] v(\bar y) \\
= & o_\alpha(1) + \int_A [u_i(\bar x + z) - v_i(\bar y + z) - (u_i(\bar x) - v_i(\bar y))] \tilde \nu_{ii}(z)dz \\
& + \int_A [u_i(\bar x + z) - u_i(\bar x)] \nu^+_{ii}(z)dz - \int_A [v_i(\bar y + z) - v_i(\bar y)] \nu^-_{ii}(z)dz.
\end{align*}
Using that $\tilde \nu_{ii}$ is nonnegative and the fact that $(\bar x, \bar y)$ is maximum point of $\Phi$, we get
\begin{align*}
& \I_{ii}[B_\delta^c(\bar x)] u(\bar x) - \I_{ii} [B_\delta^c(\bar y)] v(\bar y) \\
\leq & o_\alpha(1) + \int_A [u_i(\bar x + z) - u_i(\bar x)] \nu^+_{ii}(z)dz - \int_A [v_i(\bar y + z) - v_i(\bar y)] \nu^-_{ii}(z)dz \\
=: &  o_\alpha(1) + I_1 + I_2.
\end{align*}

The estimates of $I_1$ and $I_2$ are similar, thus we concentrate on $I_1$. Taking $0 < \delta < \min \{ |\bar x - \bar y|, \delta_0/2 \}$, we use the fact that, by maximality of $(\bar x, \bar y)$ the following inequality holds
$$
u_i(\bar x + z) - u_i(\bar x) \leq \alpha^{-2} (2|\bar x - \bar y||z| + |z|^2), \quad z \in A \subset (-\bar x, a_i - \bar x).
$$
Denote $A_{1} = \{ z \in A : |z| \leq |\bar x- \bar y|\}$ and $A_{2} = A \setminus A_{1}$. Since $\nu_{ii}^+$ is nonnegative, we can write
\begin{align*}
I_1 \leq & \alpha^{-2} \int_{A_{1}} [2|\bar x - \bar y||z| + |z|^2] \nu_{ii}^+(z)dz + 2\| u \|_\infty \int_{A_{2}} \nu_{ii}^+(z)dz.
\end{align*}
At this point, we notice that by the assumptions~\eqref{hyp-nu} on $\nu$ we have
\begin{equation*}
\nu_{ii}^+(z) \leq \Lambda |\bar x - \bar y| |z|^{-(1 + \sigma)},
\end{equation*}
from which we conclude that
\begin{equation*}
I_1 \leq \Lambda \alpha^{-2} |\bar x - \bar y|^{3 - \sigma} (2(1 - \sigma)^{-1}  + (2 - \sigma)^{-1}) + 4 \Lambda \sigma^{-1} \| u \|_\infty |\bar x - \bar y|^{1 - \sigma}.
\end{equation*}
Thus, in view of~\eqref{x-ysmall}, we obtain $I_1 \leq o_\alpha(1)$, and similarly,  $I_2 \leq o_\alpha(1)$.

Putting these estimates into~\eqref{app-ineq} yields
\begin{equation*}
\lambda M - \alpha^{-2} \delta^{1 - \sigma} - o_\alpha(1) - \alpha^{-2} o_\varrho(1) \leq 0.
\end{equation*}
Letting $\varrho \to 0, \delta \to 0$ and finally $\alpha \to 0$, we reach a contradiction. This concludes the result.
\end{proof}

\begin{teo}\label{teo-censA}
Assume $\Gamma$ is a junction with $N \geq 1$ edges.
Assume that~\eqref{steady} and~\eqref{steady-plus} hold with $\underline{\mu}>0$,
and $\theta\in\R$.
Then, there exists a unique continuous viscosity solution $u\in C(\Gamma)$ to the problem~\eqref{dirichlet-gene}
satisfying the Dirichlet boundary conditions pointwisely.
\end{teo}

\begin{proof}
From Lemma~\ref{sous-sur-sol}, there exist continuous sub- and supersolution $\psi^-$, $\psi^+$ to~\eqref{dirichlet-gene}
satisfying the boundary condition pointwisely
(the adaptation of Lemma~\ref{sous-sur-sol} to~\eqref{dirichlet-gene} when the equation is
  strictly elliptic is straightforward).
By Perron's method, we infer the existence
of a (possibly discontinuous) viscosity solution $u$ to~\eqref{dirichlet-gene} satisfying $\psi^-\leq u\leq \psi^+$
on $\Gamma$. Applying Lemma~\ref{comp-censA} with the subsolution $u^*$ and the supersolution $u_*$
satisfying $u^*=u_*$ on $\VV$, we conclude that $u^*\leq u_*$. Thus $u$ is a continuous
viscosity solution to~\eqref{dirichlet-gene}, which is unique, thanks again to Lemma~\ref{comp-censA}.
\end{proof}



\end{document}